\documentclass[letterpaper,10pt]{article}
\usepackage{amsfonts,amssymb,amsthm,eucal,amsmath}
\usepackage{amscd}
\usepackage{graphicx}
\usepackage[T1]{fontenc}
\usepackage{latexsym,url}
\usepackage{hyperref}
\usepackage[inner=30mm, outer=30mm, textheight=225mm]{geometry}

\newtheorem{thm}{Theorem}[section]
\newtheorem{lemma}[thm]{Lemma}

\newtheorem{cor}[thm]{Corollary}
\theoremstyle{definition}
\newtheorem{defn}[thm]{Definition}

\newtheorem{cond}[thm]{Condition}
\newtheorem{ex}[thm]{Example}
\newtheorem{algor}[thm]{Algorithm}
\newtheorem{rmk}[thm]{Remark}
\newcommand{\N}{\mathbb{N}}
\newcommand{\Z}{\mathbb{Z}}

\newcommand{\C}{\mathbb{C}}
\newcommand{\CP}{\mathbb{CP}}
\newcommand{\PSL}{\mathrm{PSL}(2,\mathbb{C})}
\newcommand{\Isom}{\text{Isom}}
\newcommand{\Hthree}{\mathbb{H}^3}
\newcommand{\til}[1]{\widetilde{#1}}

\newcommand{\bdry}{\partial}
\newcommand{\ord}{\text{order}}
\newcommand{\tri}{\bigtriangleup}

\usepackage{default}

\title{A generalisation of the deformation variety}
\author{Henry Segerman}

\begin{document}

\maketitle

\begin{abstract}
Given an ideal triangulation of a connected 3-manifold with non-empty boundary consisting of a disjoint union of tori, a point of the deformation variety is an assignment of complex numbers to the dihedral angles of the tetrahedra subject to Thurston's gluing equations. From this, one can recover a representation of the fundamental group of the manifold into the isometries of 3-dimensional hyperbolic space. However, the deformation variety depends crucially on the triangulation: there may be entire components of the representation variety which can be obtained from the deformation variety with one triangulation but not another. We introduce a generalisation of the deformation variety, which again consists of assignments of complex variables to certain dihedral angles subject to polynomial equations, but together with some extra combinatorial data concerning degenerate tetrahedra. This ``extended deformation variety'' deals with many situations that the deformation variety cannot. In particular we show that for any ideal triangulation of a small orientable 3-manifold with a single torus boundary component, we can recover all of the irreducible non-dihedral representations from the associated extended deformation variety. More generally, we give an algorithm to produce a triangulation of a given orientable 3-manifold with torus boundary components for which the same result holds. As an application, we show that this extended deformation variety detects all factors of the $\PSL$ A-polynomial associated to the components consisting of the representations it recovers. 
\end{abstract}

\section{Introduction}

In Thurston's ground-breaking notes \cite{thurston}, explicit computations of hyperbolic structures on cusped 3-manifolds are given. The strategy is as follows: First we decompose the 3-manifold $M$ (which we will assume to be orientable throughout this paper) into ideal tetrahedra\footnote{An ideal tetrahedron is a tetrahedron missing its vertices, which are to be thought of as being out at the cusp(s) of the manifold.} in some way, giving an ideal triangulation of the manifold. Next, we try to give each tetrahedron the geometry of an ideal hyperbolic tetrahedron\footnote{An ideal hyperbolic tetrahedron (with distinct ordered vertices) is the intersection of $\Hthree$ with the convex hull of 4 point on the sphere at infinity of $\Hthree$, which we write as $\bdry\Hthree$.}, embedded in $\Hthree$. Thurston considers a system of polynomial equations, the \textbf{gluing equations}, whose (complex) variables describe the shapes of the ideal hyperbolic tetrahedra. The equations are satisfied if and only if the ideal hyperbolic tetrahedra fit together properly around the edges, and a solution of these determines a representation of $\pi_1M$ into $\Isom(\Hthree)$. He also defined additional polynomial equations, the \textbf{completeness equations}, and showed that a solution to the gluing and completeness equations gives a discrete and faithful representation of $\pi_1M$ (as long as all tetrahedron shapes specified by the solution are positively oriented), and so a complete hyperbolic structure on $M$.  Solutions of the gluing equations near to the complete structure give incomplete hyperbolic structures, which we can view as deformations of the complete structure. The ``deformation variety'' is the affine algebraic set determined by the gluing equations, and therefore contains this family of deformations of the complete structure. Note that the deformation variety depends on the ideal triangulation in an intrinsic way, so is not an invariant of $M$.\\

The representations of $\pi_1M$ in the above construction arise as \textbf{holonomy representations} of \textbf{developing maps}. The developing map is a map $\Phi$ from the universal cover $\til{M}$ of $M$ to $\Hthree$, which can be thought of as being constructed by building a copy of $\til{M}$ in $\Hthree$ out of (possibly overlapping) ideal hyperbolic tetrahedra of the appropriate shapes. The holonomy representation of a developing map is the unique representation $\rho:\pi_1M\to\Isom(\Hthree)$ such that $\Phi$ is equivariant with respect to $\rho$. A point of the deformation variety determines a developing map up to conjugacy by elements of $\Isom(\Hthree)$, and so the holonomy representation is also only defined up to conjugacy. Thus we can also think of it as encoding essentially the same information as a point of the $\PSL$ character variety.\\

If all of the tetrahedron shapes are strictly positively oriented, then the developing map is a local homeomorphism, and we get a hyperbolic structure on the manifold. In general this isn't true if there are flat or negatively oriented tetrahedron shapes. However, the algebraic (as opposed to geometric) interpretation as a representation (up to conjugacy) into $\Isom(\Hthree)\cong\PSL$ holds, and we can think of the map from the deformation variety to the $\PSL$ character variety of $M$ as a parameterisation of the character variety. For computational purposes, this is by far the most effective way of describing the character variety known (as implemented in SnapPea and SnapPy \cite{snappea, snappy}, and also recent work by Culler on calculating the A-polynomial, see Section \ref{8 20}). However, the character variety is a canonical object, while the deformation variety is not, depending on the choice of triangulation. Many applications are concerned only with the complete hyperbolic structure of a manifold with torus boundary components, and for these uses the dependence is a minor disadvantage. If an ideal triangulation of a cusped hyperbolic manifold has all edges \textbf{essential} (an edge is essential if it cannot be homotoped into the boundary of the manifold) then the associated deformation variety has a component which maps to the Dehn surgery component of the character variety, which contains the complete structure. See~\cite{tillmann_degenerations}.\\

However, for applications that involve structures corresponding to components of the character variety other than the Dehn surgery component (for example, in finding ideal points in order to detect incompressible surfaces, as in \cite{yoshida91, tillmann_degenerations, segerman_torus_bundles}, or calculating the A-polynomial), this dependence becomes troublesome. Depending on the triangulation, entire components may be missing from the deformation variety, even if the triangulation is minimal. The problem arises when the shape of an ideal hyperbolic tetrahedron would be degenerate, i.e. that the positions of its vertices on $\bdry \Hthree$ are not distinct. This means that the (supposed) complex number associated to the shape of the tetrahedron would be 0, 1 or $\infty$. Thus these degenerate shapes do not appear as solutions to the gluing equations. In Sections \ref{LLR_ex_part_1} and \ref{8 20} we give examples of this behaviour. Moreover, there doesn't currently seem to be a general method for finding a ``good'' ideal triangulation for which the associated deformation variety does not miss components, and it isn't even known if such a triangulation exists in general. \\

The dependence on the triangulation is the issue we tackle in this paper. We introduce a generalisation of the deformation variety\footnote{For clarity, henceforth we refer to the deformation variety as the ``standard deformation variety''.} for orientable manifolds, which to a large extent solves the problem of dependence on the triangulation, whilst retaining many useful features of the standard deformation variety (in particular, computation is not significantly more complicated in many cases). Our generalisation, the ``extended deformation variety'', is another affine complex variety. It still uses a triangulation as part of its data, and it includes the standard deformation variety as a subset. The points of the extended deformation variety also give conjugacy classes of representations of $\pi_1M$ into $\PSL$, and the map from the extended deformation variety onto the character variety may strictly reduce dimension. However, as long as the triangulation satisfies certain mild conditions, this map is guaranteed to be surjective (onto irreducible non-dihedral elements of the representation variety, up to conjugation).\\

Notation for the following result: Roughly, a \textbf{horo-normal surface} is a normal surface that intersects each edge of the triangulation either zero or two times, and so cuts the manifold into an inside and outside region, the latter containing the cusp(s)\footnote{The lift of the surface in the universal cover acts similarly to a horosphere in $\Hthree$, hence the name. }. A horo-normal surface is \textbf{porous} if the preimage (under the covering map) of the inside region in the universal cover of the manifold is connected, and for every cusp, there is some tetrahedron incident to that cusp that the inside region intersects. See Definitions \ref{E_0_to_horo-normal} and \ref{porous} for details. $\widehat{\mathfrak{D}}(M;\mathcal{T})$ is the extended deformation variety for $M$ with triangulation $\mathcal{T}$, $\mathfrak{R}(M)$ is the $\PSL$ representation variety for $M$, and $\mathfrak{R}_{\mathcal{T}}$ is a canonical (up to conjugation) map from the former to the latter. A \textbf{dihedral} representation $\rho$ is such that $\rho(\pi_1M)$ is of the form $A \rtimes \Z_2$ with $\Z_2$ acting on $A$ by inverting elements. 

\begin{thm}\label{xdv_for_all_rho_if_all_horo_omni}
Let $M$ be the interior of a compact, connected, orientable 3-manifold with non-empty boundary consisting of a disjoint union of tori. Let $\mathcal{T}$ be an ideal triangulation of $M$ such that every horo-normal surface is porous.  Then $\mathfrak{R}_{\mathcal{T}} : \widehat{\mathfrak{D}}(M;\mathcal{T}) \to \mathfrak{R}(M)$ maps onto the irreducible non-dihedral representations, up to conjugation.
\end{thm}

This is proved in Section \ref{Representations}. The condition of every horo-normal surface being porous is relatively mild. For example, in Theorem \ref{small irred one cusp}, we show that for \emph{any} ideal triangulation of a small (i.e.\thinspace for which every closed incompressible surface is boundary parallel) irreducible manifold with a single cusp, every horo-normal surface is porous. Thus, for these manifolds any triangulation gives Theorem \ref{xdv_for_all_rho_if_all_horo_omni}, even in cases when the standard deformation variety misses components, or is even empty. \\

However, when the manifold has more than one cusp, or is not small, a given triangulation may not have every horo-normal surface porous. In contrast to the situation with the standard deformation variety, we give an algorithm (in Section \ref{Retriangulating}) for finding a ``good'' triangulation, for which every horo-normal surface is porous. From this, we get the following version of the result:

\begin{thm}\label{xdv_for_all_rho}
Let $M$ be the interior of a compact, connected, orientable 3-manifold with non-empty boundary consisting of a disjoint union of tori. Then there exists an ideal triangulation $\mathcal{T}_*$ of $M$ such that  $\mathfrak{R}_{\mathcal{T}_*} : \widehat{\mathfrak{D}}(M;\mathcal{T}_*) \to \mathfrak{R}(M)$ maps onto the irreducible non-dihedral representations, up to conjugation.
\end{thm}

\begin{proof}
This follows from Corollary \ref{get_all_porous} (which gives us the ideal triangulation $\mathcal{T}_*$, for which every horo-normal surface is porous) and Theorem \ref{xdv_for_all_rho_if_all_horo_omni}.
\end{proof}

Note that although the extended deformation variety still requires a particular choice of triangulation in its definition, the above result shows that we get all of the irreducible non-dihedral representations. Thus, this gives a way to describe representations using shapes of ideal tetrahedra, but in a triangulation-independent way, in the sense of describing all irreducible non-dihedral representations. \\

In Section \ref {A-polynomial} we give an application relevant to calculating the $\PSL$ A-polynomial, with the following result. Notation: The polynomials $H^{(\mathcal{T}_*;S)}(l,m)$ are the factors of the A-polynomial detected by the subset of the extended deformation variety with tetrahedra degenerate in the way described by the horo-normal surface $S$ (see Section \ref{The extended deformation variety} for details).

\begin{thm}\label{thm_get_A-poly2}
Let $N$ be a connected topologically finite 3-manifold with a single torus boundary component. Then there exists an ideal triangulation $\mathcal{T}_*$ of $N$ so that the polynomials $H^{(\mathcal{T}_*;S)}(l,m)$, ranging over each horo-normal surface $S$, contain all factors of the $\PSL$ A-polynomial of $N$ associated to components of irreducible non-dihedral representations.\end{thm}

In contrast, the corresponding result for the standard deformation variety tells us only that the associated polynomial $H(l,m)$ divides the $\PSL$ A-polynomial (see Champanerkar~\cite{champanerkar_thesis}). Using this result, we get all factors of the $\PSL$ A-polynomial associated to components of irreducible non-dihedral representations. However, it is not currently clear whether or not those factors could be repeated for different horo-normal surfaces, and so this doesn't necessarily give us the A-polynomial outright.\\

As an example, in Section \ref{8 20} we calculate a factor of the $\PSL$ A-polynomial of the complement of the knot $8_{20}$ that was not found in calculations by Marc Culler that used only the standard deformation variety associated to the triangulation computed by Joe Christy. The extra factor comes from a component of the character variety that is missed by the standard deformation variety with many choices of triangulation. In contrast, the extended deformation variety detects this component and so detects the factor of the A-polynomial.\\

The author thanks Marc Culler, Eric Katz, Thomas Mattman, Alan Reid, Stephan Tillmann and Genevieve Walsh  for helpful discussions, and the anonymous referee, whose very helpful comments greatly improved the introduction and motivation sections of the paper. This work was partially supported by an NSF RTG grant, and partially by Australian Research Council grant DP1095760.

\section{Motivation}

A key ingredient of the construction of the extended deformation variety is a tree associated to $\C((\zeta))$, the set of Laurent series over the complex numbers. This tree is a special case of the Bruhat-Tits building for $\text{GL}(2,F)$, where $F$ is a field with a discrete rank 1 valuation. In our case $F=\C((\zeta))$ and the valuation assigns to a Laurent series the minimal degree of its non-zero terms. Actions of 3-manifold groups on this tree have been widely studied, particularly in order to construct incompressible surfaces. See \cite{cullershalen83, handbook_shalen} (particularly the latter for background on this section).\\

 Given any algebraic curve $C$ in the $\text{SL}(2,\C)$ representation variety of $M$, one obtains a ``tautological'' representation of $\pi_1M$ into $\text{SL}(2,\C(C)).$ This follows by viewing the four coordinates of the representation into $\text{SL}(2,\C)$ as functions on $C$, and hence in $\C(C).$ A smooth point $p$ of $C$ determines a discrete valuation on $\C(C)$. There is a natural embedding of $\C(C)$ into $\C((\zeta))$, obtained by expanding rational functions on $C$ as Laurent series in a local coordinate $\zeta$, where $\zeta=0$ corresponds to $p$. Moreover, the restriction of the standard valuation on $\C((\zeta))$ agrees with the valuation determined by $p$. Using this, we get a representation into $\text{SL}(2,\C((\zeta)))$ for each such point $p$. We can think of this as giving a parameterisation of a neighbourhood of $p$ by the variable $\zeta$.\\

An action on the Bruhat-Tits tree is \textbf{trivial} if some vertex is fixed by every element of $\pi_1M$.  When $p$ is an ideal point we get a non-trivial action on the Bruhat-Tits tree, from which the existence of an incompressible surface follows. However, for the purposes of this paper, the relevant actions are the trivial actions, which give no information about incompressible surfaces. The stabiliser of a vertex fixed by every element of $\pi_1M$ is conjugate to a subgroup of $\text{SL}(2,\mathcal{O})$, where $\mathcal{O}$ is the valuation ring of $\C((\zeta))$, in this case the subring $\C[[\zeta]]$ of power series in $\zeta$. Setting $\zeta=0$, we get a representation into $\text{SL}(2,\C)$ (and so into $\PSL$ by projecting). This gives a natural way to associate a point of the character variety to one of these trivial actions. Although this is inaccurate in a way which will become clear later in this section, it is useful to think of a point of the extended deformation variety as encoding this trivial action in terms of shapes of possibly degenerate ideal tetrahedra. From the action on the tree, one can then recover the point of the character variety, which would give Theorem \ref{xdv_for_all_rho_if_all_horo_omni}.\\

We will not in fact construct such an action, but it is helpful to consider the reverse process: Suppose we are given an ideal triangulation $\mathcal{T}$ of a manifold $M$ and an irreducible non-dihedral representation $P:\pi_1M\to\text{SL}(2,\C((\zeta)))$ which is trivial in the sense that a vertex of the associated Bruhat-Tits tree $T_\zeta$ is fixed by the entire image of $P$. (Therefore we can actually assume that  $P:\pi_1M\to\text{SL}(2,\C[[\zeta]])$.) Then, we can construct a $\pi_1M$-equivariant map $\Psi$ from the universal cover $\til{M}$ of $M$ (with vertices of $\til{\mathcal{T}}$  adjoined, corresponding to the cusps of $\til{M}$) to the tree $T_\zeta$ (with ends adjoined) which sends each tetrahedron to the convex hull in $T_\zeta$ of a possibly degenerate 4-tuple of ends of $T_\zeta$. Since $M$ has torus boundary components, the stabiliser of a cusp of $\til{M}$ is abelian, so fixes at least one end of $T_\zeta.$ Choose one cusp of $\til{M}$ from each orbit (i.e.\thinspace one for each cusp of $M$), and map it to an arbitrarily chosen end with the same stabiliser, and then extend equivariantly.  This determines where the vertices of the triangulation $\til{\mathcal{T}}$ map to. Extend the map over the edges by sending each edge to the geodesic line joining the images of the endpoints, and then extend over the higher skeleta so that each tetrahedron is mapped into the convex hull of its vertices. The image of this map will contain a fixed vertex $H$ of $T_\zeta$, and with appropriate choices for extending the map to the edges, triangles and tetrahedra, the preimage of the set of midpoints of edges in $T_\zeta$ is a $\pi_1M$-invariant surface $\til{S}$, which is normal relative to the triangulation $\til{\mathcal{T}}$. The image of $\til{S}$ under the covering projection is a normal surface $S$ in $M$, and is a dual surface for the action on $T_\zeta$. The image of one of the complementary components of $\til{S}$ in $\til{M}$, $\til{R_\text{in}}$ say, contains the fixed vertex $H$. Throw away any components of $\til{S}$ that are not incident to $\til{R_\text{in}}$ to obtain a normal surface $\til{S'}$, which is once again $\pi_1M$-equivariant, so projects to a normal surface $S'$ in $M$. The region $\til{R_\text{in}}$ similarly projects down to a component $R_\text{in}$ which carries $\pi_1M$.\\

One might expect that $S'$ consists of vertex-linking tori. That is, that $S'$ has exactly one copy of each normal triangle, and no quads. In general however, there can be some quads. In either case, $S'$ is an example of a porous horo-normal surface studied in this paper. \\

A key example illustrating how there can be quads in the porous horo-normal surface is if a single edge $e$ of the ideal triangulation $\mathcal{T}$ of a hyperbolic manifold $M$ is inessential, i.e.\thinspace homotopic into $\bdry M$.  Note that by \cite{st_essential}, the standard deformation variety for such a triangulation is empty. However, following the above construction starting from a generic point of the Dehn surgery component and producing a representation into $\text{SL}(2,\C((\zeta)))$,
we then get a horo-normal surface which is the boundary of a small regular neighbourhood of the union of the boundary tori and $e$. Since both ends of the edge are at the same cusp of $\til{M}$, the two endpoints map to the same end of $T_\zeta$ under $\Psi$, so $e$ also maps to this end. In this example, the inessential edge is entirely contained within the ``outside'' region $R_\text{out}$ associated to the horo-normal surface (the union of the complementary components that contain the cusps). \\

Essential edges can also be entirely contained within $R_\text{out}$ for certain representations. For example, consider a manifold $M$ which is a cover of another manifold $N$. Then a subset $Y$ of the character variety of $M$ corresponds to the character variety of $N$. Suppose also that $M$ has a triangulation $\mathcal{T}$, with an edge $e$ which maps (under the covering map) to an arc that is homotopic into the boundary of $N$. Similarly to as in the previous example, a generic point on $Y$ produces a map $\Psi$ under which the edge $e$ maps into an end of $T_\zeta$, and again $e$ is contained in $R_\text{out}$. The standard deformation variety with the triangulation $\mathcal{T}$ misses such a point (since the shape of a tetrahedron that has $e$ as an edge would be degenerate).\\

For these two examples, and also for the examples in Sections \ref{LLR_ex_part_1} and \ref{8 20}, we get horo-normal surfaces that are not vertex-linking tori for all points belonging to entire components of their respective character varieties. It can also happen that for most of a component we get vertex-linking tori, and only get a horo-normal surface with quadrilaterals at isolated points. \\

The $\pi_1M$-equivariant map $\Psi:\til{M}\to T_\zeta$ is strictly analogous to a developing map  $\til{M} \to\Hthree$. The image $\Psi(t)$ of a tetrahedron $t\in\til{\mathcal{T}}$ is a subtree with 4 (ordered) ends on the boundary of $T_\zeta$, which have a well defined cross ratio in $\C((\zeta))$. These cross ratios are preserved by the action of $M$. Given an edge $e$ of a tetrahedron $t$, the two faces adjacent to $e$ map to ``tripods'' in $T_\zeta$, and the cross ratio associated to $t$, with the appropriate ordering, determines the element of $\text{SL}(2,\C((\zeta)))$ which takes one tripod to the other while preserving the line $\Psi(e)$.\\

The construction in this paper is almost a converse to the construction above. Beginning with a representation $\rho:\pi_1M\to\text{SL}(2,\C)$, we construct something which is analogous to a developing map from $\til{M}$ into $T_\zeta$. However, the analogy is, by design, very weak. We do not assume the existence of any $\text{SL}(2,\C((\zeta)))$ representation that specialises to $\rho$ when $\zeta=0$. 
We have a special vertex $H$ of $T_\zeta$ as above. The representation $\rho$ acts on $\Hthree$. Similarly to as in the above discussion in which we construct $\Psi$ from $P$, we can use $\rho$ to construct a $\pi_1M$-equivariant map $\psi$ from the universal cover $\til{M}$ of $M$ (with vertices of $\til{\mathcal{T}}$ adjoined, corresponding to the cusps of $\til{M}$) to $\Hthree$ (with $\bdry \Hthree$ adjoined). We view this $\Hthree$ and its boundary as corresponding to $H$ and the edges of $T_\zeta$ leaving $H$. The vertices of the tetrahedra of $\til{\mathcal{T}}$ are then positioned at the midpoints of these edges of $T_\zeta$ leaving $H$. If we could push these vertices out to the ends of $T_\zeta$ in an equivariant manner, then we would be able to reconstruct a $\text{SL}(2,\C((\zeta)))$ representation that specialises to $\rho$ when $\zeta=0$. As previously mentioned, we will not do this. In fact we will only push the vertices out one edge further\footnote{Actually the construction is slightly weaker even than this. We only record the relative position of two vertices of $\til{\mathcal{T}}$ that share an edge that gets mapped to a single point of $\bdry \Hthree$ by $\psi$, not their individual absolute positions.}. The point of pushing these vertices outwards is to separate vertices that would otherwise be coincident, giving each degenerate tetrahedron (that would ordinarily have shape parameter zero) a non-zero shape in $\C[[\zeta]]$. Pushing only one step outwards is equivalent to having only ``lowest order'' information about the positions of the vertices, and so we read off only ``lowest order'' information about the shapes of degenerate tetrahedra. This extra data is however enough to support developing through these degenerate tetrahedra.\\

The map $\psi$ tells us which edges of the triangulation map into single points on $\bdry \Hthree$, and this determines a horo-normal surface as the boundary of a regular neighbourhood of the complex generated by these ``zero-length'' edges. The degrees of the shapes of the tetrahedra in $\C[[\zeta]]$ are determined by the intersection of the horo-normal surface with the tetrahedron. There is a way to generate the shapes of the degenerate tetrahedra in an equivariant way.\\

If we start with the shapes of the (degenerate and non-degenerate) tetrahedra, we can develop through paths of tetrahedra into $T_\zeta$. In this construction, the developed positions of the vertices of the tetrahedra are given by elements of $\C((\zeta))$ (which we identify with the ends of $T_\zeta$), but only up to ``lowest order information'' since we only have lowest order information about the tetrahedron shapes. This effectively allows us to reconstruct the extended version of $\psi$, with the vertices pushed out one edge further. If we ignore the higher order information, ``specialising at $\zeta=0$'',  then we recover $\psi$, and from this get the data needed to reconstruct the representation $\rho$. \\

The gluing equations are generalised to a ``consistent developing condition'', which says that the lowest order positions of the vertices of $\til{M}$ are consistently determined, no matter what path of tetrahedra we develop through to get to them. Our degenerate tetrahedra have non-zero shapes of the form $\zeta z_1$ or $\zeta^2 z_2$, the non-degenerate tetrahedra have shapes of the form $z_0\in\C\setminus\{0,1\} \subset \C[[\zeta]]$ and the consistent developing condition is realised as a set of polynomial equations in the coefficients of these shapes. These determine an affine algebraic variety over $\C$, and this is the generalisation of the deformation variety, the ``extended deformation variety''. \\

In nice cases, we can derive this system of polynomial equations in a natural way. Consider an irreducible, non-dihedral representation $\rho: \pi_1M \to \PSL$,
and suppose we are trying to construct a curve in the representation variety which passes through $\rho$. Note that this is possible, even if the character of $\rho$ does not lie on a 1-dimensional component of the character variety. One might attempt this by trying to solve the gluing equations so that the tetrahedra shapes are in $\C[[\zeta]]$, carrying out the developing map construction and from there get a representation. Reversing the construction of the tautological representation, one would obtain a curve of representations. One approach to finding such a power series solution goes back to Newton and Puiseaux, and is now known as ``tropical algebraic geometry''. We replace each variable in each of the gluing equations with a formal power series in $\zeta$ and try to solve for the terms in the series recursively. For each gluing equation, setting the coefficient of the lowest degree term equal to zero produces a polynomial equation in the coefficients of the lowest degree terms of the formal series. If one can find a solution to these equations having all coordinates non-zero then then recursion can be continued. A necessary condition for being able to do this is that each equation should have at least two monomials which contribute to the lowest degree term. This is only possible if one chooses the orders of the formal power series correctly. The orders must satisfy a certain system of linear equations and inequalities. (A tropical algebraic geometer would say that these linear equations and inequalities define the tropical prevariety associated to the gluing equations, and the vector of orders must lie in this prevariety.) However, in the case of the gluing equations, at the first step of the recursion this system of linear equations and inequalities are precisely the Q-matching equations for spun-normal surfaces relative to the triangulation (with the inequalities simply saying that the quadrilateral weights are non-negative). So the starting point for using this recursive process to find a power series solution to the gluing equations is a normal surface. In our case, this normal surface is the porous horo-normal surface.\\

For a single polynomial equation, this recursive procedure will always continue to work, and will lead to a power series solution.
This is essentially the Newton-Puiseux algorithm for resolving singularities of plane algebraic curves. But it is a standard issue in tropical algebraic geometry that this fails for systems with more than one equation. However, the system of polynomial equations produced at the first step of the recursion, which is all that the normal surface produces, is already useful. These are the equations that define the extended deformation variety in this nice case. A solution to these equations need not lead to a power series solution and hence need not lead to a representation in $\text{PSL}(2,\C((\zeta)))$. However, every solution does determine the very weak analogue of a developing map and has a ``specialisation at $\zeta=0$'' which is a representation in $\PSL$. Moreover, every irreducible non-dihedral representation is realised as such a specialisation. This is the content of Theorem \ref{xdv_for_all_rho_if_all_horo_omni}. 

\section{The standard deformation variety}\label{The deformation variety}
Let $M$ be a topologically finite 3-manifold which is the interior of a compact 3-manifold with non-empty boundary consisting of a disjoint union of tori. An \textbf{ideal triangulation} $\mathcal{T}$ of $M$ consists of a pairwise disjoint union of standard Euclidean 3--simplices, $\widetilde{\Delta} = \cup_{k=1}^{n} \widetilde{\Delta}_k,$ together with a collection $\mathcal{I}$ of Euclidean isometries between the 2--simplices in $\widetilde{\Delta},$ called {\bf face pairings}, such that 
$(\widetilde{\Delta} \setminus \widetilde{\Delta}^{(0)} )/ \mathcal{I}$ is homeomorphic to $M.$
The simplices in $M$ may be singular. It is well-known that every non-compact, topologically finite 3--manifold admits an ideal triangulation. 

\begin{defn}\label{defm_var}
Let $M$ be a 3-manifold with non-empty boundary consisting of a disjoint union of tori, and with ideal triangulation $\mathcal{T}$ consisting of $N$ tetrahedra. The \textbf{standard deformation variety} of $M$ with respect to the triangulation $\mathcal{T}$, $\mathfrak{D}(M) = \mathfrak{D}(M;\mathcal{T})$ is the affine variety in $(\mathbb{C}\setminus\{0,1\})^{3N}$, defined as the solutions of gluing equations and identities between complex dihedral angles within each tetrahedron, where each of the three complex dihedral angles in each tetrahedron corresponds to a dimension of the ambient space. Specifically, for the 6 dihedral angles within each tetrahedron, angles on opposite edges are the same, as shown in Figure \ref{complex_dihedral_angles.pdf} (hence the fact that there are three complex variables for each of the $N$ tetrahedra), and $x_1,x_2,x_3$ are related to each other by:
\begin{equation}\label{dv_eq1}
x_1x_2x_3=-1
\end{equation}
\begin{equation}\label{dv_eq2}
x_1x_2 - x_1 + 1 = 0
\end{equation}
For each edge of $\mathcal{T}$, we also require that the product of the complex dihedral angles arranged around an edge of the triangulation equals 1 (these are the {\bf gluing equations}).
\end{defn}
This definition is first seen in Thurston's notes \cite{thurston}, chapter 4.\\ 

\begin{figure}[htb]
\centering
\includegraphics[width=0.3\textwidth]{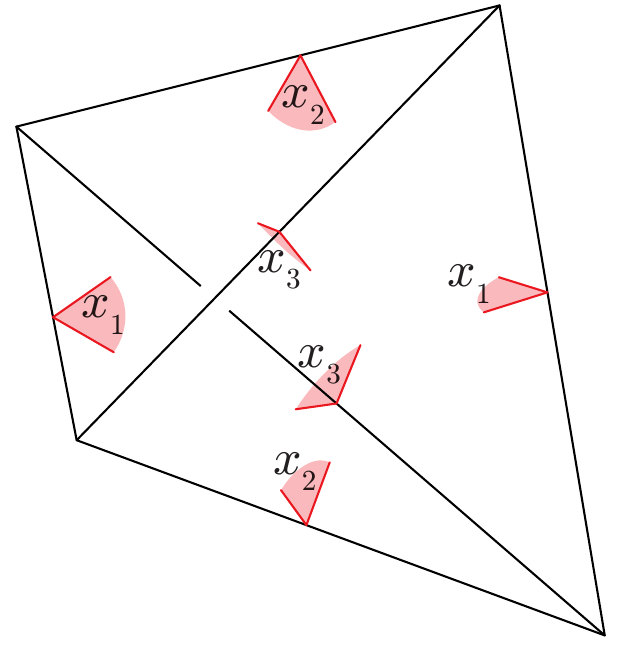}
\caption{The 6 dihedral angles in a tetrahedron.}
\label{complex_dihedral_angles.pdf}
\end{figure}

Let $\til{M}$ be the universal cover of $M$ with induced triangulation $\til{\mathcal{T}}$. Let $\til{\mathcal{V}}$ be the set of vertices of $\til{\mathcal{T}}$\footnote{$\til{\mathcal{V}}$ is also the set of cusps of $\til{M}$, and so is independent of $\mathcal{T}$.}. Given a point $Z \in \mathfrak{D}(M;\mathcal{T})$ the {\bf developing map} $\Phi_Z:\til{\mathcal{V}}\rightarrow \bdry \Hthree$ is defined up to conjugation as follows\footnote{Note that the map is usually defined as from $\til{M}$ to $\Hthree$ (as in the introduction), but our formulation encodes equivalent data and is more natural in the context of this paper.} (also see Yoshida \cite{yoshida91} and Tillmann \cite{tillmann_degenerations}).  If $\tri$ is a triangle in $\til{\mathcal{T}}$, we arbitrarily choose distinct images for the three vertices of $\tri$ in $\bdry \Hthree$. If  $t$ is one of the two tetrahedra incident to $\tri$ in $\til{\mathcal{T}}$, then $t$ has an associated ideal hyperbolic shape, given by $Z$. This shape, together with the known positions of three of its vertices on $\bdry\Hthree$ determine the position of the fourth, and so the image of that vertex under $\Phi_Z$. We repeat this procedure, spreading out through the tetrahedra of $\til{\mathcal{T}}$ to determine the whole map. The gluing equations ensure that we get the same answer, no matter which path through the tetrahedra we take to reach a given vertex. The arbitrary choice in the positions of the initial three vertices corresponds to the definition being up to conjugacy.\\
\begin{defn}
Let $\mathfrak{R}(M)$ be the set of representations $\rho:\pi_1M\rightarrow \PSL \cong \text{Isom}(\Hthree)$. The developing map construction gives a map 
$$\mathfrak{R}_\mathcal{T}: \mathfrak{D}(M;\mathcal{T}) \rightarrow \mathfrak{R}(M)$$
 (again up to conjugation). For a given $Z\in  \mathfrak{D}(M;\mathcal{T})$, we construct the developing map $\Phi_Z$, and then read off a corresponding representation (the holonomy representation) $\rho_Z$ as follows. The three vertices $v_1,v_2,v_3$ of $\tri$ have distinct positions $\Phi_Z(v_i)\in\bdry\Hthree$. An element $\gamma\in\pi_1M$ acts on the vertices as a deck transformation, and the images $\gamma v_i$ have distinct positions $\Phi_Z(\gamma v_i)\in\bdry\Hthree$. A triple of distinct points of $\bdry\Hthree$ maps to another triple of distinct points under a unique element of $\PSL$, and this is the value we take for $\rho_Z(\gamma)$.\end{defn}

\subsection{Dependence of the deformation variety on the triangulation}\label{sec_dep_on_tri}
\begin{thm}[Matveev, Theorem 1.2.5 of \cite{matveev}] \label{thm_matveev}
If $\mathcal{T}$ and $\mathcal{T}'$ are two ideal triangulations of a given manifold $M$, each of which has at least two tetrahedra, then $\mathcal{T}$ and $\mathcal{T}'$ are connected by a sequence of 2-3 and 3-2 moves (see Figure \ref{2-3moves}).
\end{thm}
\begin{figure}[htb]
\centering
\includegraphics[width=0.5\textwidth]{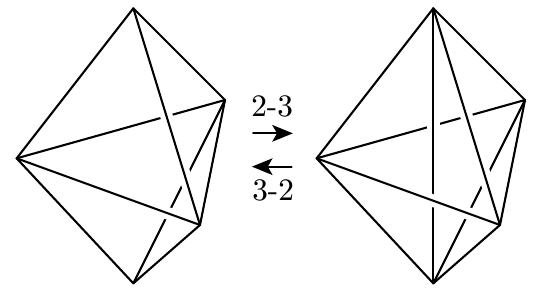}
\caption{2-3 and 3-2 moves.}
\label{2-3moves}
\end{figure}
We might hope that the deformation varieties for triangulations $\mathcal{T}_2$ and $\mathcal{T}_3$ of a manifold that differ by a 2-3 move might be equivalent in some sense, and then given the above theorem and induction we would get equivalence for any triangulation (with a possible exception for the manifolds that have a one-tetrahedron triangulation). We would expect to be able to convert between points of $\mathfrak{D}(M;\mathcal{T}_2)$ and $\mathfrak{D}(M;\mathcal{T}_3)$ as follows:\\

\begin{figure}[htb]
\centering
\includegraphics[width=0.7\textwidth]{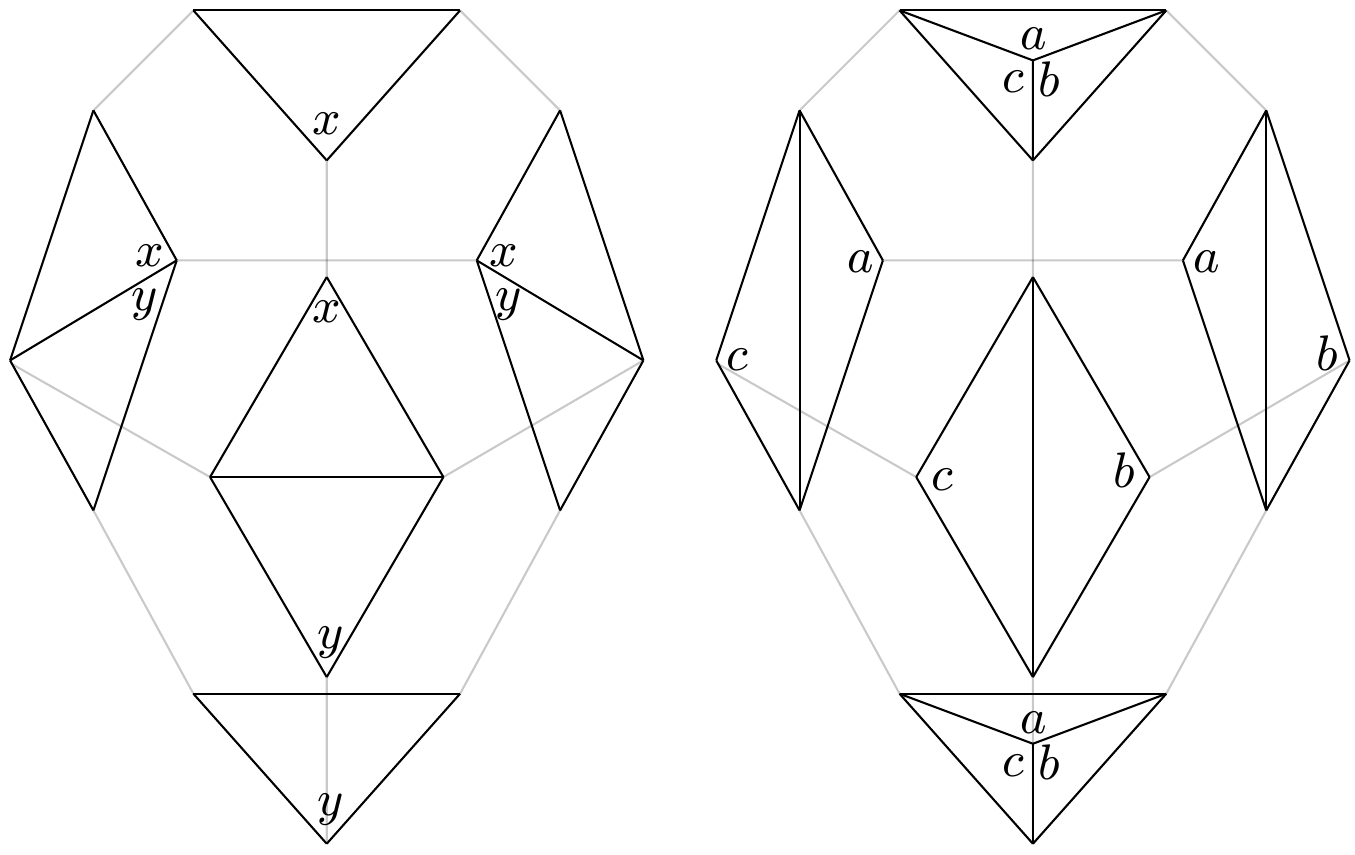}
\caption{The 2-3 and 3-2 moves with truncated ends, showing the dihedral angles.}
\label{2-3_truncated}
\end{figure}

Figure \ref{2-3_truncated} shows two tetrahedra labelled with complex dihedral angles $x$ and $y$ which share a face, and three tetrahedra labelled by $a, b$ and $c$ which all share an edge. We label the angles within one tetrahedron $x_1 = x, x_2 = \frac{x-1}{x}, x_3 = \frac{1}{1-x}$, moving clockwise from $x$ on each truncated triangular end and similarly for the other tetrahedra. The extra edge for the three tetrahedra gives us the equation $a_1b_1c_1=1$. If we are to have corresponding points of the two deformation varieties, then the dihedral angles outside of the six-sided shape in which the 2-3 move is performed must be the same. Because of the gluing equations, the dihedral angles inside must also be the same and we have the following relations:
\[
\begin{array}{ccc}a_1 = x_1 y_1 & x_1 = b_3 c_2 & y_1 = b_2 c_3 \\b_1 = x_2 y_3 & x_2 = c_3 a_2 & y_2 = a_2 b_3 \\c_1 = x_3 y_2 & x_3 = a_3 b_2 & y_3 = c_2 a_3\end{array}\]
In good situations, these equations do allow us to produce a birational map between two deformation varieties of a manifold with triangulations that differ by a 2-3 move. However, this does not always work, in particular if we happen to have the relation in $\mathfrak{D}(M;\mathcal{T}_2)$ that $xy=1$, as we will see in the following example.

\section{Examples, part 1: The once punctured torus bundle with monodromy $LLR$}\label{LLR_ex_part_1}
We will consider two examples of this phenomenon. The first is somewhat artificial but demonstrates the phenomenon well, and we will return to it in Section \ref{Example, part 2} to show how the extended deformation variety solves the problem. In Section \ref{8 20} we will see a more natural occurrence of the issue, in the calculation of the $\PSL$ A-polynomial for the knot $8_{20}$.\\

In this section we give an example of two triangulations of a manifold such that the deformation variety for one triangulation contains an entire component that does not appear in the deformation variety for the other triangulation.
We consider the punctured torus bundle $M_{LLR}$ with monodromy given by $LLR$ (see for example, Gu\'eritaud \cite{gueritaud_torus_bundles} for the notation), and show two triangulations of this punctured torus bundle, $\mathcal{T}_4$ and $\mathcal{T}_5$ (with 4 and 5 tetrahedra respectively, neither is canonical) in Figure \ref{LLR_4_vs_5_tetra.pdf}. Vertices of the triangulation of the torus boundary correspond to edges of the triangulation, and we can read off the gluing equation from the corners of triangles incident to vertices. We see both ends of each edge so each equation appears twice. \\

\begin{figure}[htb]
\centering
\includegraphics[width=1.0\textwidth]{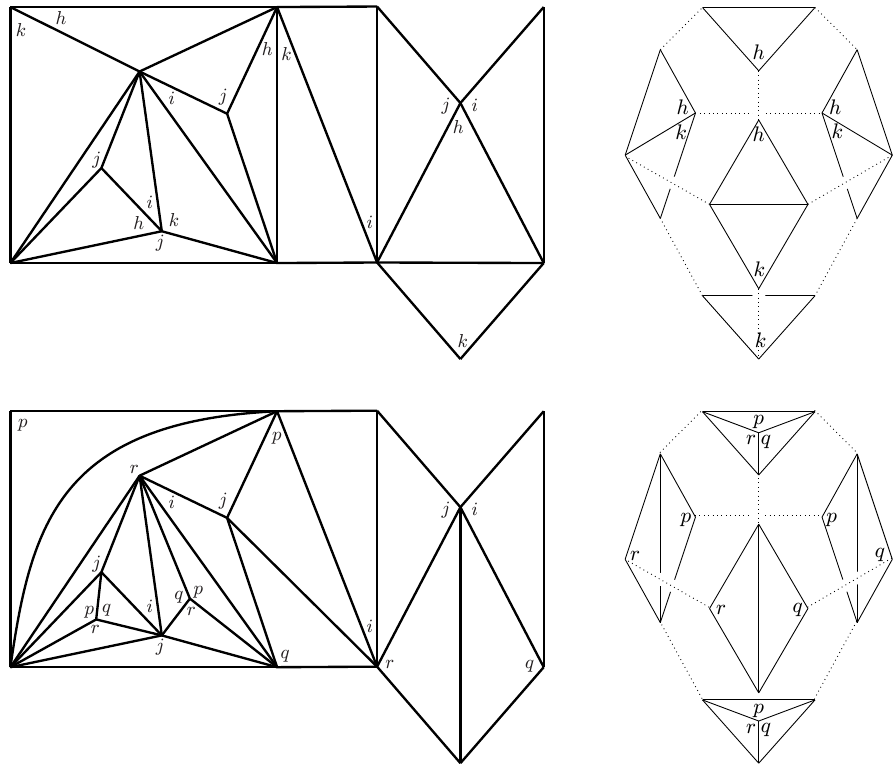}
\caption{Induced boundary torus triangulations of two triangulations, $\mathcal{T}_4$ and $\mathcal{T}_5$, of the punctured torus bundle $M_{LLR}$. The edges of each fundamental domain are identified in the obvious way. In the top left is a fundamental domain for the triangulation on the torus boundary induced by $\mathcal{T}_4$. There are four tetrahedra in this triangulation, with angles labelled $h,i,j,k$. Each tetrahedron has four ends and we see these four truncated ends as triangles on the torus boundary. To the right is shown the tetrahedra involved in the 2-3 move, and in the bottom left the resulting fundamental domain for $\mathcal{T}_5$, with angles labelled $i,j,p,q,r$. }
\label{LLR_4_vs_5_tetra.pdf}
\end{figure}

For $\mathfrak{D}(M_{LLR};\mathcal{T}_4)$ we obtain the following gluing equations:
\begin{eqnarray}
\label{hijk_first} hijk & = & 1\\
\frac{i-1}{i} j \frac{1}{1-h} & = & 1\\
i \frac{1}{1-i} \frac{j-1}{j} \frac{1}{1-j} \frac{1}{1-h} \left(\frac{k-1}{k}\right)^2 & = & 1\\
\label{hijk_last} \frac{i-1}{i} \frac{1}{1-i}\frac{j-1}{j} \frac{1}{1-j}h\left(\frac{h-1}{h}\right)^2k\left(\frac{1}{1-k}\right)^2 &=&1
\end{eqnarray}
One gluing equation always depends on the others, so we discard the last of these and the others simplify to:
\begin{eqnarray}
hijk & = & 1\\
\frac{i-1}{i} j \frac{1}{1-h} & = & 1\\
 \frac{i}{1-i} \frac{-1}{j} \frac{1}{1-h} \left(\frac{k-1}{k}\right)^2 & = & 1
\end{eqnarray}

The variety consists of two 1-dimensional components, one of which contains the complete structure, and the other of which satisfies the extra condition that $hk=1$. Then these equations become:
\begin{eqnarray}
ij &=& 1\\
\frac{i-1}{i} j \frac{1}{1-h} & = & 1\\
\frac{i}{1-i} \frac{-1}{j} (1-h) & = & 1
\end{eqnarray}
The latter two equations are redundant and we get a 1-dimensional variety. For $\mathfrak{D}(M_{LLR};\mathcal{T}_5)$ we obtain:

\begin{eqnarray}
\label{ijpqr_first} pqr &=& 1 \\
\label{ijpqr_second} ij\frac{q-1}{q}\frac{1}{1-q}\frac{r-1}{r}\frac{1}{1-r}&=& 1 \\
\frac{i-1}{i}j\frac{1}{1-p}\frac{q-1}{q}&=& 1 \\
i \frac{1}{1-i} \frac{j-1}{j} \frac{1}{1-j}\frac{p-1}{p}\frac{1}{1-q}r&=& 1 \\
\label{ijpqr_last} \frac{i-1}{i} \frac{1}{1-i}\frac{j-1}{j}\frac{1}{1-j} p \frac{p-1}{p}\frac{1}{1-p}q\frac{r-1}{r}\frac{1}{1-r}&=&1
\end{eqnarray}
Equations (\ref{ijpqr_second}) through (\ref{ijpqr_last}) correspond to (\ref{hijk_first}) through (\ref{hijk_last}), and we get the extra equation (\ref{ijpqr_first}). Again discarding the last equation and simplifying we get:
\begin{eqnarray}
\label{trouble1} pqr &=& 1 \\
\label{trouble2} ij\frac{1}{q}\frac{1}{r}&=& 1 \\
\frac{i-1}{i}j\frac{1}{1-p}\frac{q-1}{q}&=& 1 \\
\frac{i}{1-i} \frac{-1}{j}\frac{p-1}{p}\frac{1}{1-q}r&=& 1 
\end{eqnarray}

This time, if $ij=1$ then equations (\ref{trouble1}) and (\ref{trouble2}) imply that $p=1$, and there is no solution. 

\begin{rmk}
More generally, if a particular component of the representation variety has an extra relation on it (in addition to those defining the variety) which implies that certain cusps always appear in the same place under $\Psi_\rho$ (see Definition \ref{psi_rho}) as we vary $\rho$ within the component, then a triangulation with an edge between two such cusps will have an associated deformation variety that will not see this component. 

\end{rmk}

\section{Extending the deformation variety: Overview}

In this paper we define and then prove properties of a generalisation of the deformation variety, which solves these problems of degenerate tetrahedra and dependence on the triangulation whilst retaining many useful properties of the deformation variety, in particular, a map to the representation variety. \\

The first step, in Section \ref{tree}, is to move to a more general setting, replacing dihedral angles (i.e. cross ratios) in $\C \setminus \{0,1\}$ with dihedral angles in $\C((\zeta)) \setminus \{0,1\}$, formal Laurent series in a variable $\zeta$.  This allows tetrahedra to be degenerate whilst retaining enough information to develop through them, and so define a developing map. Ohtsuki~\cite{ohtsuki97} uses a similar idea to describe the shape of a tree on which the fundamental group of a manifold acts in the Culler-Shalen construction associated to an ideal point. The elements of the field $\C((\zeta)) \setminus \{0,1\}$ can be viewed as the ends of a tree, the Bruhats-Tits building for $\text{GL}(2,\C((\zeta)))$. Also see Serre \cite{serre_trees}.\\

In this setting the ends of the tree (which are Laurent series) correspond to cusps in the universal cover of the manifold, and so the cross ratio associated to an ideal tetrahedron should be a Laurent series as well. In the case of an ideal point we expect tetrahedra to be degenerate, which corresponds to cross ratios with first non-zero term $\zeta^k z_k$ for $k>0$ and $z_k \in \C$. In general, we would expect the Laurent series to continue with infinitely many further terms for higher powers of $\zeta$.\\

In our case we may also have degenerate tetrahedra, but they do not necessarily come from an ideal point. In Section \ref{LLR_ex_part_1} for example,  there is generally no way to approach the degenerate shapes from non-degenerate shapes. We also want our extension of the deformation variety to be a finite dimensional complex variety, and so we do not want to deal with cross ratios being Laurent series with infinitely many complex coefficients. Instead, the idea is to consider only the first non-zero term (``lowest order term'') of a Laurent series (in fact we will only use cross ratios of the form $z$, $\zeta z$ and $\zeta^2 z$ to describe degenerate tetrahedra). It turns out that in good situations, the lowest order terms of the cross ratios are enough to determine the developed positions of vertices of tetrahedra, and therefore the developing map, up to lowest order. Making this idea precise is the main purpose of Section \ref{tree}.\\ 

In Section \ref{The extended deformation variety} we consider the subset of edges of a given triangulation that are supposed to be of ``zero length'', meaning that the vertices at each end are meant to be in the same place under the developing map. These must satisfy a certain condition (``no bad loops''), and given this condition we construct a normal surface (which we call a ``horo-normal surface'') surrounding the zero length edges. In fact we can see the surface as the boundary of a neighbourhood of the complex generated by the zero length edges. The horo-normal surface and the set of zero length edges determine each other, and so record the same data. The surface cuts the manifold $M$ into inside and outside regions, and in the case that the lift of the inside region is connected in $\til{M}$ we say that the surface is ``porous''.\\

 We define the extended deformation variety $\widehat{\mathfrak{D}}(M;\mathcal{T};S)$ of a manifold $M$ with a given triangulation $\mathcal{T}$ and porous horo-normal surface S, by specifying the lowest order terms of cross ratios for tetrahedra that intersect the inside region. The results of Section \ref{tree} then imply that we can develop through paths of triangles in the inside region using only the lowest order terms of the cross ratios. The surface needs to be porous for us to be able to develop through paths of triangles in the inside region to reach every vertex. In place of the gluing equations of the deformation variety, we have a more general condition that states that if we develop along paths of triangles to a vertex, the position should not depend on the path. We show in Theorem \ref{is_variety} that $\widehat{\mathfrak{D}}(M;\mathcal{T};S)$ is indeed an affine algebraic variety.\\
 
In Section \ref{porous horo-normal surfaces} we deal with the issue of requiring the surface to be porous. In Section \ref{small one cusp manifolds} we show that if the manifold is small and has a single torus boundary component, then for any triangulation, every horo-normal surface is porous. If however the manifold is not small or has multiple boundary components, we may need to retriangulate. In Section \ref{Retriangulating} we give an algorithm that modify a given triangulation $\mathcal{T}$ by inserting ``pillows'', pairs of tetrahedra that share two faces along a single edge. Algorithm \ref{fix_all_horo-normal} modifies a triangulation so that any horo-normal surface relative to the new triangulation is porous.\\

If we have a developing map from the universal cover of $M$ to $\Hthree$ then we get a representation into $\PSL$, and so we get a homomorphism from $\widehat{\mathfrak{D}}(M;\mathcal{T};S)$ to the representation variety $\mathfrak{R}(M)$, defined up to conjugation. We make this precise and prove the main result, Theorem \ref{xdv_for_all_rho_if_all_horo_omni},  in Section \ref{Representations}.

\section{The tree associated to $\C((\zeta))$}\label{tree}
The following definition is the same as Definition 1.3 in Ohtsuki \cite{ohtsuki97}. Ohtsuki refers to it as ``Serre's tree'', although Serre \cite{serre_trees} attributes it to Bruhat and Tits. It is a special case of the Bruhat-Tits building \cite{bruhat-tits_building}:

\begin{defn}
For an indeterminate variable $\zeta$ we define the tree $T_\zeta$ with the following four conditions:
\begin{enumerate}
\item $T_\zeta$ has a special vertex $H$; we call it the {\bf home vertex}.
\item We call an infinitely long path from the home vertex an {\bf end}. The set of ends of $T_\zeta$ is identified with the set $\C((\zeta)) \cup \{ \infty \}$, to which we give the discrete topology.
\item The part of $T_\zeta$ from $H$ to the set of ends $\C[[\zeta]]$ (power series in $\zeta$) is identified with the series
$$ \{H\} \leftarrow \C[\zeta]/(\zeta) \leftarrow \C[\zeta]/(\zeta^2) \leftarrow \C[\zeta]/(\zeta^3) \leftarrow \cdots \leftarrow \C[[\zeta]]$$ where the maps are the natural projections. Here we identify the set of vertices at distance $n$ from $H$ with the set $\C[\zeta]/(\zeta^n)$, the ring of polynomials of degree at most $n-1$. Two vertices are connected by an edge if one of the maps takes one vertex to the other.
\item The part of $T_\zeta$ from $H$ to the set of ends $\C((\zeta)) \cup \{ \infty \} \setminus \C[[\zeta]]$ is homeomorphic to $(\zeta) \subset \C[[\zeta]]$ by the map $\C((\zeta)) \cup \{ \infty \} \rightarrow \C((\zeta)) \cup \{ \infty \}$, which takes $x$ to $x^{-1}$.
\end{enumerate}
\end{defn}

See Figure \ref{Serretree.pdf} for a very small part of the tree, showing how paths from the ends of the tree to $H$ meet each other.

\begin{figure}[htb]
\centering
\includegraphics[width=1.0\textwidth]{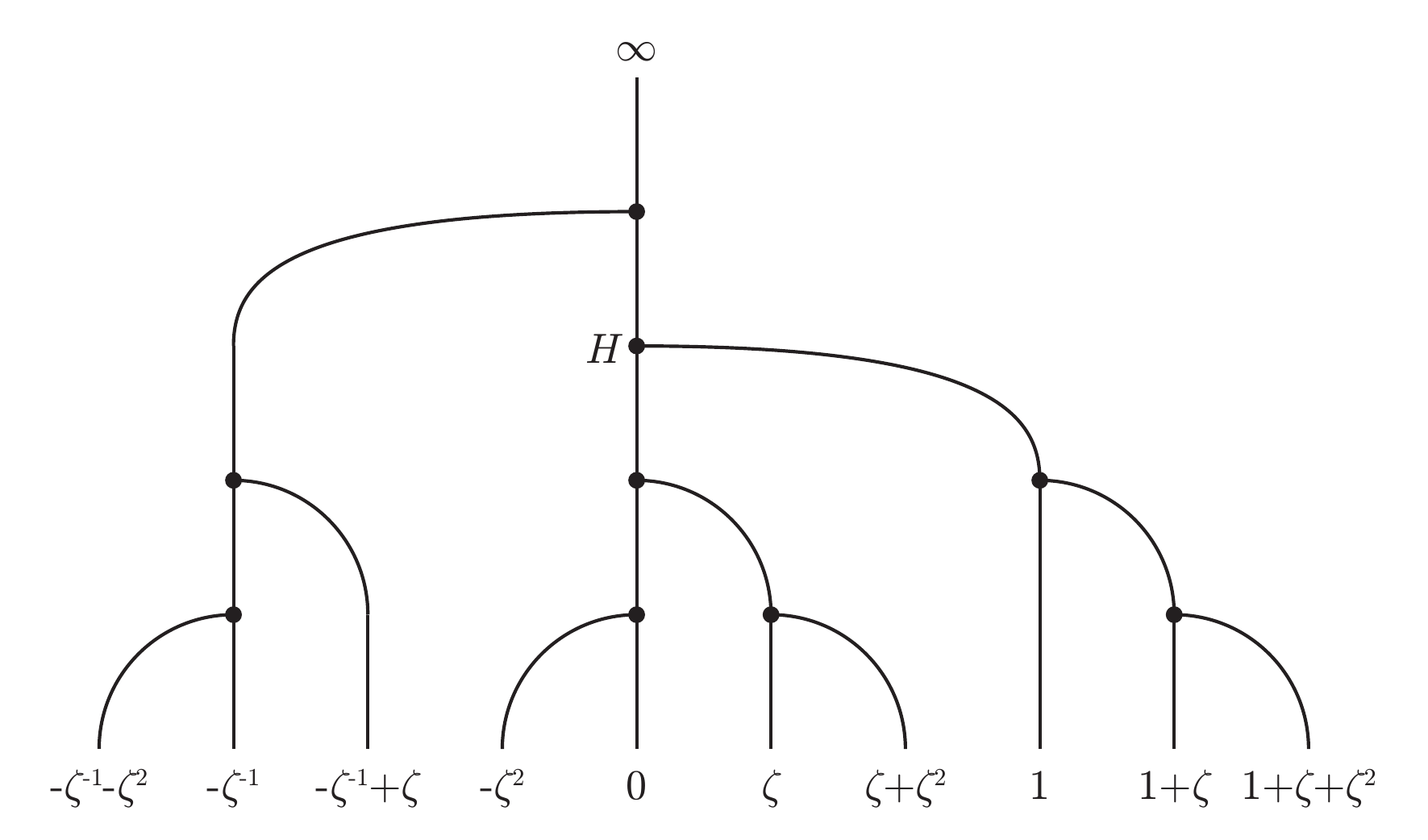}
\caption{A very small subset of $T_\zeta$, with the home vertex $H$ marked. In fact each vertex has a $\CP^1$'s worth of neighbours.}
\label{Serretree.pdf}
\end{figure}

\begin{defn}
For $x \in \C((\zeta)) \setminus \{0\}$ with coefficients $x_k$, we define 
$$\ord(x) := \min\{k \in \Z | x_k \neq 0\}$$
$$x_* := x_{\ord(x)} \in \C\setminus\{0\}$$
We set $\ord(0) := \infty$ and $\ord(\infty) := -\infty$.
\end{defn}
If $n = \ord(x)$ then $x = x_n\zeta^n+x_{n+1}\zeta^{n+1} + \cdots$\\

Given four distinct ends $a,b,c$ and $d$ of $T_\zeta$ we define the cross ratio 
\begin{equation} \label{cross_ratio_eq}
z = \frac{(a-c)(b-d)}{(a-d)(b-c)} \in \C((\zeta))\setminus \{0,1\}
\end{equation}
Just as in $\C\cup\{\infty\}$, if $a=\infty$, for example, we interpret this as $z = \frac{(b-d)}{(b-c)}$, using the rules $\infty + x = \infty$ for $x \in \C((\zeta))$ and $\infty/\infty = 1$.
There are six possible cross ratios to take, but only three if we preserve orientation. They are related to each other as $z, \frac{z-1}{z}, \frac{1}{1-z}$. At least one of these three is a {\bf preferred cross ratio}, meaning that $z \in \C[[\zeta]]$ and $z_0 \neq 1$. If $z \in \C((\zeta))\setminus\C[[\zeta]]$ then $\frac{1}{1-z}\in \C[[\zeta]]$ and $(\frac{1}{1-z})_0 = 0$, and if $z_0 = 1$ then $(\frac{z-1}{z})_0 = 0$. If $z\in\C[[\zeta]], z_0\notin\{0,1\}$ then all of the three cross ratios are preferred.  Note that the order of the preferred cross ratio determines the length of the {\bf spine} determined by the four ends, as in Figure \ref{spine_and_tetra.pdf}.\\

\begin{figure}[htb]
\centering
\includegraphics[width=\textwidth]{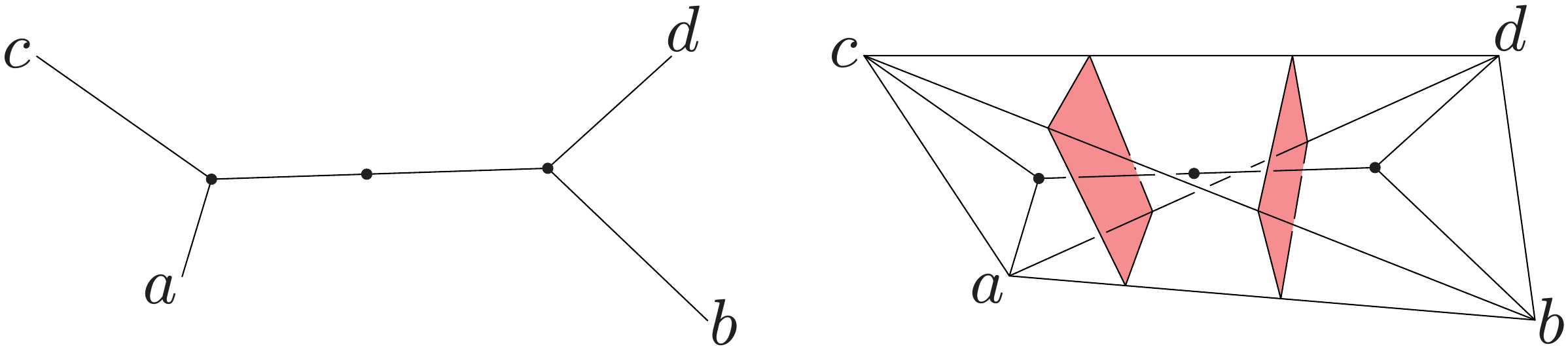}
\caption{On the left, a spine in $T_\zeta$ determined by four ends $a,b,c,d$. The order of the preferred cross ratio is the same as the length of the spine (given by the number of edges in the midsection), which is 2 in this example. On the right, a spine isometric to the first viewed as part of the dual tree to a normal surface within a tetrahedron, this time $a,b,c,d$ are vertices of the tetrahedron.}
\label{spine_and_tetra.pdf}
\end{figure}

Given three distinct ends $a,b,c \in \C((\zeta))\cup \{\infty\}$ and cross ratio $z$, solving for the fourth gives 
\begin{equation}\label{develop_eq}
d = \frac{(b-c)za-(a-c)b}{(b-c)z - (a-c)}.
\end{equation}
We deal with $\infty$ in this equation in the same way we did for the cross ratio.\\

For each pair of triangles which share an edge, we may assign a {\bf dihedral angle} between them, which is a cross ratio in $\C((\zeta))\setminus \{0,1\}$. Given specified ends $e, e', e'' \in\C((\zeta))\cup \{\infty\}$ for the locations of the vertices of one triangle, this tells us the location of the fourth vertex using equation (\ref{develop_eq}). See Figure \ref{develop_dihedral_angle.pdf}. We call this process of determining the position of a vertex from a dihedral angle and the positions of three vertices {\bf developing}. The cross ratio assigned to the pair of triangles that share an edge is the one (of the three possibilities) so that the labelling on the left hand diagram of Figure \ref{develop_dihedral_angle.pdf} matches equation (\ref{cross_ratio_eq})\footnote{The reason we call this choice of cross ratio the dihedral angle is that if we then set $a=\infty, b=0, c = 1$ and put the picture in the upper half space model of $\Hthree$, then with appropriate interpretation of $\infty$ in equation (\ref{cross_ratio_eq}) we get that $z=d$, and so the usual meaning of angle comparing $c-b = 1 - 0$ with $d-b = d-0$ is the argument of $z$.}. The right hand diagram of Figure \ref{develop_dihedral_angle.pdf} and the independence of the cross ratio under swapping $a$ with $b$ and $c$ with $d$ shows us that it doesn't matter which way up the picture is when we choose which cross ratio is the dihedral angle.\\

\begin{defn}\label{chain_of_triangles}
A {\bf chain of triangles} in $\mathcal{T}$ (or $\til{\mathcal{T}}$) is a sequence of triangular faces $\tri^{(1)}, \tri^{(2)},\ldots,\tri^{(n)} \in \mathcal{T}$ (or $\til{\mathcal{T}}$) such that neighbouring triangles are two faces of a tetrahedron of $\mathcal{T}$ (or $\til{\mathcal{T}}$) (and therefore share an edge).
\end{defn}

By induction we can develop positions of vertices for chains of triangles with dihedral angles between neighbours.\\

\begin{figure}[htb]
\centering
\includegraphics[width=0.6\textwidth]{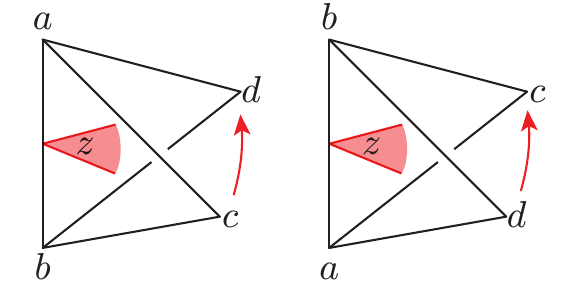}
\caption{Developing across a dihedral angle between two neighbouring triangles. }
\label{develop_dihedral_angle.pdf}
\end{figure}

\begin{rmk}
For any field $\mathbb{F}$, $\text{PSL}_2(\mathbb{F})$ acts freely and transitively on $\mathbb{PF}^1$ and preserves cross ratios, and so this is true for the field $\C((\zeta))$. Therefore choosing different initial points $e, e', e''$ for the first triangle moves the set of developed vertex positions consistently. In particular, for any construction in this paper there will be only countably many developed vertex positions but a $\mathbb{CP}^1$'s worth of vertices next to $H$ so we may choose $e, e', e''$ so that \emph{all} vertex positions are in $\C[[\zeta]]$, avoiding the one edge leading away from $H$ in the direction of $\infty$. 
\end{rmk}

\begin{lemma}\label{*det_cr}
If $z$ is the preferred cross ratio of $a,b,c,d \in \C[[\zeta]]$  then $$z_*=\frac{(a-c)_*(b-d)_*}{(a-d)_*(b-c)_*}$$
\end{lemma}
\begin{proof}
Let $(a-c) = p, (b-d)=q,(a-d)=r, (b-c)=s$, all in $\C[[\zeta]]$. Then
$$z = \frac{pq}{rs} = \zeta^k \frac{(p_*+\zeta p')(q_*+\zeta q')}{(r_*+\zeta r')(s_*+\zeta s')} \;\;\;\;p',q',r',s'\in \C[[\zeta]], k\in \N$$
where $k\geq0$ since this is the preferred cross ratio. Then
$$z = \zeta^k\frac{p_*q_* + \zeta e}{r_*s_* + \zeta f}=\left(\frac{\zeta^k}{r_*s_*}\right)\frac{p_*q_* + \zeta e}{1-\zeta(\frac{-f}{r_*s_*})} \;\;\;\;e,f\in \C[[\zeta]]$$
so
$$z = \left(\frac{\zeta^k}{r_*s_*}\right)(p_*q_* + \zeta e)(1 + \zeta g) = \zeta^k \left(\frac{p_*q_*}{r_*s_*} + \zeta h\right) \;\;\;\;g,h\in\C[[\zeta]]$$
hence $$z_* = \frac{p_*q_*}{r_*s_*}$$
\end{proof}
This provides motivation that under good conditions we should be able to ignore all higher order information about cross ratios and developed positions if we only care about lowest order information about those objects. This is the subject of the next few lemmas.
\begin{lemma}\label{*det}
Suppose that $a,b,c \in \C[[\zeta]]$, that at most one of $\ord(a-b), \ord(a-c), \ord(b-c)$ is greater than 0, $z$ is the preferred cross ratio for the four ends $a,b,c,d$ and that $d \in \C[[\zeta]]$. Then $(d-a)_*, (d-b)_*$ and $(d-c)_*$ are determined by  $(a-b)_*, (a-c)_*, (b-c)_*$, $z_*$ and the orders of those terms.
\end{lemma}
\begin{proof}

Note the following relations:
\begin{eqnarray}
(d-a) &=& \left(\frac{1}{(b-c)z - (a-c)}\right)(a-b)(a-c)\\
(d-b) &=& \left(\frac{1}{(b-c)z - (a-c)}\right)(b-c)(a-b)z\\
(d-c) &=& \left(\frac{1}{(b-c)z - (a-c)}\right)(a-c)(b-c)(z-1)
\end{eqnarray} 
Let $q = (b-c)z - (a-c)$. If $\ord(q) = n \geq 0$ then $q_* = q_n \neq 0$ and 
$$ \frac{1}{q} = \frac{1}{\zeta^n q_n + \zeta^{n+1} q'} = \frac{1}{\zeta^n q_n} \frac{1}{1+\zeta(\frac{q'}{q_n})} = \frac{1}{\zeta^n q_n}\sum_{m=0}^\infty \left(-\frac{q'}{q_n}\right)^m\zeta^m$$
where $q' \in \C[[\zeta]]$. Thus $(\frac{1}{q})_* = \frac{1}{q_n} = \frac{1}{q_*}$. \\

First, $z$ is the preferred cross ratio, so $\ord(z) \geq 0$. 
$\ord(q) \geq \min\{\ord(b-c)+\ord(z), \ord(a-c)\}$, with equality unless $\ord(b-c) + \ord(z) = \ord(a-c)$ and $(b-c)_*z_* - (a-c)_* = 0$.  We rule this out by considering the possible cases for this given that at most one of $\ord(a-b), \ord(a-c), \ord(b-c)$ is greater than 0:
\begin{enumerate}
\item $\ord(b-c) > 0$ and so $\ord(a-c)=0$. If this is true, then since $\ord(z) \geq 0$, then  $((b-c)z - (a-c))_0 = -(a-c)_0 \neq 0$, so this case is impossible.
\item $\ord(a-b) > 0$ and so $\ord(b-c)=\ord(a-c)=0$. $\ord(b-c) + \ord(z) = \ord(a-c)$ so $\ord(z)=0$. Then we have that $(b-c)_0z_0 - (a-c)_0 = 0$, so $(b_0-c_0)z_0 - (a_0-c_0) = 0$. Since $\ord(a-b) > 0, a_0 = b_0$ and therefore $z_0=1$, which is ruled out by the choice of preferred cross ratio.
\item $\ord(a-c) = n \geq 0,\ord(a-b)=\ord(b-c)=0$. $\ord(b-c) + \ord(z) = \ord(a-c)$ so $\ord(z) = n$.  Then the numerator of $d$ in equation (\ref{develop_eq}) is
\begin{eqnarray*}
(b-c)z a - (a-c) &=& (b-c)_0 z_n \zeta^n a - (a-c)_n\zeta^n b + \text{h.o.t.}\\
                 &=& (a-c)_n\zeta^n(a-b) + \text{h.o.t.}\\
                 &=& (a-c)_n(a-b)_0\zeta^n + \text{h.o.t.}
\end{eqnarray*}
Here $\text{h.o.t.}$ stands for ``higher order terms''. Since $(a-c)_n(a-b)_0 \neq 0$, the lowest power of $\zeta$ in the numerator is less than that in the denominator, which is strictly greater than $n$ because of the cancellation, and so $d \notin \C[[\zeta]]$, contradicting the hypothesis.
\end{enumerate}
So we have equality: $\ord(q) = \min\{\ord(b-c)+\ord(z), \ord(a-c)\}$, and $q_*$ is one of $(b-c)_*z_*, (a-c)_*$ or $(b-c)_*z_* - (a-c)_*$, none of which is zero. We then have that
\begin{eqnarray}
(d-a)_* &=& \left(\frac{1}{q_*}\right)(a-b)_*(a-c)_*\\
(d-b)_* &=& \left(\frac{1}{q_*}\right)(b-c)_*(a-b)_*z_*\\
(d-c)_* &=& \left(\frac{1}{q_*}\right)(a-c)_*(b-c)_*(z-1)_*
\end{eqnarray}
For $(d-c)_*$, since $z_0 \neq 1$, $(z-1)_* = (z-1)_0 = z_0 -1$, which is either $z_*-1$ or just $-1$ depending on $\ord(z)$. In all cases therefore, $(d-a)_*, (d-b)_*$ and $(d-c)_*$ and their orders are determined by  $(a-b)_*, (a-c)_*, (b-c)_*$, $z_*$ and their orders.
\end{proof}
Similar analysis shows that case 3 in the lemma is the only case in which $d$ can fail to be in $\C[[\zeta]]$.
\begin{defn}
For each triple $\{a,b,c\}$ of distinct ends of $T_\zeta$ we define the {\bf tripod of $\{a,b,c\}$} to be the smallest connected subtree of $T_\zeta$ that contains an infinite tail of each end (recall that an end is an infinitely long path from $H$). The unique trivalent vertex of the tripod is the {\bf center} of the tripod.
\end{defn}
\begin{defn}
A triple $\{a,b,c\}$ of distinct ends of $T_\zeta$ is {\bf domestic} if the tripod of $\{a,b,c\}$ contains $H$ and {\bf foreign} otherwise.
\end{defn}
We can thus rephrase the condition in Lemma \ref{*det} on the orders of $(a-b), (a-c)$ and $(b-c)$ as the triple $\{a,b,c\}$ being domestic.
\begin{lemma}\label{*det2}
 Let $\tri^{(1)}, \tri^{(2)},\ldots,\tri^{(n)}$ be a chain of triangles with specified dihedral angles $w^{(i)}$ between $\tri^{(i)}$ and $\tri^{(i+1)}$. Suppose initial points $e, e', e''$ for the locations of the 3 vertices of $\tri^{(1)}$ are chosen so that $\ord(e-e') = \ord(e'-e'') = \ord(e''-e) = 0$,  all developed vertices are in $\C[[\zeta]]$ and suppose that each triple of developed vertices given by those triangular faces is domestic. Then for every developed vertex $f \in \C[[\zeta]]$ corresponding to a vertex of one of the $\tri^{(i)}$, $f_0$ depends only on $e_0, e'_0, e''_0$ and the $z_*^{(i)},\ord(z^{(i)})$ where $z^{(i)}$ is the preferred cross ratio corresponding to the dihedral angle $w^{(i)}$.
\end{lemma}
\begin{proof}
This is essentially an induction using Lemma \ref{*det}. Each triangular face with vertex positions $\{a,b,c\}\subset \C[[\zeta]]$ we develop from is domestic. This implies that at most one of the orders $\ord(a-b), \ord(a-c), \ord(b-c)$ is greater than 0. If $a$ and $b$ are positions for vertices at opposite ends of an edge of a $\tri^{(i)}$, then by induction, $(a-b)_*$ and $\ord(a-b)$ are determined by $(e-e')_*$, $(e'-e'')_*$, $(e''-e)_*$, the $z_*^{(i)}$ and their orders. All of those orders are 0 by our choice of $e,e',e''$, and so $(e-e')_* = (e-e')_0 = e_0 - e'_0$ and similarly for the others, so in fact $(a-b)_*$ and $\ord(a-b)$ are determined by $e_0, e'_0, e''_0$, the $z_*^{(i)}$ and $\ord(z^{(i)})$.\\

Note that $(a-b)_0$ is either 0 or $(a-b)_*$ depending on if $\ord(a-b) > 0$ or not, so $(a-b)_0 = a_0 - b_0$ is also determined by the same data. Now consider a path of edges walking along a sequence of vertex positions $f^{(1)}, f^{(2)}, \ldots, f^{(m)}$ where $f^{(1)} \in \{e,e',e''\}$ and $f^{(m)} = f$. Then for each neighbouring pair $(f^{(j)}, f^{(j+1)})$,  $f^{(j)}_0 - f^{(j+1)}_0$ is determined by the data, we start from one of $e_0, e'_0, e''_0$ and so again by induction $f^{(m)}_0 = f_0$ is also determined by the data.
\end{proof}

Now suppose we have an ideal triangulation $\mathcal{T}$ of an orientable 3-manifold with boundary $M$ and an embedded surface $S \subset M$ in (spun-)normal form relative to $\mathcal{T}$. We lift $\mathcal{T}$ to $\til{\mathcal{T}}$, a triangulation of $\til{M}$, the universal cover of $M$ and lift $S$ to $\til{S} \subset \til{M}$ a surface in (spun-)normal form relative to $\til{\mathcal{T}}$. We allow the possibility that $S$ is boundary parallel. Dual to $\til{S}$ we have a tree denoted $T_{\til{S}}$, which we can view as made by gluing together spines, with the tree dual to the intersection of $\til{S}$ with a tetrahedron of $\til{\mathcal{T}}$ being such a spine, as in Figure \ref{spine_and_tetra.pdf}. Vertices of $T_{\til{S}}$ correspond to connected components of $\til{M}\setminus \til{S}$.\\

We will assign dihedral angles (elements of $\C((\zeta))\setminus\{0,1\}$) to the six edges between pairs of triangles in each tetrahedron of $\mathcal{T}$. Given this information and any chain of triangles in $\mathcal{T}$ together with the positions of the vertices of the first triangle we can develop the positions of all vertices along the chain. 

\begin{defn}
For $f \in \C((\zeta)) \cup \{\infty\}$, the {\bf direction of $f$ from $H$} is
$$ f_H := \left\{
\begin{array}{ccl}
f_0 & \text{ if } & f \in \C[[\zeta]] \\
\infty & \text{ if } & f \in (\C((\zeta)) \cup \{\infty\}) \setminus \C[[\zeta]]
\end{array} \right.
$$
\end{defn}

\begin{cond}\label{degen_order}
The {\bf degeneration order condition} on dihedral angles assigned to a tetrahedron $t$ states that if $t$ contains no quadrilateral of $S$ then all dihedral angles $w$ must have $w_H \neq 0,1,\infty$ (and hence $\ord(w) = 0$), and if $t$ has $k$ quadrilaterals then $w_H$ is as in Figure \ref{degen_dihedral_angles.pdf}, and the order of the preferred cross ratio corresponding to $w$ is $k$.
\end{cond}

This is consistent with the connection between ideal points of the deformation variety and spun-normal surfaces as in \cite{tillmann_degenerations}. This also means that for a tetrahedron $t \in \til{\mathcal{T}}$, the spine dual to $S\cap t$ is the same shape as the corresponding spine in $T_\zeta$ we get when developing through any pair of triangles of $t$. Spines corresponding to neighbouring tetrahedra glue to each other in the same way in both contexts.

\begin{figure}[htb]
\centering
\includegraphics[width=0.3\textwidth]{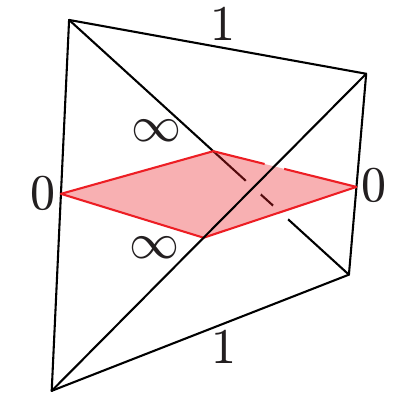}
\caption{The $w_H$ for dihedral angles $w$ in a tetrahedron with a quadrilateral. }
\label{degen_dihedral_angles.pdf}
\end{figure}

\begin{lemma}\label{region_domestic}
Suppose we assign dihedral angles to each tetrahedron of $\mathcal{T}$ that satisfy the degeneration order condition with respect to the surface $S$. Let $\tri^{(1)}, \tri^{(2)},\ldots,\tri^{(n)}$ be a chain of triangles in $\mathcal{\til{T}}$ and $R_0$ be the component of $\til{M} \setminus \til{S}$ that contains the central region of $\tri^{(1)}$. Suppose we can trace a path following $R_0$ continuously through the $\tri^{(i)}$ and we develop positions of the vertices starting from $e, e', e''$ for the locations of the 3 vertices of $\tri^{(1)}$ as in Lemma \ref{*det2}, and all developed positions are in $\C[[\zeta]]$. Then the positions for the triple of vertices for each $\tri^{(i)}$ is domestic.
\end{lemma}
\begin{proof}
The triple of images of the vertices of $\tri^{(1)}$ is domestic by assumption. When we develop to $\tri^{(2)}$, the image of the new vertex ($f$ say) is arranged with respect to $e,e',e''$ in $\C[[\zeta]]$ in the same arrangement as the corresponding vertices in $T_{\til{S}}$, by the degeneration order condition (\ref{degen_order}). The tripod in $T_{\til{S}}$ corresponding to $\tri^{(1)}$ has center the vertex dual to $R_0$ and the tripod corresponding to $\tri^{(2)}$ together with the tripod corresponding to $\tri^{(1)}$ form a spine as in the right diagram of Figure \ref{spine_and_tetra.pdf}, as a subtree of $T_{\til{S}}$. See also Figure \ref{two_tris_and_spine.pdf}.\\
 
As we continue to develop to further triangles, the condition that $R_0$ has non-empty intersection with each $\tri^{(i)}$ corresponds to the condition that each new tripod in $T_{\til{S}}$ contains the central vertex of the tripod for $\tri^{(1)}$. The tripods we add in $\C[[\zeta]]$ satisfy the corresponding condition, which is that they contain the home vertex, and so are all domestic.  
\end{proof}

\begin{figure}[htb]
\centering
\includegraphics[width=0.6\textwidth]{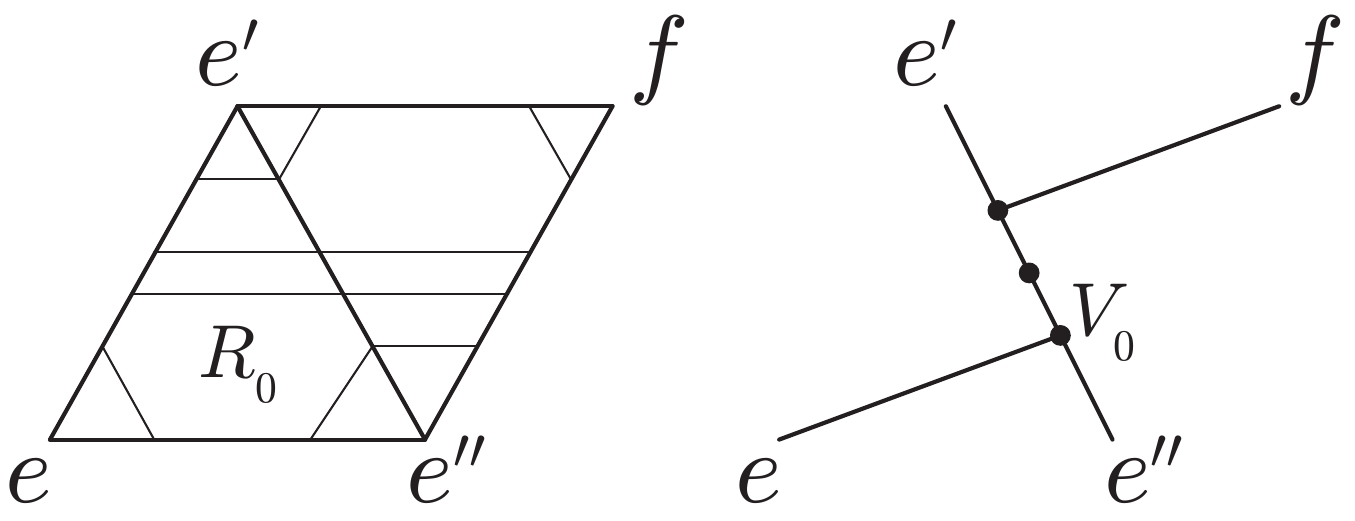}
\caption{Two neighbouring triangles and their intersections with $\til{S}$. The corresponding spine in $T_{\til{S}}$. The central region of the first triangle is marked $R_0$ and the corresponding vertex of the spine is marked $V_0$.}
\label{two_tris_and_spine.pdf}
\end{figure}

\begin{thm}\label{*det3}
Suppose we assign dihedral angles $w$ to each tetrahedron of $\mathcal{T}$ that satisfy the degeneration order condition with respect to the surface $S$. Let $\tri^{(1)}, \tri^{(2)},\ldots,\tri^{(n)}$ be a chain of triangles in $\mathcal{\til{T}}$ (so neighbouring triangles are both part of a single tetrahedron) and $R_0$ be the component of $\til{M} \setminus \til{S}$ that contains the central region of $\tri^{(1)}$. Suppose we can trace a path following $R_0$ continuously through the $\tri^{(i)}$ and we develop positions of the vertices starting from $e, e', e''$ for the locations of the three vertices of $\tri^{(1)}$ where the tripod of $\{e,e',e''\}$ has center $H$. Then for every developed position $f$ corresponding to a vertex of one of the $\tri^{(i)}$, $f_H$ depends only on $e_H, e'_H, e''_H$ and the $z_*,\ord(z)$ where $z$ is the preferred cross ratio for the dihedral angle $w$. 

\end{thm}
\begin{proof}
This is a combination of Lemmas \ref{*det2}, \ref{region_domestic} and the observation that we can drop the condition that we develop only into $\C[[\zeta]]$ by noting that the result is true when we do restrict to $\C[[\zeta]]$, and we can always conjugate the developed vertex positions so that they are all in $\C[[\zeta]]$ without altering the cross ratios, and hence the developing map. 
\end{proof}
We will use the same trick again later, assuming without loss of generality that we develop only into $\C[[\zeta]]$ but then dropping this requirement for the result. 
\begin{rmk}\label{how_to_devel}
If we identify $\C[\zeta]/(\zeta) \cup \{\infty\}$ with $\bdry \Hthree$, then this gives us a way to define a developing map from chains of triangles that contiguously contain $R_0$ into $\Hthree$. Some of the triangles may be degenerate, having two vertices at the same location on $\bdry \Hthree$, although due to the construction we never have any triangles with all three vertices at the same location (this would correspond to a foreign triple of vertices).
\end{rmk}

We have yet to put any conditions on the assigned dihedral angles other than the degeneration order condition. In the construction following we will add some more conditions, in particular appropriate versions of the relations from equations (\ref{dv_eq1}) and (\ref{dv_eq2}), although in some cases only for certain tetrahedra. We will also require a condition playing the role of the gluing equations.

\section{The extended deformation variety}\label{The extended deformation variety}

Let $\mathcal{E}$ and $\til{\mathcal{E}}$ be the edge sets of $\mathcal{T}$ and $\til{\mathcal{T}}$. We consider disjoint subsets $\mathcal{E}_0, \mathcal{E}_+ \subset \mathcal{E}$ whose union is $\mathcal{E}$, and define $\til{\mathcal{E}}_0, \til{\mathcal{E}}_+$ as their preimages in $\til{\mathcal{T}}$. The idea is that $\til{\mathcal{E}}_0$ will be the set of developed edges of zero length in $\Hthree$, and $\til{\mathcal{E}}_+$ the set of developed edges of positive (i.e. non-zero) length. We will refer to edges in $\til{\mathcal{E}}_0$ and $\til{\mathcal{E}}_+$ as \textbf{zero-length} and \textbf{positive-length} edges respectively. \begin{defn}\label{bad_loop}
Given a subset of edges $\mathcal{E}_0 \subset \mathcal{E}$ we say that a loop of edges in $\til{\mathcal{E}}$ is a {\bf bad loop} if exactly one edge is in $\til{\mathcal{E}}_+$ and all others are in $\til{\mathcal{E}}_0$.  A subset of edges $\mathcal{E}_0 \subset \mathcal{E}$ has {\bf no bad loops} if there is no such loop.
\end{defn}
We will only allow subsets $\mathcal{E}_0$ with no bad loops. In particular, if two of the three edges of a face of $\mathcal{T}$ are in $\mathcal{E}_0$ then the third must also be to avoid a bad loop.

\begin{defn}
We refer to the three possibilities for the arrangement of edges of $\mathcal{E}_0$ around a triangle of $\mathcal{T}$ or  $\til{\mathcal{T}}$ as  {\bf types} 111 (no zero-length edges), 21 (one zero-length edge), or 3 (all zero-length edges). 
\end{defn}
\begin{defn}
There are five possibilities for the arrangement of edges in $\mathcal{E}_0$ around each tetrahedron of $\mathcal{T}$ or  $\til{\mathcal{T}}$, as in Figure \ref{5_tetra_types_w_graphs.pdf}. We refer to these five possibilities as {\bf types} $1111$, $211$, $22$, $31$ and $4$ respectively. 
\end{defn}
The naming convention in these two definitions is derived from the size and number of equivalence classes of vertices within the tetrahedron, where two vertices are equivalent if they are connected by a zero-length edge. See also Figure \ref{4_tetra_and_regions.pdf}.

\begin{figure}[htb]
\centering
\includegraphics[width=\textwidth]{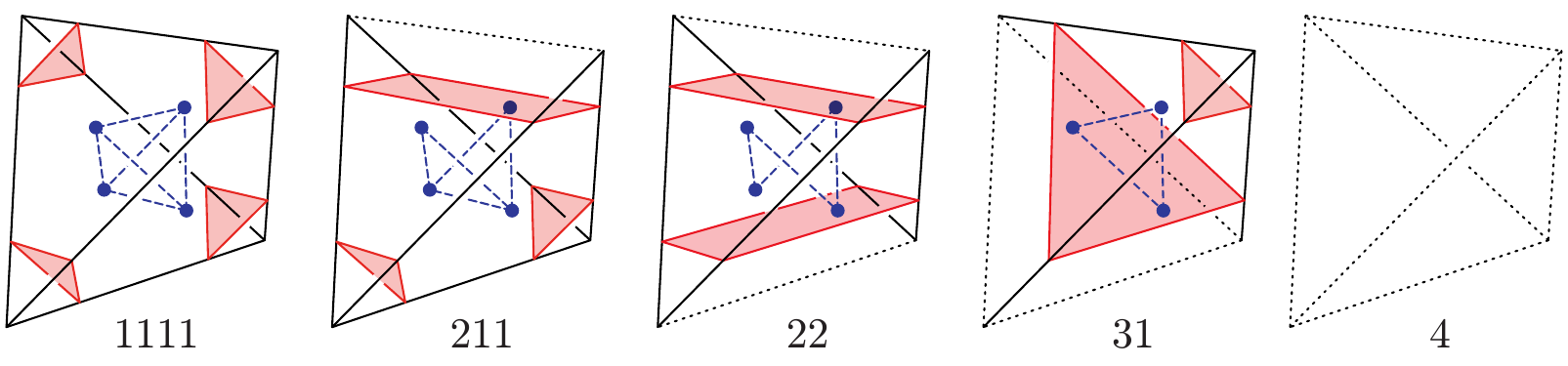}
\caption{The 5 possible arrangements of edges in $\mathcal{E}_0$ (the dotted tetrahedron edges) around a tetrahedron with corresponding parts of normal surfaces. The barycenters of faces of the tetrahedra are shown with large dots, and a valid chain of triangles corresponds to a path of these dots along the graph formed by the dashed lines.}
\label{5_tetra_types_w_graphs.pdf}
\end{figure}

\begin{figure}[htb]
\centering 
\includegraphics[width=\textwidth]{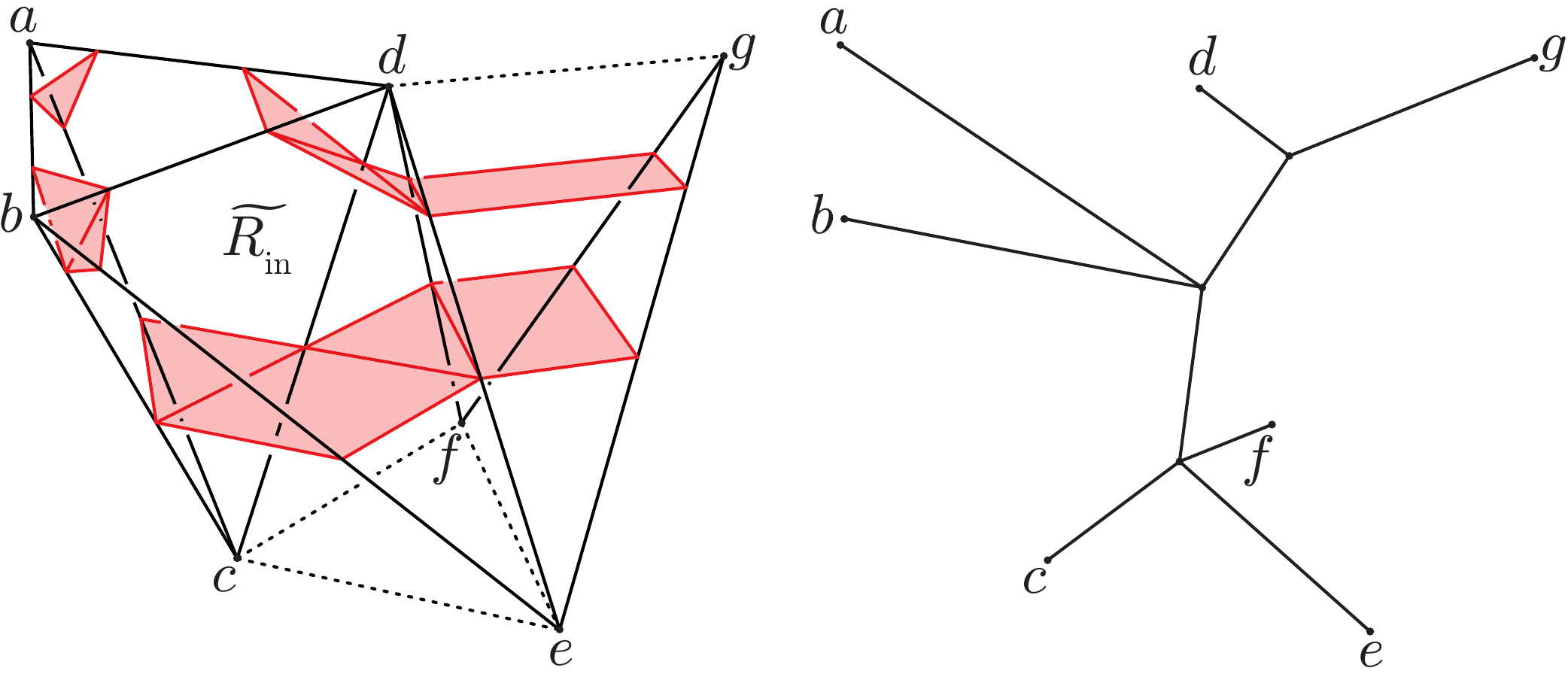}
\caption{From left to right, tetrahedra in $\til{\mathcal{T}}$ of type $1111$, $211$, $31$, $22$. The dual tree.}
\label{4_tetra_and_regions.pdf}
\end{figure}

\begin{defn}\label{E_0_to_horo-normal}
Given a subset $\mathcal{E}_0 \subsetneq \mathcal{E}$ with no bad loops, we construct a corresponding normal surface as follows: each tetrahedron $t\in\mathcal{T}$ has one of the five arrangements of edges in $\mathcal{E}_0$ as in Figure \ref{5_tetra_types_w_graphs.pdf}. We also place normal quadrilaterals and triangles in these tetrahedra as in the figure. Each face of $\mathcal{T}$ is of type 111, 21 or 3. These faces have either three normal arcs cutting off each vertex,  two normal arcs parallel to the edge in in $\mathcal{E}_0$, or no normal arcs, respectively. Therefore these normal quadrilaterals and triangles match up across the faces of $\mathcal{T}$ to form a normal surface.\\

We do not allow the empty surface, corresponding to $\mathcal{E}_0=\mathcal{E}$. If $S$ is a (non-empty) normal surface corresponding to a subset of edges with no bad loops, then we say $S$ is in {\bf horo-normal form\footnote{As mentioned in the introduction, the surface is similar to a horosphere, cutting the vertices that are all in the same place on $\bdry\Hthree$ away from all of the other vertices in other locations.}} (relative to $\mathcal{T}$), and that $S$ is a {\bf horo-normal surface}.
\end{defn}
\begin{rmk} We can construct $\mathcal{E}_0$ from a horo-normal surface by reading off from Figure \ref{5_tetra_types_w_graphs.pdf}, and so horo-normal surfaces and subsets $\mathcal{E}_0$ with no bad loops are in one to one correspondence.\end{rmk}
\begin{rmk}\label{horo is nbd of E0} An alternate way to see the horo-normal surface corresponding to a subset $\mathcal{E}_0$ with no bad loops is as the boundary of a small regular neighbourhood of the subcomplex of the triangulation generated by $\mathcal{E}_0$.

\end{rmk}
\begin{lemma}\label{horo-normal_properties}
A horo-normal surface $S$ is orientable, closed, and cuts $M$ into a compact inside region $R_\text{in}$ and an outside region  $R_\text{out}$ which contains $\bdry{M}$.
\end{lemma}
\begin{proof}
The quadrilateral and triangle parts of $S$ are prescribed and finite in each of the finitely many tetrahedra of $\mathcal{T}$, and so $S$ is made from finitely many parts and is therefore finite and closed. $S$ is orientable because there is a consistent choice of the side of the parts of $S$ which face the vertices of the tetrahedra. It follows that $S$ cuts $M$ into the two regions (not necessarily connected).
\end{proof}

\begin{defn}
A chain of triangles $\left(\tri^{(0)}, \tri^{(1)}, \ldots, \tri^{(n)}\right)$ in $\mathcal{T}$ (or $\til{\mathcal{T}}$) is {\bf valid} if $\til{R_{\text{in}}}$ (the lift of $R_\text{in}$, as in Lemma \ref{horo-normal_properties}) intersects $\tri^{(0)}$ as its central region (in other words all edges of $\tri^{(0)}$ are in $\til{\mathcal{E}}_+$), and the chain of triangles contiguously intersects $\til{R_{\text{in}}}$ (in other words, the edge between neighbouring triangles is in $\til{\mathcal{E}}_+$). 
\end{defn}

\begin{defn}\label{xdv_data}
The {\bf data for a point of the extended deformation variety} of $M$ with triangulation $\mathcal{T}$ and horo-normal surface $S$ consists of:
\begin{itemize} 
\item A cross ratio $z \in \C \setminus \{0,1\}$ for each tetrahedron of type $1111$
\item A cross ratio $z\zeta \in \C[[\zeta]]$, $z \in \C \setminus \{0\}$ for each tetrahedron of type $211$
\item A cross ratio $z\zeta^2 \in \C[[\zeta]]$, $z \in \C \setminus \{0\}$ for each tetrahedron of type $22$
\item Two complex angles $z_1,z_2 \in \C \setminus \{0\}$ for each tetrahedron of type $31$
\end{itemize}
This determines dihedral angles for each edge of a tetrahedron that is in $\til{\mathcal{E}}_+$, as in Figure \ref{4_tetra_types_w_dihedral_angles.pdf} (only lowest orders are shown; the dihedral angles for type $211$ for example are $z\zeta, \frac{z\zeta-1}{z\zeta}$ and $\frac{1}{1-z\zeta}$, so the preferred cross ratio is always $z\zeta$). We record no data for tetrahedra of type $4$. 
\end{defn}

\begin{rmk}\label{xdv_data_v2}
What the data of Definition \ref{xdv_data} amounts to is the following:
\begin{enumerate}
\item Choose dihedral angles from $\C((\zeta)) \setminus \{0,1\}$ for each tetrahedron of types $1111$, $211$, and $22$ for each of the six edges, subject to:
\begin{enumerate}  
	\item The degeneration order condition
	\item Opposite dihedral angles being equal
	\item Equations (\ref{dv_eq1}) and (\ref{dv_eq2}) (from Section \ref{The deformation variety}) holding as equations in $\C((\zeta))$
\end{enumerate}
\item Choose dihedral angles from $\C((\zeta)) \setminus \{0,1\}$ for each tetrahedron of type $31$ for the 3 edges not in $\mathcal{E}_0$ subject to:
\begin{enumerate}
	\item The degeneration order condition
	\item Equation (\ref{dv_eq1}) holding as an equation in $\C((\zeta))$
\end{enumerate}
\end{enumerate}
\end{rmk}

\begin{rmk}
The data for a tetrahedron of type $31$ may seem strange at first. The conditions on the dihedral angles for a type $31$ tetrahedron are weaker than for other types of tetrahedron. This is to do with the fact that only the difference in positions between two vertices joined by a zero-length edge matters. The absolute positions of these vertices (other than their direction from $H$) does not. See the proof of Theorem \ref{xdv_for_a_rho} for more details.
\end{rmk}

\begin{figure}[htb]
\centering
\includegraphics[width=\textwidth]{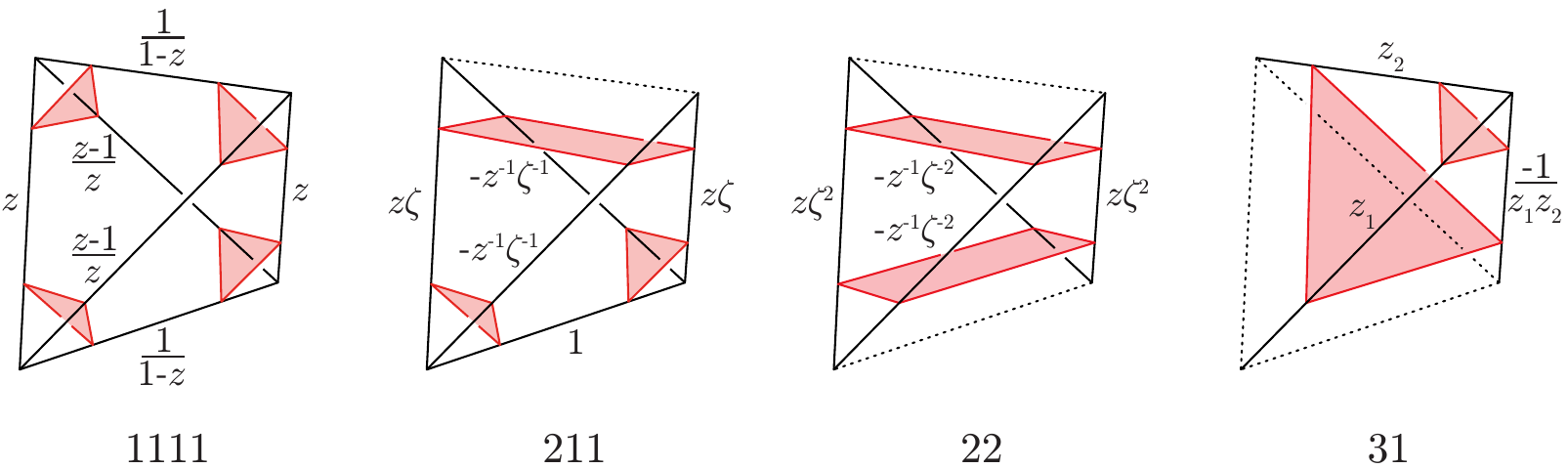}
\caption{Lowest order coefficients for dihedral angles associated to edges of tetrahedra in $\til{\mathcal{E}}_+$ as in Definition \ref{xdv_data}. When we develop we only ever use the preferred cross ratios.}
\label{4_tetra_types_w_dihedral_angles.pdf}
\end{figure}
For types $1111$, $211$, and $22$ the two equations and opposite dihedral angles being equal reduce the choice for each tetrahedron to a single cross ratio, and the degeneration order condition implies the power on the cross ratios as in Definition \ref{xdv_data}. In light of Theorem \ref{*det3} we only need to record the lowest order information for the preferred cross ratio in order to be able to develop positions of vertices along valid chains of triangles. For type $31$ we similarly only record data for dihedral angles that a valid chain of triangles in $\til{\mathcal{T}}$ could turn through. We likewise never need to record data for type $4$, since we will never develop through any of these dihedral angles.\\

The data of Definition \ref{xdv_data} then allows us develop through any valid chain of triangles, starting from three ends of $T_\zeta$ whose tripod center is $H$. We require one further condition on this data:

\begin{cond}\label{consistent_dev}
The {\bf consistent development condition} on the data for a point of the extended deformation variety states that given any two valid chains of triangles that start from the same triangle in $\til{\mathcal{T}}$, if a vertex $v\in\til{\mathcal{V}}$ is a vertex of triangles in both chains then the developed positions $f$ and $f'$ of $v$ under the two chains satisfy $f_H = f'_H$.
\end{cond}
\begin{rmk}
Conditions 1(b), 1(c) and 2(b) in Remark \ref{xdv_data_v2} ensure that valid chains contained within a single tetrahedron satisfy the consistent development condition.
\end{rmk}

\begin{ex}
If $\mathcal{E}_0 = \phi$ then $\mathcal{E}_+ = \mathcal{E}$, the horo-normal surface $S$ is boundary parallel and $\til{R_{\text{in}}}$ is isotopic to $\til{M}$, all tetrahedra are of type $1111$, the consistent development condition is equivalent to the gluing equations holding and we get the standard deformation variety for $M$ with triangulation $\mathcal{T}$. 
\end{ex}

\begin{defn}\label{xdv}
 Let $M$ be a 3-manifold with boundary a disjoint union of tori and with ideal triangulation $\mathcal{T}$ and $S$ a surface in horo-normal form relative to $\mathcal{T}$ such that:
 \begin{enumerate}
 \item $\til{R_\text{in}}$ is connected. 
 \item For each component of $\bdry M$ there exists an $e \in \mathcal{E}_+$ that has at least one endpoint on that component. 
 \item There exists a triangle $\tri \in \til{\mathcal{T}}$ of type 111 (i.e. all three edges are in $\til{\mathcal{E}}_+$, or equivalently $\til{R_{\text{in}}}$ intersects $\tri$ as its central region).
 \end{enumerate}
 Then the {\bf extended deformation variety} of $M$ with triangulation $\mathcal{T}$ and horo-normal surface $S$, $\widehat{\mathfrak{D}}(M;\mathcal{T};S)$ consists of points of the extended deformation variety (as in Definition \ref{xdv_data}) subject to the consistent development condition. We will also consider the disjoint union of all such varieties with a fixed triangulation but ranging over all horo-normal surfaces satisfying these conditions. We call this set the extended deformation variety of $M$ with triangulation $\mathcal{T}$, and write it as $\widehat{\mathfrak{D}}(M;\mathcal{T})$.
\end{defn}


Note that the second condition is automatic for manifolds with only one boundary component and the third holds as long as there are any tetrahedra of type $1111$ or $211$.
\begin{defn}\label{porous}
A surface in horo-normal form relative to $\mathcal{T}$ that satisfies conditions 1 and 2 of Definition \ref{xdv} is called {\bf porous}. 
\end{defn}
A surface is porous if, like a sponge, we can move through $\til{R_\text{in}}$ to get to every cusp of $\til{M}$, without having to go through $\til{R_\text{out}}$.

\begin{defn}\label{dev_map}
Given an element $Z \in \widehat{\mathfrak{D}}(M;\mathcal{T};S)$ we define the {\bf developing map} 
$$\Phi_Z: \til{\mathcal{V}} \rightarrow \bdry \Hthree$$
up to conjugation, determined by the distinct images (positions on $\bdry \Hthree$) we choose for the vertices of $\tri$. (In fact we choose ends of $T_\zeta$ for those positions, but by Theorem \ref{*det3} only the lowest order information of these matters.) To determine the image of any other vertex $v$ of $\til{\mathcal{V}}$, take any valid chain of triangles from $\tri$ to a triangle containing $v$, and develop along that chain. Such a chain exists by the porousity conditions: each $v$ corresponds to a lift of some component of $\bdry M$, let $e \in \mathcal{E}_+$ be as in condition 2 of porousity, then some lift $\til{e}$ of $e$ has $v$ as an endpoint and intersects $\til{R_{\text{in}}}$. Traverse triangles with non-empty intersection with $\til{R_\text{in}}$ from $\tri$ to $\til{e}$ (which we can always do by condition 1) and we get a position for $v$ on $\bdry \Hthree$ by forgetting all but the lowest order information, which is independent of the chain we took by the consistent development condition, Condition \ref{consistent_dev}. 
\end{defn}

\begin{lemma}\label{E0_is_order_1}
Suppose $\til{e} \in \til{\mathcal{E}}$ has endpoints $v,v'$ and we can develop through a valid chain of triangles to a triangle containing $\til{e}$. Then $\til{e} \in \til{\mathcal{E}}_0$ if and only if $\Phi_Z(v) = \Phi_Z(v')$.
\end{lemma}
\begin{proof}
We prove a slightly stronger result: that if $a,a' \in \C[[\zeta]]$ are the positions of $v,v'$, developed along some valid chain of triangles then 
$$\ord(a-a') = \left\{ \begin{array}{cl} 
1 & \text{ if } e \in \mathcal{E}_0\\
0 & \text{ if } e \in \mathcal{E}_+
\end{array}\right.$$
We prove this by induction along chains of triangles starting with $\tri$. By definition, no edges of $\tri$ are in $\mathcal{E}_0$ so we start out developing from three distinct points of $\bdry \Hthree$, by which we mean three ends $e, e', e'' \in T_\zeta$ such that $e_H, e'_H$ $e''_H$ are all different. This shows that the base case holds. The inductive step follows from the degeneration order condition: for every new triangle we develop to, the order of the difference between the developed positions of the endpoints is determined by the orders for the edges in the triangle we develop from and the dihedral angle we develop through. The degeneration order condition ensures that we get the correct order for the new edges.  
\end{proof}

\begin{lemma}\label{verts_thru_S_coincide}
If $v, v' \in \til{\mathcal{V}}$ are in the same component $\left(\til{R_{\text{out}}}\right)_0$ of $\til{R_{\text{out}}}$, the lift of $R_{\text{out}}$ to $\til{M}$, then $\Phi_Z(v) = \Phi_Z(v')$.
\end{lemma}
\begin{proof}
We first develop from $\tri$ out to $v$, and then note that one valid chain of triangles from $\tri$ out to $v'$ is to go via a triangle containing $v$ and then traverse along the lift of $S$ that separates $\left(\til{R_{\text{out}}}\right)_0$ from $\til{R_{\text{in}}}$. This gives us a path of edges on the $\left(\til{R_{\text{out}}}\right)_0$ side, all of which are in $\til{\mathcal{E}}_0$. The vertex at the start of the path is $v$, the vertex at the end is $v'$. Now apply Lemma \ref{E0_is_order_1} at each step of the path.
\end{proof}

\begin{rmk}
If we allowed in Definition \ref{xdv} a choice of $S$ and hence $\mathcal{E}_0$ with a bad loop consisting of all but one edge in $\til{\mathcal{E}}_0$ and one edge in $\til{\mathcal{E}}_+$, then the two endpoints of the path of edges in $\til{\mathcal{E}}_0$, $v$ and $v'$, would necessarily be in the same component of $\til{R_{\text{out}}}$. By the above lemma, their developed positions would be the same on $\bdry \Hthree$, and then the tetrahedra that contain the single edges of $\til{\mathcal{E}}_+$ would not be able to have dihedral angles that satisfied both the degeneration order and consistent development conditions.
\end{rmk}

\begin{rmk}
If there is no $\tri \in \til{\mathcal{T}}$ of type 111 then all tetrahedra are of types $22$, $31$ or $4$. $\til{R_{\text{in}}}$ is then sandwiched between only two components of $\til{R_{\text{out}}}$, and given Lemma \ref{verts_thru_S_coincide}, the only sensible developing map we might define (see footnote \ref{footnote_dev_from_degen} in the proof of Theorem \ref{is_variety} on how we can develop starting from a triangle of type 21) would have all vertices appear at one of only two positions on $\bdry \Hthree$, and there would be no hope of constructing a representation from this data. If all tetrahedra are of type 4 then $\til{R_{\text{out}}} = \til{M}$ and the developing map would have to have all vertices appear at only one position on $\bdry\Hthree$ (this case is ruled out since $\mathcal{E}_0 \neq \mathcal{E}$). See also the proof of Theorem \ref{xdv_for_a_rho}.
\end{rmk}

\begin{thm}\label{is_variety}
$\widehat{\mathfrak{D}}(M;\mathcal{T};S)$ is an affine algebraic variety.
\end{thm}
\begin{proof}
We may view $\widehat{\mathfrak{D}}(M;\mathcal{T};S)$ as a subset of $\C^{N_1 + 2N_2}$, where $N_1$ is the number of tetrahedra of $\mathcal{T}$ of types $1111$, $211$ and $22$, and $N_2$ the number of type $31$, as in Definition \ref{xdv_data}. We now need to show that the consistent development condition \ref{consistent_dev} gives a finite number of polynomial conditions on these variables.\\

First consider all the dihedral angles as elements of $\C((\zeta))\setminus\{0,1\}$ and assume without loss of generality that all developed positions are in $\C[[\zeta]]$. By repeated application of equation (\ref{develop_eq}), every developed position $f$ is a rational function of the preferred cross ratios $z^{(i)} \in \C[[\zeta]]$ corresponding to the dihedral angles, and the initial three points $e,e',e''\in \C[[\zeta]]$, so 
$$f = f(z^{(i)},e,e',e'') = \frac{p(z^{(i)},e,e',e'')}{q(z^{(i)},e,e',e'')}$$
where $p$ and $q$ are polynomials. We can then write $p = \zeta^k(p_0 + \zeta p_1)$ where $p_0$ is a polynomial in the $z^{(i)}_*$ and $e_0,e'_0,e''_0$ (but not $\zeta$), and $p_1$ is a polynomial. We can similarly write $q = q_0 + \zeta q_1$ (where there is no factor of $\zeta$ since $f \in \C[[\zeta]]$). Then 
$$f = \zeta^k\frac{p_0+\zeta p_1}{q_0 + \zeta q_1}=\zeta^k\frac{p_0+\zeta p_1}{q_0(1 - \zeta \frac{q_1}{-q_0})}=\zeta^k\frac{p_0+\zeta p_1}{q_0}\left(1+\zeta \frac{q_1}{-q_0}+\left(\zeta \frac{q_1}{-q_0}\right)^2 + \ldots\right)$$
Thus $f_H = \frac{p_0}{q_0}$, or 0 if $k >0$, and the equation $f_H=f'_H$ can be rearranged into the form of a polynomial in the $z^{(i)}_*$ together with $e_0,e'_0,e''_0$ being equal to 0. Whether or not the two developed points coincide does not depend on the starting positions $e_0,e'_0,e''_0$, since moving those points consists of conjugating, which likewise conjugates all of the developed vertices of triangles in the chains, since cross ratios are preserved. Thus if we fix the values of $e'_0,e''_0$ and the $z^{(i)}_*$ we get a polynomial in only $e_0$, but which is constantly zero and so the polynomial cannot depend on $e_0$ at all. Similarly for $e'_0$ and $e''_0$, and the polynomial equation depends only on the $z^{(i)}_*$. \\

We require that there are a finite number of such polynomial equations needed to ensure consistent development. Again we assume without loss of generality that all developed positions are in $\C[[\zeta]]$. Given a triangle $\tri' \in \til{\mathcal{T}}$ with vertices $u',v',w'$ to which we have developed positions in $\C[[\zeta]]$, then due to Lemma \ref{*det}, $\Phi_Z(u'),\Phi_Z(v'),\Phi_Z(w')$ together with the first order coefficient of the difference between developed positions of a pair of these vertices with the same image under $\Phi_Z$ (if such a pair exists) are all the information we need in order to entirely determine $\Phi_Z(v)$ for further vertices\footnote{This is the sense in which we can develop starting from a triangle of type 21.\label{footnote_dev_from_degen}}. Moreover, the consistent development condition also means that this first order coefficient of the difference is also independent of the chain we take. If it were not, we could continue developing from two chains that produced different answers for the first order coefficient of the difference along some valid chain that connects $\tri'$ to a triangle of type 111. Then the different answers from the two chains would produce different positions of the images under $\Phi_Z$, contradicting consistent development. (If we cannot reach a triangle of type 111 then there was no choice of $\tri$ to start developing from in the first place.)\\

Let $C_1$ and $C_2$ be two valid chains of triangles that start at $\tri$ and end at $\tri'$:
$$C_1 = \left(\tri = \tri^{(0)}_1,\tri^{(1)}_1,\ldots,\tri^{(n_1)}_1 = \tri' \right)$$ 
$$C_2 = \left(\tri = \tri^{(0)}_2,\tri^{(1)}_2,\ldots,\tri^{(n_2)}_2 = \tri' \right)$$
Now suppose $C_1$ and $C_2$ agree at some triangle in the middle, at $\tri'' := \tri^{(j_1)}_1 = \tri^{(j_2)}_2$. We can then split these chains in two, getting:
$$\begin{array}{rcl}
C_{1,A} &=& \left(\tri = \tri^{(0)}_1,\tri^{(1)}_1,\ldots,\tri^{(j_1)}_1 = \tri'' \right)\\ 
C_{2,A} &=& \left(\tri = \tri^{(0)}_2,\tri^{(1)}_2,\ldots,\tri^{(j_2)}_2 = \tri'' \right)\\
C_{1,B} &=& \left(\tri'' = \tri^{(j_1)}_1,\tri^{(j_1+1)}_1,\ldots,\tri^{(n_1)}_1 = \tri' \right)\\ 
C_{2,B} &=& \left(\tri'' = \tri^{(j_2)}_2,\tri^{(j_2+1)}_2,\ldots,\tri^{(n_2)}_2 = \tri' \right)\end{array}$$ 

Then by the above discussion, consistency of development (possibly including consistency of first order coefficient of differences) between $C_1$ and $C_2$ is implied by consistency between $C_{1,A}$ and $C_{2,A}$, and between $C_{1,B}$ and $C_{2,B}$.\\

Consider the graph $\til{G}$ with vertex set the triangles of $\til{\mathcal{T}}$ and edges between two vertices if the corresponding triangles are in the same tetrahedron and the edge between them is in $\til{\mathcal{E}}_+$ (see the large dots and dashed lines in Figure \ref{5_tetra_types_w_graphs.pdf}). Any valid chain of triangles corresponds to a path in $\til{G}$. Showing that any pair of valid chains of triangles develop to the same result is equivalent to showing that any valid chain that is a loop in $\til{G}$ develops back to the starting position of the first triangle (for any choice of first triangle). \\

\begin{figure}[htb]
\centering
\includegraphics[width=0.4\textwidth]{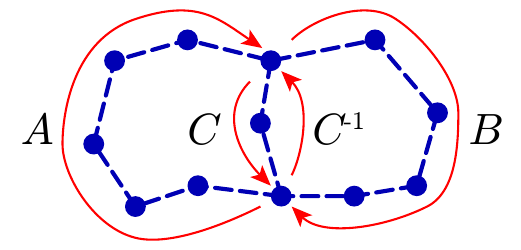}
\caption{Chains of triangles, combining a large loop from two smaller ones.}
\label{add_loops.pdf}
\end{figure}

See Figure \ref{add_loops.pdf}. Given a loop $A\circ B$ (i.e. following the chain $A$ then $B$), $A\circ B$ develops consistently if and only if $A\circ C \circ C^{-1}\circ B$ does (since all of our developing steps are reversible), which by the above discussion is implied by the loops $A\circ C$ and $C^{-1}\circ B$ developing consistently. Thus, all we need is a generating set of loops of triangles to develop consistently in order to ensure that every loop develops consistently. A finite set of such loops can be found, using a set of generating loops for $H_1(G)$, where $G$ is the graph defined analogously to $\til{G}$, but with vertices corresponding to the triangles of $\mathcal{T}$ rather than $\til{\mathcal{T}}$. The consistency of developing around each loop is a polynomial condition in the dihedral angles, and the result follows.
\end{proof}

We should worry that perhaps there are some strange solutions to the consistent development condition equations, for a triangulation for which we cannot define the developing map in the first place. This is analogous to the situation for the standard deformation variety, where it is conceivable that we could have a solution to the gluing equations but an edge of the triangulation not be essential. An edge is {\bf essential} if it cannot be homotoped into the boundary of the manifold. If an edge of the triangulation is not essential, then in the lift of the triangulation to the universal cover, the lift of the edge will still be inessential. But then we run into trouble in defining a developing map from the universal cover to $\Hthree$, since the two endpoints of the inessential edge would have to map to the same point on $\bdry \Hthree$, but this cannot happen for an ideal hyperbolic tetrahedron. See \cite{st_essential} for details. It is shown in \cite{st_essential} that this situation cannot happen; that the existence of a point of the standard deformation variety implies that all edges are essential.\\

For our purposes, an inessential edge is only a problem to defining a developing map if it is in $\mathcal{E}_+$. The analogous result is:
\begin{lemma}
Let $\mathcal{E}_0, \mathcal{E}_+$ be the partition of the edge set $\mathcal{E}$ of $\mathcal{T}$, corresponding to the horo-normal surface $S$. If $\widehat{\mathfrak{D}}(M;\mathcal{T};S)\neq \emptyset$ then every edge in $\mathcal{E}_+$  is essential.
\end{lemma}
\begin{proof}
The proof is similar to that of Lemma 1 in \cite{st_essential}. The main difference is that in addition to drilling out tubular neighbourhoods of the edges to form the handlebody manifold $H$, we also cut the manifold along the horo-normal surface and remove the component containing the cusps. Call the resulting drilled and cut manifold $H'$. In fact, cutting away the outside region(s) has only a small effect on the combinatorics of the drilled manifold. The combinatorics of drilled tetrahedra of types 1111, 211 and 22 are unchanged, and the effect is that we remove tetrahedra of type 4 and introduce boundary faces on tetrahedra of type 31 (see Figure \ref{5_tetra_types_w_graphs.pdf}). Thus $H'$ is also a handlebody, and so similarly to as in \cite{st_essential}, a point of the extended deformation variety gives us a developing map from the universal cover of $H'$, and so a representation $\rho:\pi_1H'\to\PSL$. The consistent development condition replaces the gluing equations, and tells us that any curve in the manifold that is null-homotopic before drilling and cutting maps to the identity isometry under $\rho$. The rest of the proof follows.
\end{proof}

This implies that whenever we have a point of the extended deformation variety, then we can indeed define a developing map for it.

\section{Porous horo-normal surfaces}\label{porous horo-normal surfaces}

In order to be able to define the extended deformation variety $\widehat{\mathfrak{D}}(M;\mathcal{T};S)$, we require that the horo-normal surface $S$ is porous. In fact, we will want each of the (finitely many) horo-normal surface relative to the triangulation to be porous. In Section \ref{Retriangulating}, we give an algorithm for changing a triangulation for which some horo-normal surfaces are not porous into a triangulation (of the same manifold) for which all are porous. However, for many triangulations, modification is unnecessary.

\subsection{Small manifolds with a single torus boundary component}\label{small one cusp manifolds}

\begin{thm}\label{small irred one cusp}
Suppose that $M$ is a small irreducible manifold with a single torus boundary component, and $\mathcal{T}$ a triangulation of
$M$. Then every horo-normal surface in $\mathcal{T}$ is porous.
\end{thm}
By a \textbf{small} manifold, we mean that the only closed incompressible surfaces are boundary parallel.

\begin{proof}
First note that $R_\text{out}$ is connected. This is true because, by Remark \ref{horo is nbd of E0}, $R_\text{out}$ is a small regular neighbourhood of the subcomplex of $\mathcal{T}$ generated by $\mathcal{E}_0$, and this subcomplex is connected since the vertices of all cells are at the single boundary component.\\

Let $R_\text{in}^{(1)}, R_\text{in}^{(2)}, \ldots, R_\text{in}^{(n)}$ be the components of $R_\text{in}$, with boundaries $S^{(i)}$. Compress all of the components of all of the $S^{(i)}$ as much as possible, to produce surfaces $S_*^{(i)}$ (each of which may not be a connected surface). We can do this in an ``innermost first'' order, keeping all of the resulting surfaces disjoint. That is, if we want to make a compression along a disk which intersects the union of the surfaces, first perform compressions along the innermost circle(s) in the intersection.  A compression move can only disconnect regions in the complement of a surface, never connect regions that were not connected before, so it makes sense to classify the new regions in the complement of the surface as either \textbf{outside} or \textbf{inside}, depending on whether they came from $R_\text{out}$ or one of the $R_\text{in}^{(i)}$. Each compression move has a reverse, which in terms of the complementary regions either adds a 1-handle to an inside region (or between two inside regions), or adds a 1-handle to an outside region (or between two outside regions), which we interpret as drilling a hole through an inside region. There are three possibilities for each component of the resulting surfaces $S_*^{(i)}$. Each is either: 

\begin{enumerate}
\item A closed incompressible surface (possibly a sphere) that is not boundary parallel, 
\item or a closed incompressible surface, that is a torus parallel to the boundary torus,
\item or an $S^2$ that bounds a ball.
\end{enumerate}

Case 1 is ruled out because the manifold is small and irreducible, so there is no graph of groups decomposition of $\pi_1M$ with surface groups as edge groups.  So every component is either a sphere that bounds a ball, or a boundary parallel torus. Suppose that there are $k$ boundary parallel tori. Ignoring for a moment any spheres from case 3, there are then $k$ alternating inside and outside $T^2 \times I$ regions between the tori, the outermost incident to $\bdry M$, and a central region homeomorphic to $M$. Now add in the spheres: each region has some number of possibly nested spheres within it, each of which flips its interior from inside to outside or vice versa.\\

Consider first the case that $k$ is even. Then the central region is an outside region. Now undo all of the compression moves, by attaching 1-handles and drilling through the inside regions. This connects all of the outside regions together to form $R_\text{out}$. This, together with the central region being an outside region implies that in fact $\til{R_\text{out}}$ is connected. To see this, note that adding 1-handles to, and drilling holes through the inside regions cannot divide the lift of the central outside region to the universal cover into more than one component, nor can it divide the lifts of any of the outside $T^2 \times I$ regions. \\

Next, since $\til{R_\text{in}}$ is not empty (since by definition, the horo-normal surface is non-empty), there is some edge in $\til{\mathcal{E}_+}$. The edge intersects the surface twice, going from $\til{R_\text{out}}$ to  $\til{R_\text{in}}$ to $\til{R_\text{out}}$. The two endpoints of this edge are in the same (the only) component of $\til{R_\text{out}}$. Therefore this edge, together with some edges from $\til{\mathcal{E}_0}$ form a bad loop, which is a contradiction to the surface being horo-normal.\\

Now suppose that $k$ is odd. Then the central region is an inside region. Again, undo all of the compression moves. We claim that $\til{R_\text{in}}$ is connected. If not, there must be a component that does not get connected to the central region as we undo the compression moves, resulting in some region $R_\text{in}^{(i)}$. This region is made from some number of $T^2 \times I$ pieces (possibly zero) and balls, connected together by 1-handles, and those pieces and handles drilled out in some way. The handles may be quite complicated, for example there could be a hole drilled through the central region along a curve that is non-trivial in $\pi_1M$, and then a 1-handle attached between $T^2 \times I$ pieces, following along inside the hole. Nevertheless, we claim that a given connected lift of the region, $\til{R_\text{in}^{(i)}}$, has a single component of $\til{R_\text{out}}$ surrounding it, and so by the same argument as for $k$ being even, this gives us a bad loop and a contradiction. To prove the claim, notice that near a cusp in $\til{M}$, there is a single component surrounding $\til{R_\text{in}^{(i)}}$, since $R_\text{out}$ is connected. When the part of $\til{R_\text{in}^{(i)}}$ near this cusp connects to a part near some other cusp, it does so by 1-handles, which are always surrounded by the same, single outside component.\\

Thus, there can be only one component of $\til{R_\text{in}}$. This proves condition 1 for a surface being porous. Condition 2 follows trivially from the fact that there is only one cusp and that the horo-normal surface is non-empty, which means that $\mathcal{E}_+$ is non-empty.
\end{proof}

\subsection{Retriangulating}\label{Retriangulating}
In this section we show how to alter any given triangulation of a manifold in order to turn all of the horo-normal surfaces that are not porous into ones that are. The tool we will use to do this involves the following object.

\begin{defn}
A {\bf pillow} is a pair of tetrahedra that share an edge and are the only two tetrahedra incident to that edge. See Figure \ref{pillow.pdf}. 
\end{defn}

\begin{figure}[htb]
\centering
\includegraphics[width=0.25\textwidth]{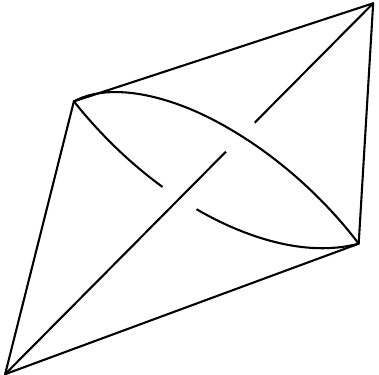}
\caption{A pillow.}
\label{pillow.pdf}
\end{figure}

\begin{defn}
The move of {\bf inserting a pillow} alters a triangulation $\mathcal{T}$ as follows. Take a pair of triangles of $\mathcal{T}$ which share an edge $e$, open up a gap between the two pairs of faces of tetrahedra that meet at the pair of triangles and insert a pillow between them. \end{defn}
\begin{rmk}
Inserting a pillow corresponds to Matveev's lune move (Definition 1.2.9 of \cite{matveev}).
\end{rmk}
\begin{defn}
In the new triangulation, $\mathcal{T}'$, in place of the edge $e$ there are two corresponding edges $e^\uparrow$ and $e^\downarrow$, and in place of each triangle involved are two corresponding triangles, each of which shares the same vertices as the edge or triangle it came from. We refer to these edges and triangles as {\bf splits} of $e$ and the original triangle respectively. If we go on to make further pillow insertions we will also recursively refer to splits of those splits of $e$ as splits of $e$.
\end{defn}

Let $\mathcal{E}$ be the edge set of the triangulation $\mathcal{T}$ and $\mathcal{T'}$ the triangulation obtained by inserting a pillow across an edge $e$. $\mathcal{T'}$ has edge set $\mathcal{E}'=(\mathcal{E} \setminus \{e\})\cup\{e^\uparrow,e^\downarrow,f\}$ where $f$ is the edge joining opposite vertices of the pair of triangles. If we have a subset of zero-length edges $\mathcal{E}_0 \subset \mathcal{E}$, we can ask what possible subsets  $\mathcal{E}'_0 \subset \mathcal{E}'$ are compatible with $\mathcal{E}_0$, meaning that $\mathcal{E}_0\subset \mathcal{E}_0'$ if $e\notin \mathcal{E}_0$ and $(\mathcal{E}_0\setminus \{e\}) \cup \{e^\uparrow,e^\downarrow\} \subset \mathcal{E}_0'$ if $e\in \mathcal{E}_0$. We analyse the possibilities in Figure \ref{pillows_E0.pdf}, which shows all possible configurations without bad loops in the original pair of triangles. Note that the splits of $e$ must either both be in $\mathcal{E}'_0$ or both be in $\mathcal{E}'_+$ in order to avoid a bad loop.

\begin{figure}[htb]
\centering
\includegraphics[width=\textwidth]{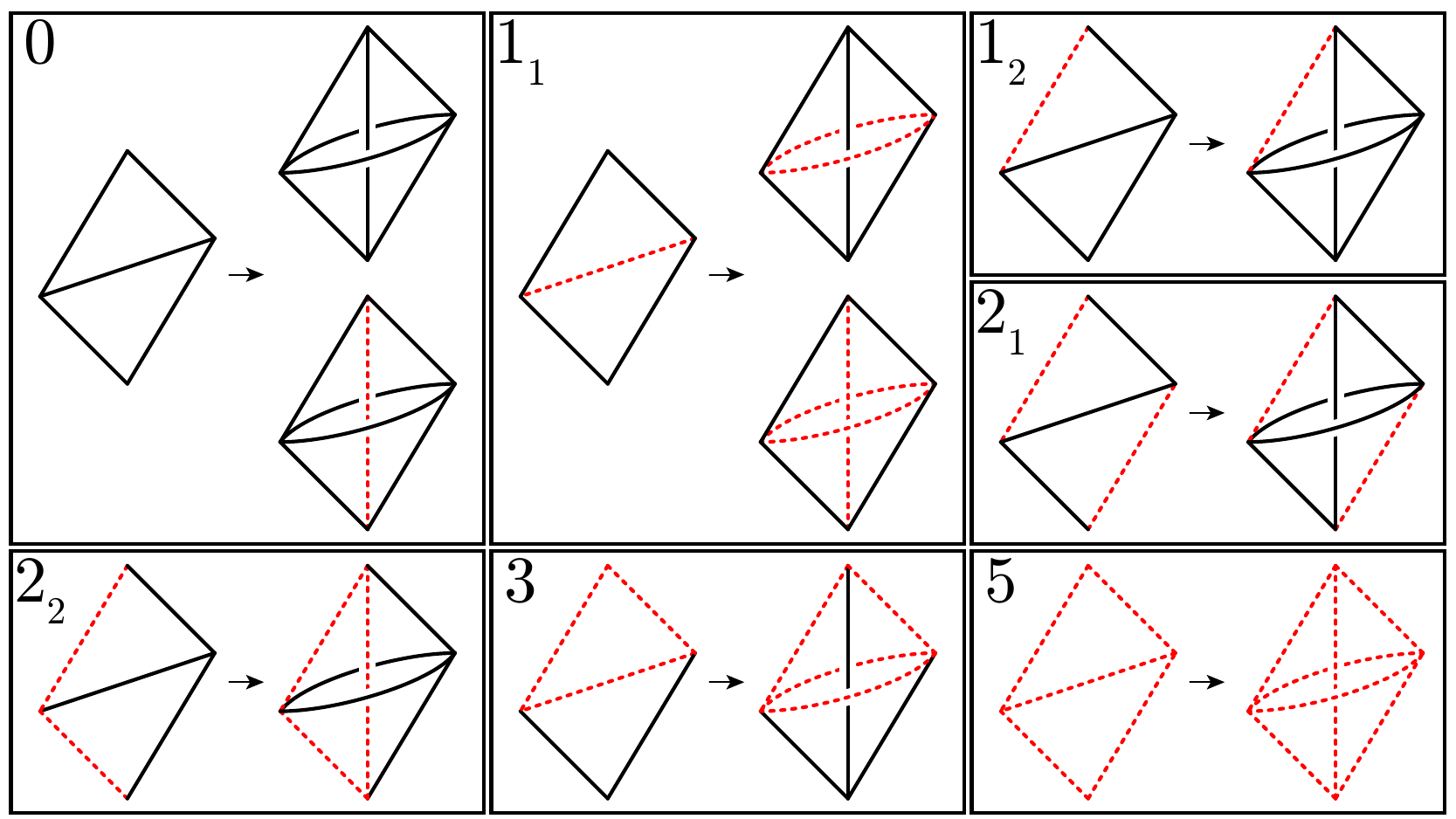}
\caption{The seven possible configurations (up to symmetry) of edges in $\mathcal{E}_0$ or $\mathcal{E}_+$ for two triangles that meet at an edge, and the 9 possible configurations after inserting a pillow. Edges in $\mathcal{E}_0$ are shown with dotted lines. The numerical labels refer to the number of edges of the two triangles that are in $\mathcal{E}_0$, and where there is more than one configuration with that many we have assigned subscripts to the labels arbitrarily. For each of the five configurations  $1_2,2_1,2_2,3$ and $5$ there is only one possibility for whether the added edge is in $\mathcal{E}_0$ or $\mathcal{E}_+$ that avoids bad loops. For the two configurations $0$ and $1_1$ the local picture does not force this, although edges outside of the two triangles may do so.}
\label{pillows_E0.pdf}
\end{figure}

\begin{lemma}\label{child_exists}
Suppose for a triangulation $\mathcal{T}$ we have subset $\mathcal{E}_0$ with no bad loops. Let $\mathcal{T}'$ be the triangulation obtained from $\mathcal{T}$ by inserting a pillow, so $\mathcal{E}'=(\mathcal{E}\setminus\{e\})\cup\{e^\uparrow,e^\downarrow,f\}$. Let $\mathcal{E}''_0 = (\mathcal{E}_0\setminus\{e\}) \cup \{e^\uparrow, e^\downarrow\}$ if $e\in\mathcal{E}_0$, and $\mathcal{E}''_0 = \mathcal{E}_0$ if not. Then at least one of the subsets $\mathcal{E}'_0 = \mathcal{E}''_0$ or $\mathcal{E}'_0 = \mathcal{E}''_0 \cup \{f\}$ has no bad loops in $\mathcal{E}'$.  
\end{lemma}
\begin{proof}
Assume for contradiction that both choices for $f$ would result in a bad loop. First, $f$ being positive-length gives a bad loop. This means that for each lift $\til{f}_\lambda$ of $f$ to $\til{\mathcal{T}}$, there is a path $\delta_\lambda$ of zero-length edges in $\til{\mathcal{E}'\setminus\{f\}}$ connecting the two endpoints of $\til{f}_\lambda$. We also have that $f$ being zero-length gives a bad loop $\gamma$. Replace each lift $\til{f}_\lambda$ of $f$ in $\gamma$ with the corresponding path $\delta_\lambda$. Now replace 
any splits $e^\uparrow$ or $e^\downarrow$ in $\gamma$ with $e$ and we get a bad loop for the subset of zero-length edges $\mathcal{E}_0 \subset \mathcal{E}$, contradicting the hypothesis.
\end{proof}
\begin{defn}
Because of Lemma \ref{child_exists}, given a horo-normal surface relative to a triangulation $\mathcal{T}$, there are either one or two corresponding horo-normal surfaces relative to the triangulation $\mathcal{T}'$ where $\mathcal{T}'$ is obtained from $\mathcal{T}$ by inserting a pillow. We call such a derived horo-normal surface a {\bf child} of the original surface, and the original surface the {\bf parent} of the derived surface. If we go on to insert further pillows, we call the children of children (with any number of generations) of a surface the {\bf descendants} of that surface. We extend these definitions to the subsets $\mathcal{E}_0$ with no bad loops corresponding to the horo-normal surfaces in the obvious way. Likewise, the set of regions of $\til{R_\text{in}}$ (in the complement of $S$) has either one or two descendant sets of regions in the complements of the descendants of $S$.
\end{defn}
\begin{lemma}
Suppose that $\mathcal{T}'$ is obtained from $\mathcal{T}$ by inserting a pillow and that $S$ is a horo-normal surface relative to $\mathcal{T}'$. Then $S$ has exactly one parent.
\end{lemma}

\begin{proof}
Let $\mathcal{E}'_0$ be the subset corresponding to $S$, and let $e^\uparrow, e^\downarrow,f$ be the edges added by inserting the pillow. The subset $\mathcal{E}'_0$ has no bad loops, and collapsing $e^\uparrow,e^\downarrow$ to $e$ and removing $f$ cannot create a bad loop, so taking $\mathcal{E}_0 =(\mathcal{E}'_0 \setminus \{e^\uparrow, e^\downarrow,f\})\cup\{e\}$ if $e^\uparrow,e^\downarrow\in\mathcal{E}'_0$, or $\mathcal{E}_0 =\mathcal{E}'_0 \setminus \{f\}$ if not produces a parent of $S$. Any parent of the subset $\mathcal{E}'_0$ must agree with $\mathcal{E}'_0$ everywhere but at $f$, and contains $e$ if and only if $\mathcal{E}'_0$ contains $e^\uparrow$ and $e^\downarrow$, so the parent is unique.
\end{proof}
Given a sequence of pillow insertions starting from a triangulation $\mathcal{T}$, the horo-normal surfaces of the various triangulations thus form a graded forest (a disjoint union of trees, one tree for each horo-normal surface relative to $\mathcal{T}$), with all roots of trees at the level of $\mathcal{T}$ and a non-decreasing number of nodes at successive levels. 

\begin{lemma}\label{insert_pillow_improves}
Suppose that $S$ is a surface in horo-normal form relative to a triangulation $\mathcal{T}$. Suppose that we insert a pillow, changing $\mathcal{T}$ to $\mathcal{T}'$, and let $S'$ be a descendant of $S$ in $\mathcal{T}'$. Let $\til{R_\text{in}}$ and $\til{R_\text{in}}'$ be the inside regions relative to $S$ and $S'$ respectively. Then there is a natural surjective map from the set of components of $\til{R_\text{in}}$ to the set of components of $\til{R_\text{in}}'$.
\end{lemma}
\begin{proof}
First note that we can see regions of $\til{R_\text{in}}$ from diagrams of the edges of tetrahedra marked as being in $\til{\mathcal{E}}_0$ or  $\til{\mathcal{E}}_+$ by looking at the midpoints of edges in $\til{\mathcal{E}}_+$. Observe that two midpoints are part of the same component of $\til{R_\text{in}}$ if and only if they are connected by a path of midpoints of edges in $\til{\mathcal{E}}_+$ where neighbouring midpoints are midpoints of edges of triangles of $\til{\mathcal{T}}$ (see Figure \ref{5_tetra_types_w_graphs.pdf}). Now consider the possible moves in Figure \ref{pillows_E0.pdf}. In all possible configurations of pairs of triangles and pillows only the pair of triangles in configuration $1_1$ has the midpoints of edges in $\til{\mathcal{E}}_+$ corresponding to potentially distinct components of $\til{R_\text{in}}$. Those potentially distinct components get connected together when we insert the pillow, no matter which choice we make for the added edge. Thus connectivity can only increase for the descendants of a horo-normal surface. \\

The above argument means that it makes sense to view a component of $\til{R_\text{in}}$ as having a descendant component of $\til{R_\text{in}}'$. A descendant component may merge with the descendant of some other component, but no component can have multiple descendants.
\end{proof}

\begin{rmk}\label{splits_connected}
Notice that the midpoints of splits of an edge in $\til{\mathcal{E}}_+$ are always connected to each other, and so this will be true for all splits of that edge produced by further pillow insertions. 
\end{rmk}

\begin{lemma}\label{fixes_stick}
Suppose that $S$ is a surface in horo-normal form relative to a triangulation $\mathcal{T}$ which is porous. Then all descendants of $S$ are also porous.
\end{lemma}
\begin{proof}
Let $S'$ be a descendant of $S$. Condition 1 of porousity (Definition \ref{xdv}) follows directly from Lemma \ref{insert_pillow_improves}. Condition 2 is clear since inserting a pillow only removes edges when replacing them with splits (which have the same endpoints and are in $\mathcal{E}_+$ if the original edge is). Therefore the subset $\mathcal{E}_+'$  corresponding to $S'$ contains an edge or a split of it if the subset $\mathcal{E}_+$ corresponding to $S$ does. 
\end{proof}

The results so far tell us that the pillow insertion move can never introduce truly new horo-normal surfaces, only descendants of surfaces we started with. We also know that the move cannot make a horo-normal surface ``less porous'', and can only improve matters. Next, we will need a structure that we can use to organise a sequence of pillow insertion moves, in order to make the descendant horo-normal surfaces porous.

\begin{defn}\label{strip_of_tris}
A {\bf strip of triangles} in $\mathcal{T}$ (or $\til{\mathcal{T}}$) is a sequence of triangular faces $(\tri_{1}, \tri_{2},\ldots,\tri_{n})$ of $\mathcal{T}$ (or $\til{\mathcal{T}}$) alternating with a sequence of edges $(e_1,e_2,\ldots,e_{n-1})$ such that neighbouring triangles $\tri_{k}, \tri_{k+1}$ share the edge $e_{k}$ (the {\bf internal\footnote{As in, internal to the strip.} edge} between triangles). For each $\tri_{k}$, $k = 2,3,\ldots, n-1$, the internal edges either side, $e_{k-1}$ and $e_{k}$, must be distinct. We allow repeated triangles and edges in the strip, even consecutive repeated triangles.\\

 We will sometimes write a strip as $(\tri_{1},e_1, \tri_{2},e_2,\ldots,e_{n-1},\tri_{n})$. Compare with Definition \ref{chain_of_triangles}.
\end{defn}
 
\begin{defn}
If $\mathcal{T}$ is an ideal triangulation of a manifold $M$, we construct a corresponding {\bf handle decomposition} $\mathcal{T}^\odot$ of $M$ by ``thickening up'' the edges and triangles of $\mathcal{T}$. Each edge $e$ is replaced by a polygonal tubular neighbourhood (or {\bf thick edge}) $e^\odot$, where the number of sides of the polygon is equal to the number of triangles incident to the edge. Each triangle $\tri$ is replaced by a triangular prism (or {\bf thick triangle}) $\tri^\odot$ with the rectangular faces coinciding with the rectangular faces of the thick edges. Each tetrahedron $t$ remains combinatorially the same but shrinks a little to become $t^\odot$ so that its faces coincide with the triangular faces of the thick triangles. See Figure \ref{thick_T.pdf}. Inserting a pillow into $\mathcal{T}$ has a corresponding effect on $\mathcal{T}^\odot$ and splits of thick edges and triangles are defined in an analogous way.
\end{defn}

\begin{figure}[htb]
\centering
\includegraphics[width=0.9\textwidth]{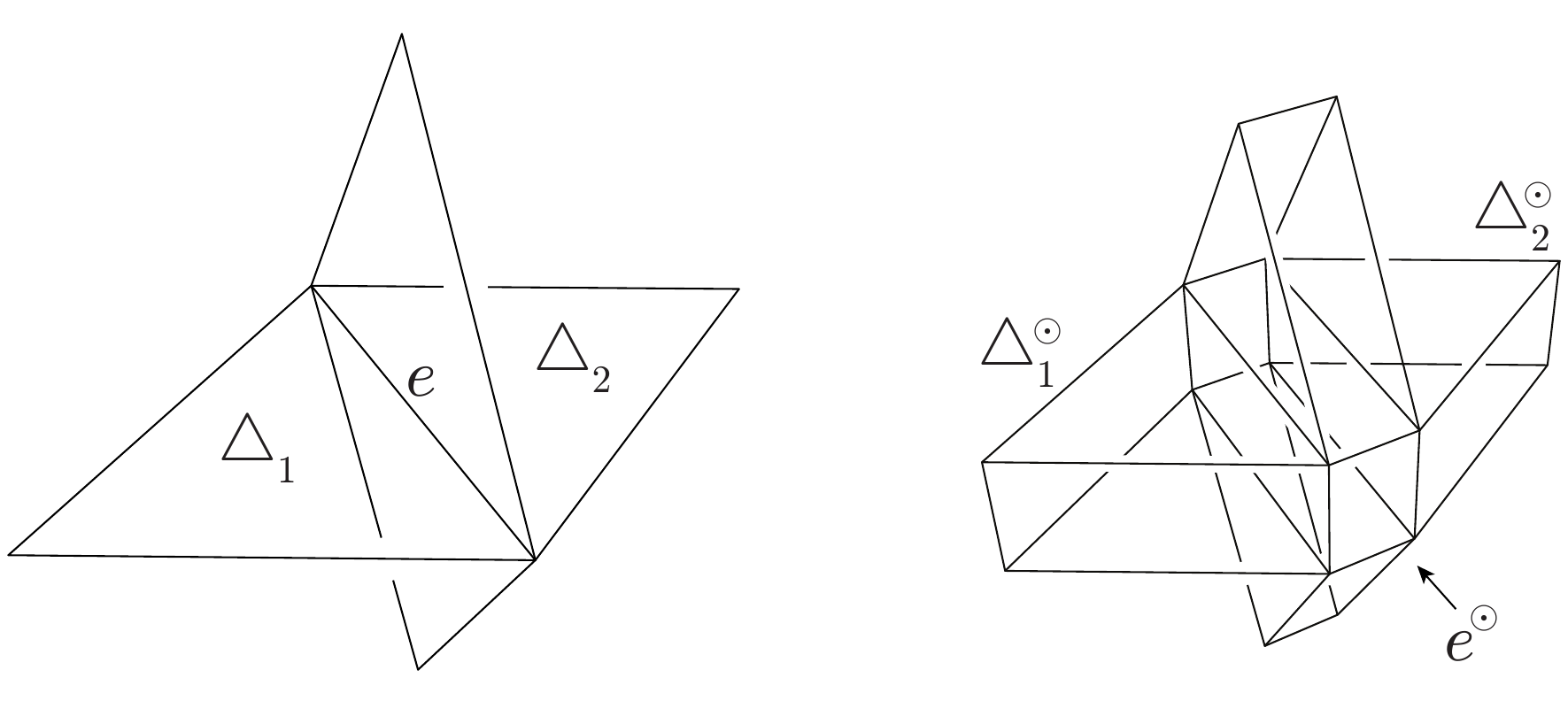}
\caption{Four triangles of $\mathcal{T}$ that meet at an edge and the corresponding four thick triangles of $\mathcal{T}^\odot$ meeting at the corresponding thick edge.}
\label{thick_T.pdf}
\end{figure}

\begin{defn}
A {\bf strip of triangles} in $\mathcal{T}^\odot$ (or $\til{\mathcal{T}}^\odot$) is a sequence of $n$ triangles embedded in and respecting the product structure of the thick triangles, alternating with $n-1$ rectangular pieces  embedded in and respecting the product structure of the thick edges such that the pieces connect together analogously to as in Definition \ref{strip_of_tris}. See Figure \ref{strip_in_thick.pdf}. If we collapse all of the product structures to recover $\mathcal{T}$ then a strip of triangles in $\mathcal{T}^\odot$ becomes a strip of triangles in $\mathcal{T}$.
\end{defn}
We can think of a strip of triangles in $\mathcal{T}^\odot$ as having all of the data of a strip of triangles in $\mathcal{T}$ together with ordering information in the case when the strip passes through a triangle or edge multiple times. We will refer to a strip of triangles $\sigma'$ in $\mathcal{T}^\odot$ that collapses to a  strip of triangles $\sigma$ in $\mathcal{T}$ as an \textbf{ordered version} of $\sigma$.

\begin{figure}[htb]
\centering
\includegraphics[width=0.4\textwidth]{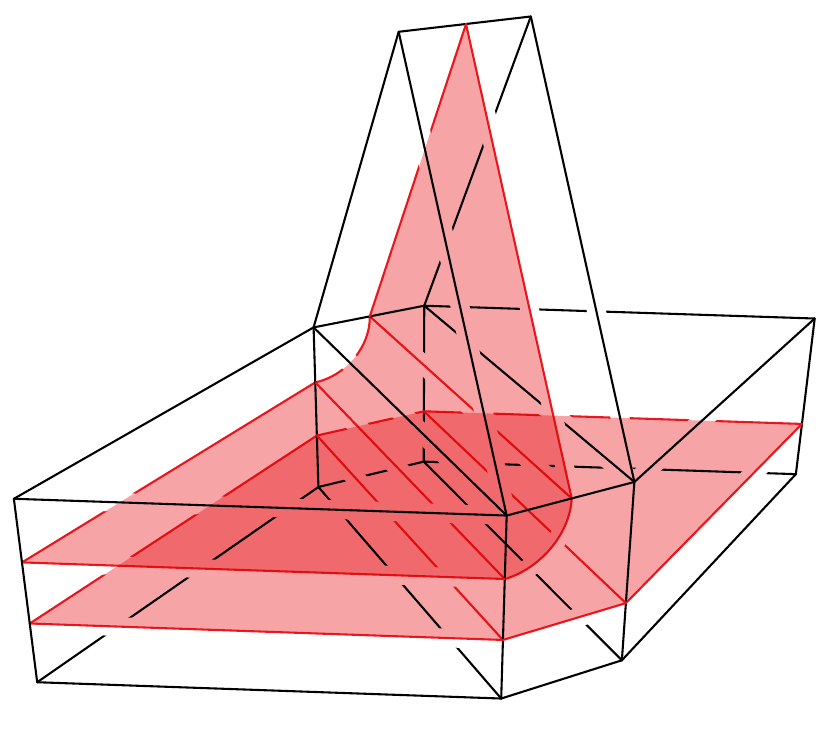}
\caption{Strips embedded in part of $\mathcal{T}^\odot$.}
\label{strip_in_thick.pdf}
\end{figure}

\begin{algor}(Insert pillows along a strip of triangles in $\mathcal{T}^\odot$)
\label{insert pillows}

\indent{\bf Input:} an ideal triangulation $\mathcal{T}$ of a connected 3-manifold $M$, together with a strip of triangles $\sigma'=(\tri'_1,e'_1, \tri'_2,e'_2,\ldots,e'_{n-1},\tri'_{n})$ embedded in $\mathcal{T}^\odot$.\\
\indent{\bf Output:} a triangulation of $M$ obtained from $\mathcal{T}$ by ``inserting pillows along $\sigma'$''.\\

Let $\sigma=(\tri_1,e_1, \tri_2,e_2,\ldots,e_{n-1},\tri_{n})$ be the image of $\sigma'$ after collapsing the handle decomposition  $\mathcal{T}^\odot$ to  $\mathcal{T}$.  Note that the strip may go through a handle of $\mathcal{T}^\odot$ multiple times, and so there may be repeated triangles and edges in this list.\\

First we insert a pillow between $\tri_{1}$ and $\tri_{2}$ across $e_1$. This changes the triangulation $\mathcal{T}$ to a new triangulation which we will call $\mathcal{T}_1$. See Figures \ref{thick_pillow.pdf} and \ref{thick_pillow_compare.pdf} for the handle decomposition of a thick pillow and how the handle decomposition changes under a pillow insertion. The handle decomposition $\mathcal{T}^\odot_1$ corresponding to $\mathcal{T}_1$ inherits an embedded strip of triangles under the pillow insertion move in a natural way: all of the handles of $\mathcal{T}^\odot$ other than $\tri_1^{\odot},e_1^\odot$ and $\tri_{2}^\odot$ remain in $\mathcal{T}^\odot_1$, and they keep the strip parts within them, connected to each other in the same way. The handles $\tri_1^{\odot},e_1^\odot$ and $\tri_{2}^\odot$ are replaced by the handles as shown in Figure \ref{thick_pillow.pdf}. Parts of $\sigma'$ that went through $\tri_1^{\odot},e_1^\odot$ and $\tri_{2}^\odot$ now go through either $\tri_1^{\uparrow\odot},e_1^{\uparrow\odot}$ and $\tri_{2}^{\uparrow\odot}$ or  $\tri_1^{\downarrow\odot},e_1^{\downarrow\odot}$ and $\tri_{2}^{\downarrow\odot}$, depending on the original ordering of those parts of $\sigma'$ in  $\tri_1^{\odot},e_1^\odot$ and $\tri_{2}^\odot$ relative to $\tri'_1,e'_1,$ and $\tri'_2$, and connect to other handles in the natural way.  Finally, one of either $\tri_L$ or $\tri_R$ (as in Figure \ref{thick_pillow.pdf}) is incident to $e_2$, and we add a copy of that triangle to the front of the new strip of triangles in $\mathcal{T}^\odot_1$, connecting onto the rectangle $e'_2$. This gives a new strip of triangles, $\sigma'_1$ embedded in $\mathcal{T}^\odot_1$.\\

We now repeat the process, using the first two triangles of the new strip to insert a new pillow, to construct $\mathcal{T}^\odot_2$ with a new embedded strip of triangles $\sigma'_2$, and so on. Finally, the strip contains only two triangles, and after inserting a pillow between these two triangles we are done.\\

See Figure \ref{pillows_in_strip.pdf} for a schematic picture showing the result of this algorithm.\\

Note that it is possible that the first two triangles in the sequence are in the same thick triangle, with the strip folded over, backtracking immediately. In this case the pillow insertion move is slightly different, opening up only one triangle and inserting a ``folded pillow''. Allowing this modified move does not change any of the arguments in this section.

\end{algor}

\begin{figure}[htb]
\centering
\includegraphics[width=\textwidth]{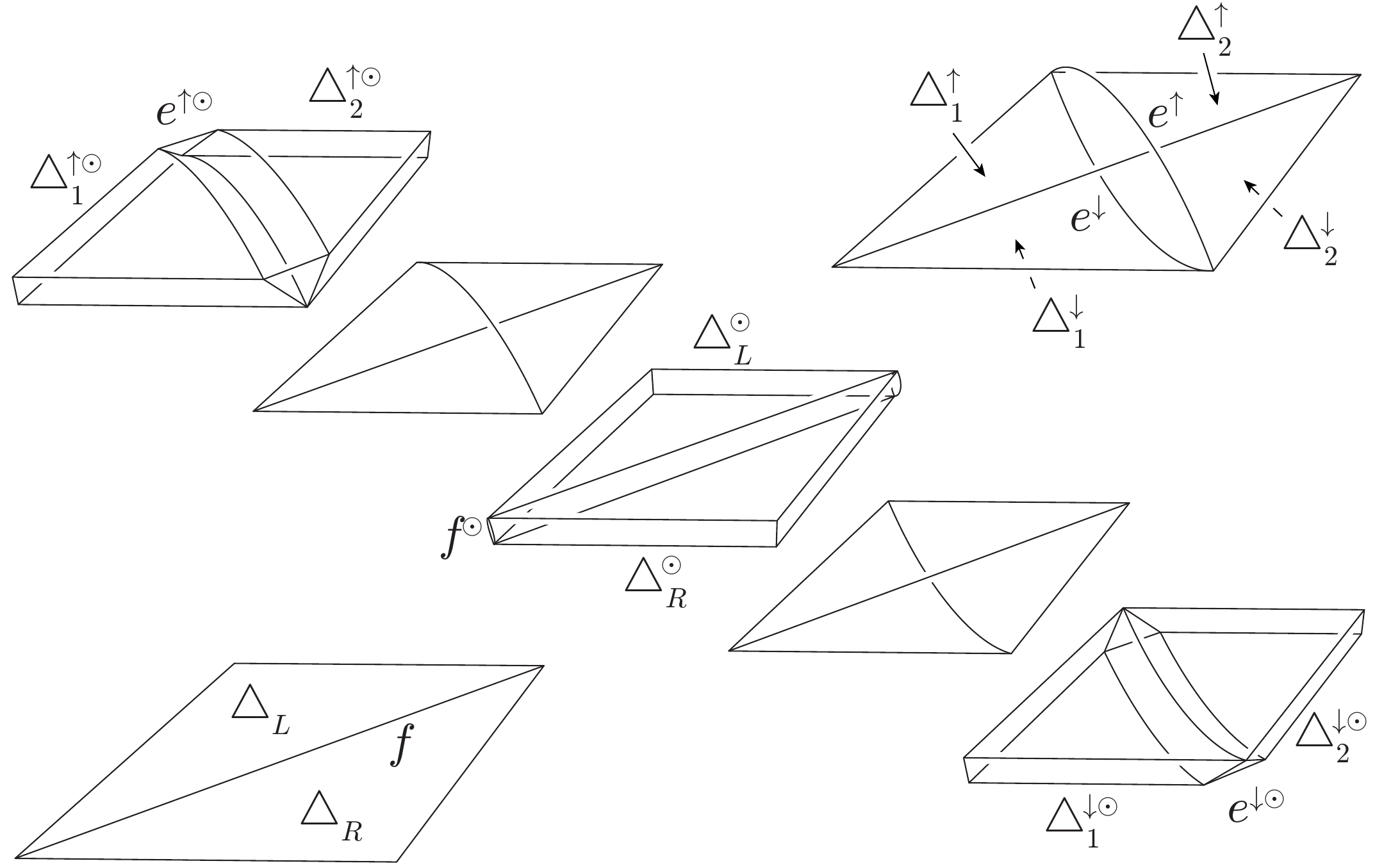}
\caption{Top right: a pillow in $\mathcal{T}$. Bottom left: the triangles inside the pillow, redrawn for clarity. Diagonally from top left to bottom right: the corresponding thick pillow in $\mathcal{T}^\odot$ in exploded view. Top left are two thick triangles and a thick edge, next a tetrahedron, next two thick triangles and a thick edge. The lower half is the mirror image of the upper half.}
\label{thick_pillow.pdf}
\end{figure}

\begin{figure}[htb]
\centering
\includegraphics[width=\textwidth]{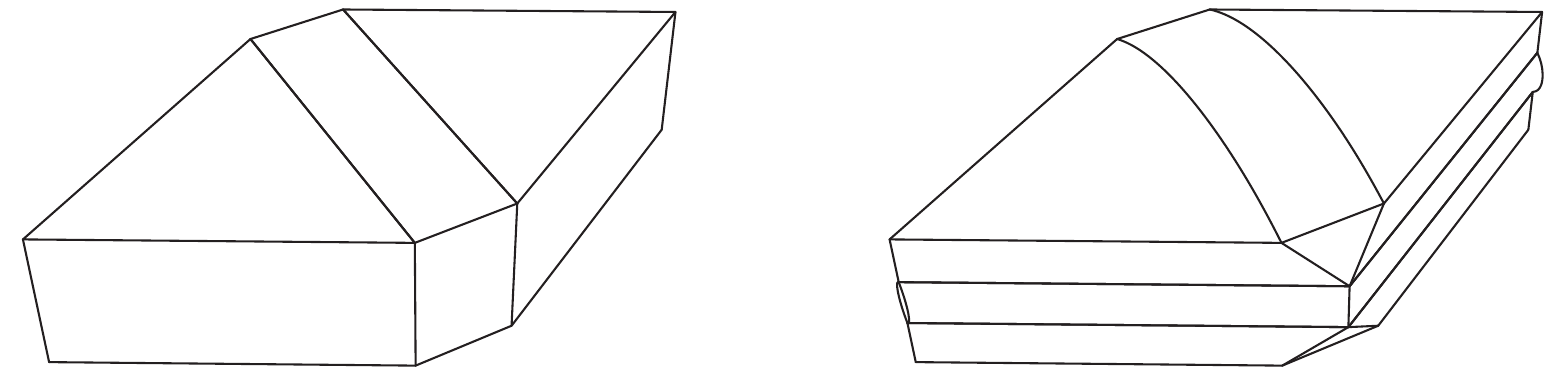}
\caption{Thick versions of a pair of triangles and the corresponding pillow, only the boundary faces are shown. The versions shown here correspond to the situation in Figure \ref{thick_T.pdf}, with one triangle incident to the central edge above and one below. If there are more above then the rectangular face is cut into more strips, in both the pair of triangles and pillow diagrams, and similarly below.} 
\label{thick_pillow_compare.pdf}
\end{figure}

\begin{figure}[htb]
\centering
\includegraphics[width=0.8\textwidth]{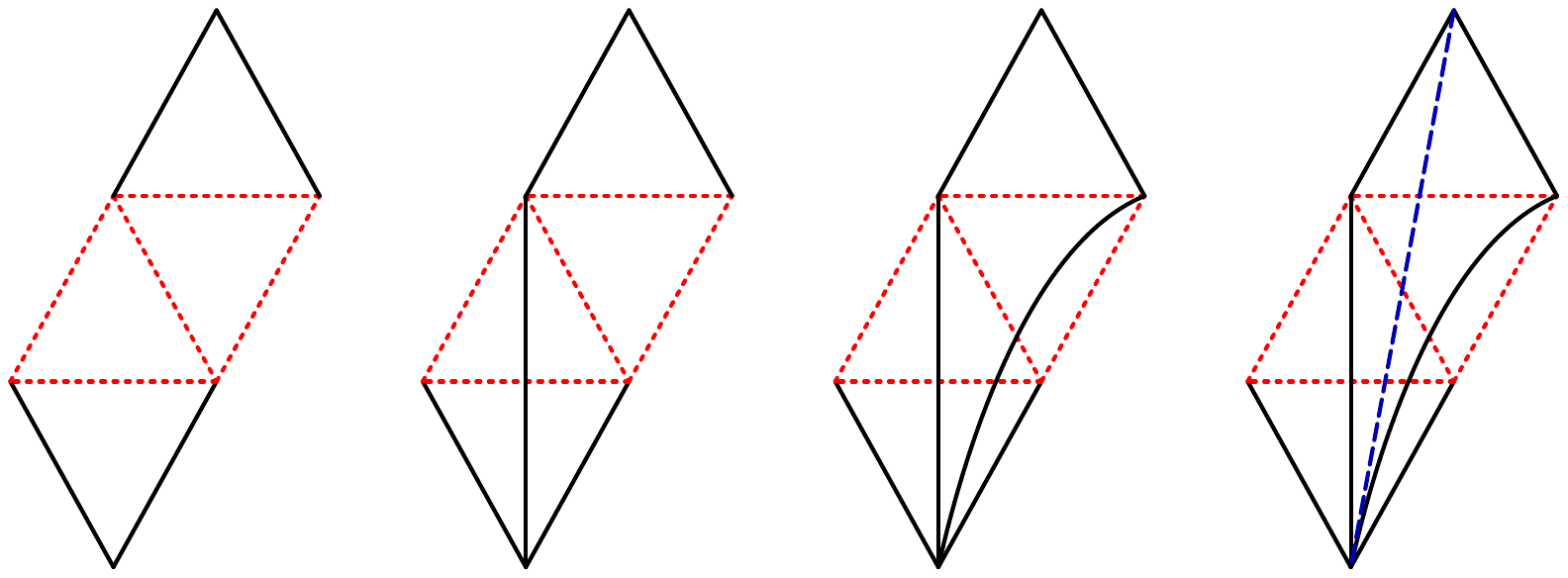}
\caption{Inserting pillows along a strip joining two components of $\til{R_\text{in}}$. Edges in $\til{\mathcal{E}}_+$ are solid, edges in $\til{\mathcal{E}}_0$ are dotted, and the last added edge (which could be in either) is dashed.}
\label{pillows_in_strip.pdf}
\end{figure}

\begin{algor}(Connect two components of $\til{R_\text{in}}$ using a strip of triangles.)
\label{connect two components}

\indent{\bf Input:} a horo-normal surface $S$ relative to an ideal triangulation $\mathcal{T}$ of a connected 3-manifold $M$, and a strip of triangles $\til{\sigma}$ in $\til{\mathcal{T}}$ which projects to a strip of triangles $\sigma = (\tri_1,e_1, \tri_2,e_2,\ldots,e_{n-1},\tri_{n})$ in $\mathcal{T}$ which has an ordered version $\sigma'$ embedded in $\mathcal{T}^\odot$. We also require that $\tri_1$ and $\tri_n$ are of type $21$ and $e_1$ and $e_{n-1}$ are zero-length relative to the horo-normal surface $S$. Let $A$ and $B$ be the components of  $\til{R_\text{in}}$ that contain the midpoints of the positive-length edges of $\tri_1$ and $\tri_n$ respectively.\\
\indent{\bf Output:} a triangulation $\mathcal{T}^*$ of $M$ obtained from $\mathcal{T}$ by a finite number of pillow insertions such that all descendants of $S$ in horo-normal form relative to $\mathcal{T}^*$ each have their associated descendants of the regions $A$ and $B$ merged together into single components.\\

First consider the case that where $\tri_{1}, \tri_{n}$ are type 21 and $\tri_{2}, \ldots, \tri_{n-1}$ are type 3. (This is the case illustrated in Figure \ref{pillows_in_strip.pdf}.)
We insert pillows along the strip $\sigma'$, as in Algorithm \ref{insert pillows}. If $n=2$, and the strip $\sigma$ has only two triangles, then we are in configuration $1_1$ of Figure \ref{pillows_E0.pdf}. When we insert the pillow, the descendants of the previously disconnected components of $\til{R_\text{in}}$ become connected for both possible children of $S$. Otherwise, after the first pillow insertion we are in configuration $3$, the added edge $f$ is positive-length and we have two new triangles which have $f$ as an edge.  Both of these triangles are of type 21, and one of them (the new triangle added to the front of the new strip to form $\sigma'_1$) is arranged together with $\tri_{3}$ in configuration $3$ of Figure \ref{pillows_E0.pdf} (if $n>3$) or configuration $1_1$ (if $n=3$). We repeat, as in Algorithm \ref{insert pillows}, inserting pillows until we reach the same situation as when $n=2$, and the descendants of the previously disconnected components of $\til{R_\text{in}}$ become connected. See Figure \ref{pillows_in_strip.pdf}. Notice that pillow insertions early in the strip may split edges that the strip revisits. However, all of the splits of these edges will be zero-length edges, since the original edges were also zero-length edges. So the the pairs of triangles we insert pillows along are still in configuration $3$, or $1_1$ at the very end.\\

Now suppose that the middle triangles of the strip are not all of type 3. There must be some triangles of type 21, and possibly some of type 111. See Figure \ref{strip_decomposes.pdf}. The strip decomposes into a number of substrips, each of which has triangle types $(21,3,3,\ldots,3,21)$, or has a region of $\til{R_\text{in}}$ running along it. We run the above procedure on each of the substrips with types $(21,3,3,\ldots,3,21)$, and so connect the intermediate regions of $\til{R_\text{in}}$ together. This eventually connects the regions of $\til{R_\text{in}}$ at either end of the strip together.

\begin{figure}[htb]
\centering
\includegraphics[width=0.5\textwidth]{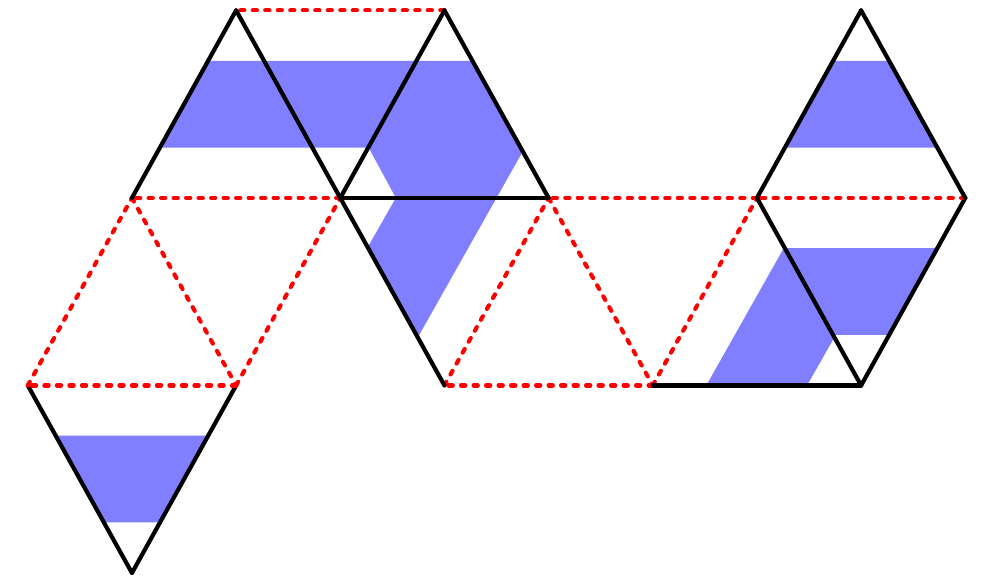}
\caption{A strip of triangles with a horo-normal surface. Edges in $\til{\mathcal{E}}_+$ are solid, edges in $\til{\mathcal{E}}_0$ are dotted. The regions of $\til{R_\text{in}}$ are shaded.}
\label{strip_decomposes.pdf}
\end{figure}

\end{algor}

\begin{rmk} \label{can do condition 2 also}
We will apply Algorithm \ref{connect two components} to make a horo-normal surface (or rather, its descendants) satisfy condition 1 of porousity. A very similar process also works to make a surface satisfy condition 2. If condition 2 fails then a vertex $v \in\mathcal{V}$ has only edges in $\mathcal{E}_0$ incident to it, and so all triangles incident to it are of type 3. All we have to do is insert pillows along a strip from a type 21 triangle through type 3 triangles out to $v$, and this works in exactly the same way as in Algorithm~\ref{connect two components}. 
\end{rmk}

\begin{algor}(Convert a set of strips of triangles in $\mathcal{T}$ into a set of strips of triangles whose union is embedded in $\mathcal{T}^\odot.$)
\label{make strips embedded}

\indent{\bf Input:}  a set of strips of triangles, $\{\sigma_i\}_{i=1}^n$, each in the ideal triangulation $\mathcal{T}$.\\
\indent{\bf Output:} a set of strips of triangles $\{\sigma'_i\}_{i=1}^n$ embedded in $\mathcal{T}^\odot$, where the first and last triangles, and first and last internal edges of each $\sigma'_i$ collapse to the corresponding first and last triangles and internal edges of $\sigma_i$.\\

We start by building the strip $\sigma'_1$ in $\mathcal{T}^\odot$. Suppose that  $\sigma_1 = (\tri_{1},e_{1}, \tri_{2},e_{2},\ldots,e_{m-1},\tri_{m})$. We build it in order from the start, placing triangles and rectangles into the handles, arbitrarily making any choices of where to place the next triangle relative to other triangles in a thick triangle, until we reach a barrier, in the form of a rectangle in $e_{i+1}^\odot$ blocking us from extending the strip from $\tri_{i}^\odot$ into $\tri_{i+1}^\odot$. See Figure \ref{push_aside_behind.pdf}.\\

\begin{figure}[htb]
\centering
\includegraphics[width=0.7\textwidth]{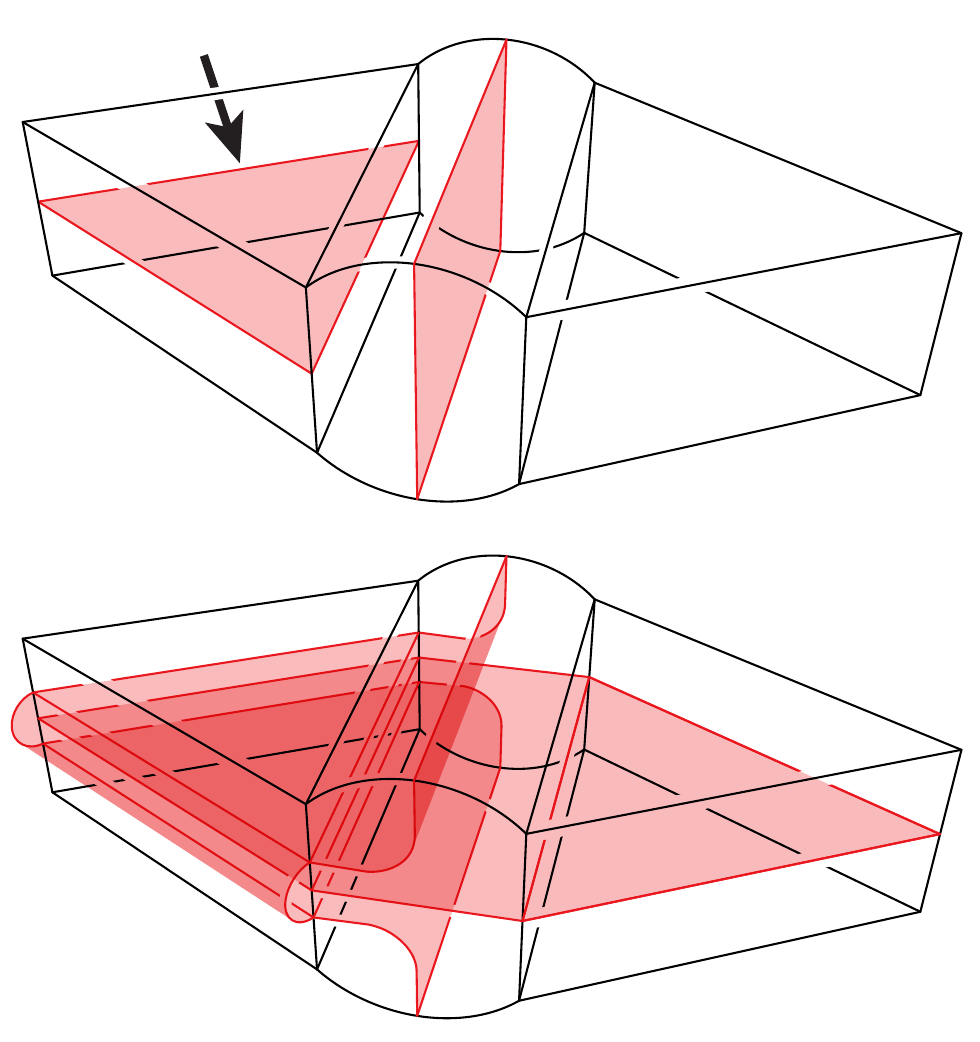}
\caption{A rectangle in the way of extending the strip and the result of pushing it aside. The left thick triangle is $\tri_{i}^\odot$, the right thick triangle is $\tri_{i+1}^\odot$ and the thick edge in the middle is $e_{i+1}^\odot$. The strip enters $\tri_{i}^\odot$ from the thick edge $e_{i}^\odot$ marked with an arrow. In these diagrams we do not specify how many thick triangles are incident above and below $e_{i+1}^\odot$, and the number does not alter the argument.}
\label{push_aside_behind.pdf}
\end{figure}
 
We ``push aside'' any such rectangles behind us, removing them and reconnecting the strip by adding two triangles and three rectangles each as shown, one of the rectangles pushing into the thick edge incident to $\tri_{i}^\odot$ other than $e_{i+1}^\odot$ and $e_i^\odot$. This process adds a finite number of pieces to the strip and the result is once again embedded. Using this move, we can deal with any further obstructions, and we end up with an altered strip, embedded in $\mathcal{T}^\odot$.\\

Note that this process of pushing the strip aside cannot move any preexisting triangles in the strip, nor can it change which edges the triangles are connected to in the strip. Thus, although the resulting strip $\sigma'_1$ may be longer than the original strip $\sigma_1$, it starts and ends in the same way as $\sigma_1$ does, as required.\\

We continue in exactly the same way for all of the other strips: we start building them in the appropriate thick triangle, with an arbitrary position relative to any preexisting triangles in that thick triangle, and build along the strip, pushing parts of this or other strips aside as needed. By the same argument as before, these moves do not change the start and end of any of the previous strips. The result is a set of strips of triangles embedded in $\mathcal{T}^\odot$, with the required beginning and ending properties.
\end{algor}

\begin{algor}(Make all horo-normal surfaces porous.)
\label{fix_all_horo-normal} 

\indent{\bf Input:}  an ideal triangulation $\mathcal{T}$.\\
\indent{\bf Output:} an ideal triangulation $\mathcal{T}_*,$ obtained from $\mathcal{T}$ by a finite number of pillow insertions, such that every horo-normal surface in $\mathcal{T}_*$ is porous.\\

There are a finite number of horo-normal surfaces in $\mathcal{T}$. For each one of them, we need to make it (or rather, its descendants) porous. We will do this for each surface by specifying a set of strips of triangles in $\til{\mathcal{T}}$ which will connect the components of $\til{R_\text{in}}$ (for that surface) together (condition 1 of porousity, using Algorithm \ref{connect two components}), and connect $\til{R_\text{in}}$ out to any vertices that do not already have a positive-length edge incident to them (condition 2 of porousity, using the variant of Algorithm \ref{connect two components} outlined in Remark \ref{can do condition 2 also}). We choose these strips in $\til{\mathcal{T}}$, project them to $\mathcal{T}$, and then use Algorithm \ref{make strips embedded} to modify them so that their union is embedded in $\mathcal{T}^\odot$. This done, we can apply Algorithm \ref{connect two components} and the variant outlined in Remark \ref{can do condition 2 also}. Notice that the modifications in Algorithm \ref{make strips embedded} do not move the start and end triangles of the strips, so the modified strips work the same way to connect components of $\til{R_\text{in}}$, or to connect $\til{R_\text{in}}$ to vertices. The edges of the start and end triangles of a strip $\sigma'$ may be split many times by pillow insertions along other strips. However, by Remark \ref{splits_connected}, the splits of positive-length edges that we actually connect to when inserting pillows along $\sigma'$ have midpoints in the corresponding descendants of the original components of $\til{R_\text{in}}$.  \\

 For each horo-normal surface, there are a finite number of strips that are needed. To see this, first choose some connected fundamental domain $D_0\subset\til{M}$ for $M$ made from some union of tetrahedra of the original triangulation $\til{\mathcal{T}}$. Consider the connected components of $\til{R_\text{in}}$ that intersect $D_0$. There will be some finite number of these since $D_0$ is made from finitely many tetrahedra, each of which can intersect at most one component of $\til{R_\text{in}}$. $D_0$ has some finite number of $\pi_1M$-translates that touch it along faces of $\til{\mathcal{T}}$. Connecting all of the components that intersect $D_0$ together, and connecting the resulting single component for $D_0$ to each of the components intersecting the touching $\pi_1M$-translates (if they are not already the same component as the one intersecting $D_0$) results in a connected $\til{R_\text{in}}$. In addition, the manifold has only finitely many boundary components, lifts of all of which must be incident to tetrahedra in $D_0$. Therefore we also need only finitely many strips to connect out to the vertices corresponding to those boundary components.\\
 
Therefore, applying Algorithm \ref{make strips embedded} to the finite set of strips of triangles, then applying Algorithm \ref{connect two components} and the variant outlined in Remark \ref{can do condition 2 also} produces $\mathcal{T}_*$ via a finite number of pillow insertions.
\end{algor}

\begin{cor}\label{get_all_porous}
Suppose $M$ is the interior of a compact, connected, orientable 3-manifold with non-empty boundary. Then there exists a triangulation $\mathcal{T}_*$  such that all surfaces in horo-normal form relative to $\mathcal{T}_*$ are porous.
\end{cor}
\begin{proof}
It is well-known that every manifold that is the interior of a compact 3-manifold with non-empty boundary admits an ideal triangulation (for example, Proposition 1.2 of \cite{tillmann_norm_surf} is a summary of results in \cite{matveev} proving this). We can then apply Algorithm \ref{fix_all_horo-normal} to obtain $\mathcal{T}_*$.
\end{proof}

\begin{rmk}
The algorithm given in this section is likely to be very inefficient in the number of tetrahedra required. One simple way to improve it would be to check if inserting pillows along strips to improve one horo-normal surface would also improve the others. If so, we would be able to use fewer strips of triangles.
\end{rmk}

\section{Representations}\label{Representations}

\begin{defn}\label{map_to_R(M)}
We define a map 
$$\mathfrak{R}_{(\mathcal{T};S)}:\widehat{\mathfrak{D}}(M;\mathcal{T};S) \rightarrow \mathfrak{R}(M)$$
up to conjugation as follows: Given $Z \in \widehat{\mathfrak{D}}(M;\mathcal{T};S)$, the developing map $\Phi_Z$ gives us for each vertex in $\til{\mathcal{V}}$ a position on $\bdry \Hthree$. To construct a representation $\mathfrak{R}_{(\mathcal{T};S)}(Z) = \rho_Z:\pi_1M\rightarrow \PSL$, for each $\gamma \in \pi_1M$ we consider the translate $\gamma\tri$ of $\tri$ (the triangle we start developing from), which is another triangle of $\til{\mathcal{T}}$. No edges of $\tri$ are in $\til{\mathcal{E}}_0$, and the same is true for the edges of $\gamma\tri$, since both are lifts of the same triangle of $\mathcal{T}$. The positions of the vertices $v,v',v''$ of $\tri$ on $\bdry\Hthree$ are distinct by definition, and the developed positions of the vertices $\gamma v,\gamma v',\gamma v''$ of $\gamma\tri$ are distinct by Lemma \ref{E0_is_order_1}. We define $\rho_Z(\gamma)$ to be the unique element of $\PSL$ that takes the positions of $v,v',v''$ on $\bdry \Hthree$ to the positions of $\gamma v,\gamma v',\gamma v''$. We also define 
$$\mathfrak{R}_{\mathcal{T}}:\widehat{\mathfrak{D}}(M;\mathcal{T}) \rightarrow \mathfrak{R}(M)$$
as $\mathfrak{R}_{(\mathcal{T};S)}$ on each $\widehat{\mathfrak{D}}(M;\mathcal{T};S)$ in the disjoint union that makes up $\widehat{\mathfrak{D}}(M;\mathcal{T})$.
\end{defn}

To show that the map results in a representation, we need the following:
\begin{lemma}\label{get_homomorphism}
$\mathfrak{R}_{(\mathcal{T};S)}(Z) = \rho_Z:\pi_1M\rightarrow \PSL$ is a homomorphism. 
\end{lemma}
\begin{proof}
Suppose we have valid chains of triangles 
$$C_1 = \left(\tri = \tri^{(0)}_1,\tri^{(1)}_1,\ldots,\tri^{(n_1)}_1 = \gamma_1\tri \right)$$ 
$$C_2 = \left(\tri = \tri^{(0)}_2,\tri^{(1)}_2,\ldots,\tri^{(n_2)}_2 = \gamma_2\tri \right).$$
An example of a chain from $\tri$ to $\gamma_2\gamma_1\tri$ is given by 
$$C_3 = \left( \tri = \tri^{(0)}_2,\tri^{(1)}_2,\ldots,\tri^{(n_2)}_2 = \gamma_2\tri = \gamma_2\tri^{(0)}_1,\gamma_2\tri^{(1)}_1,\ldots,\gamma_2\tri^{(n_1)}_1 = \gamma_2\gamma_1\tri \right).$$
Also let $C_0 = \left(\tri\right)$ be the trivial chain.
Let $C^{(j)}$ denote the subchain of $C$ up to the $j^{\text{th}}$ triangle.
Let $\overline{\Phi}_{(C,Z)}(v)$ be the developed position of $v$ into the ends of $T_\zeta$ along a chain of triangles $C$ (so $\overline{\Phi}_{(C,Z)}(v)_H = \Phi_Z(v)$, but the higher order information can depend on $C$). Let $I_{(C,Z)} \in \text{PSL}_2\big(\C((\zeta))\big)$ be the unique element that takes the developed position of the first triangle of $C$ to the developed position of the last.\\

By definition, $$\overline{\Phi}_{(C_2,Z)}(\gamma_2\tri) = I_{(C_2,Z)}\overline{\Phi}_{(C_0,Z)}(\tri)$$
Here we are abusing notation in having $\overline{\Phi}$ take input the ordered triple of vertices forming a triangle rather than just one vertex, and $\overline{\Phi}_{(C_0,Z)}(\tri)$ is a complicated way to write the position of $\tri$ that we start developing from. \\
  
We would like to show that for each $j=0,\ldots,n_1$:
\begin{equation}\label{devel_steps}
\overline{\Phi}_{(C_3^{(n_2+j)}, Z)}(\gamma_2\tri^{(j)}_1) = I_{(C_2,Z)}\overline{\Phi}_{(C_1^{(j)},Z)}(\tri^{(j)}_1)
\end{equation}
When $j=0$ this is the previous equality, and when $j=n_1$ this says that:
$$\overline{\Phi}_{(C_3, Z)}(\gamma_2\gamma_1\tri) = I_{(C_2,Z)}\overline{\Phi}_{(C_1,Z)}(\gamma_1\tri)$$
and so 
$$\overline{\Phi}_{(C_3, Z)}(\gamma_2\gamma_1\tri) = I_{(C_2,Z)}\overline{\Phi}_{(C_1,Z)}(\gamma_1\tri) = I_{(C_2,Z)}I_{(C_1,Z)} \overline{\Phi}_{(C_0,Z)}(\tri)  $$
which is what we need. We show equation (\ref{devel_steps}) by induction. As noted before, we have the case when $j=0$. When we develop from $\gamma_2\tri^{(j)}_1$ to $\gamma_2\tri^{(j+1)}_1$, and from $\tri^{(j)}_1$ to $\tri^{(j+1)}_1$, in both cases we are developing out one further vertex position, the position determined by the positions of the triangles we are developing from and the dihedral angle. The dihedral angle is the same in both cases because the corresponding dihedral angle in $\mathcal{T}$ is the same. Since elements of $\text{PSL}_2\big(\C((\zeta))\big)$ preserve cross ratios, the positions of the two new vertices are related in the appropriate way by $I_{(C_2,Z)}$.
\end{proof}
A similar argument shows that the representation we get is independent (up to conjugation) of the triangle $\tri$ we choose to start developing from.\\

The following definition closely follows an argument of Tillmann \cite{tillmann_degenerations}.\\

\begin{defn}\label{psi_rho}
Let $T_1,\ldots,T_h$ be the torus boundary components of $M$, and so $\bdry M = \bigcup T_i$. A vertex $v \in \til{\mathcal{V}}$ is stabilised by a unique subgroup $P_v \subset \pi_1M$ which is conjugate to $\text{im}(\pi_1T_i\rightarrow \pi_1M)$ for some $i$. Let $v_1,\ldots, v_h$ be a choice of such vertices, one for each of the tori. Let $\rho:\pi_1M\rightarrow \PSL$ be any representation. The subgroup $\rho(P_{v_i}) \subset \PSL$ fixes either one or two points of $\bdry \Hthree$, or the whole of $\bdry \Hthree$ if $\rho(P_{v_i}) = \{1\}$. For each $i$, choose a $w_i \in \bdry \Hthree$ which is fixed by $\rho(P_{v_i})$. We define a map 
$$\Psi_\rho:\til{\mathcal{V}} \rightarrow \bdry \Hthree$$ 
by extending to all other vertices equivariantly. 
\end{defn}
Compare with Definition \ref{dev_map}. $\Psi_\rho$ depends on the choices of fixed point $w_i$ but the maps are otherwise well defined up to conjugation of $\Hthree$. There are at most $2^h$ choices of $\Psi_\rho$ unless $\rho(P_{v_i}) = \{1\}$ for some $v_i$, in which case there are infinitely many choices.

\begin{defn}
Let $\rho\in\mathfrak{R}(M)$ and fix a choice of $\Psi_\rho$. We associate a horo-normal surface 
$$S(\Psi_\rho)$$
to $\Psi_\rho$ as follows: 
Let $\mathcal{E}$ be the edge set of $\mathcal{T}$.  Consider the images of the endpoints of edges in $\til{\mathcal{E}}$ under $\Psi_\rho$. Let $\til{\mathcal{E}}_0$ be the set of such edges whose endpoints map to the same point of $\bdry \Hthree$. By the equivariance of $\Psi_\rho$, $\til{\mathcal{E}}_0$ descends to $\mathcal{E}_0 \subset \mathcal{E}$. Let $\mathcal{E}_+ = \mathcal{E} \setminus \mathcal{E}_0$. Let $S(\Psi_\rho)$ be the horo-normal surface associated with the subsets $\mathcal{E}_0$ and $\mathcal{E}_+$ as in Definition \ref{E_0_to_horo-normal}.
\end{defn}
\begin{lemma}
The subset $\mathcal{E}_0$ in the above definition has no bad loops.
\end{lemma}
\begin{proof}
Suppose for contradiction that this subset $\mathcal{E}_0$ has a bad loop. Then by definition of a bad loop (Definition \ref{bad_loop}) all but one edge of the loop is in $\til{\mathcal{E}}_0$. Let the loop be given by edges $e_1,e_2,\ldots,e_n \in\til{\mathcal{E}}$, where $e_i$ has endpoints $v_i,v_{i+1}\in\til{\mathcal{V}}$ for $i=1,2,\ldots,n-1$ and $e_n$ has endpoints $v_n,v_1$. Suppose that $e_1,e_2,\ldots,e_{n-1} \in \til{\mathcal{E}}_0$ and $e_{n} \in \til{\mathcal{E}}_+$. Then $\Psi_\rho(v_1) = \Psi_\rho(v_2) = \cdots = \Psi_\rho(v_{n})$, but $\Psi_\rho(v_n) \neq \Psi_\rho(v_1)$, which is our contradiction.
\end{proof}

\begin{defn}
If $A$ is an abelian group, the \textbf{generalised dihedral group} of $A$ is the semidirect product $A \rtimes \Z_2$ with $\Z_2$ acting on $A$ by inverting elements. A representation $\rho$ is \textbf{dihedral} if $\rho(\pi_1M)$ is a generalised dihedral group.
\end{defn}

\begin{lemma}\label{irred_not_dih_works}
If $\rho\in\mathfrak{R}(M)$ is irreducible then either $\rho$ is dihedral or every $\Psi_\rho$ is such that $\left|\Psi_\rho(\til{\mathcal{V}})\right| \geq 3$.
\end{lemma}

\begin{proof}
We use the same notation as in Definition \ref{psi_rho}.  Let $w_i$ be a point fixed by $\rho(P_{v_i})$, so $w_i \in \Psi_\rho(\til{\mathcal{V}})$.  Let $\text{fix}(w_i) \subset G = \rho(\pi_1M)$ be the set of isometries that have $w_i$ as a fixed point. The representation $\rho$ is irreducible, which means that no point of $\bdry \Hthree$ is fixed by all of $G$, and in particular $\text{fix}(w_i)$ is a proper subset of $G$. \\

Either there are at least three translates of $w_i$ (and we have $\left|\Psi_\rho(\til{\mathcal{V}})\right| \geq 3$) or every element of $G \setminus \text{fix}(w_i)$ is of order 2, taking $w_i$ to some $w_i' \neq w_i$ and back. In this case we must also have $\text{fix}(w_i') = \text{fix}(w_i)$, otherwise if we could move one without moving the other, we would obtain a third point. Thus the subgroup $\text{fix}(w_i)$ is abelian, since it fixes two distinct points on $\bdry \Hthree$.\\

If we arrange the fixed points at 0 and $\infty$ on $\bdry\Hthree$ then every element $a \in \text{fix}(w_i)$ is diagonal and every element $r \in G \setminus \text{fix}(w_i)$ is anti-diagonal. One can verify that $rar=a^{-1}$, and so $G = \text{fix}(w_i) \rtimes \Z_2$ is a generalised dihedral group.
\end{proof}

\begin{thm}\label{xdv_for_a_rho}
Let $M$ be the interior of a compact, connected, orientable 3-manifold with non-empty boundary consisting of a disjoint union of tori, and $\mathcal{T}$ an ideal triangulation of $M$. Let $\rho\in\mathfrak{R}(M)$ such that there is a choice of $\Psi_\rho$ with $\left|\Psi_\rho(\til{\mathcal{V}})\right| \geq 3$ and $S(\Psi_\rho)$ porous. Then there exists
 $Z_\rho \in \widehat{\mathfrak{D}}(M;\mathcal{T};S(\Psi_\rho))$ such that $\mathfrak{R}_{(\mathcal{T};S(\Psi_\rho))}(Z_\rho) = \rho$ up to conjugation.
\end{thm}
\begin{proof}
  
By assumption $S(\Psi_\rho)$ is porous. Now suppose for contradiction that condition 3 of Definition \ref{xdv} fails. Then all tetrahedra are of types 22, 31 or 4. They cannot all be of type 4 since then $\left|\Psi_\rho(\til{\mathcal{V}})\right|=1$. $\til{R_{\text{in}}}$ is connected by condition 1 of porousity and intersects edges of every vertex in $\til{\mathcal{V}}$ by condition 2, so we can follow chains of triangles that contiguously intersect $\til{R_{\text{in}}}$ starting from some triangle $\tri$ of type 21, and going to each vertex. Every triangle we move through is of type 21, and we see that the vertices of $\til{\mathcal{V}}$ fall into two sets, those that are connected by paths of edges in $\til{\mathcal{E}}_0$ to either the pair of vertices of $\tri$ connected by an edge of $\til{\mathcal{E}}_0$, or the other vertex of $\tri$. Vertices from these two sets are never connected by an edge of $\til{\mathcal{E}}_0$, since that would give a bad loop. Thus in this case $\left|\Psi_\rho(\til{\mathcal{V}})\right|=2$.\\

So $S(\Psi_\rho)$ is a horo-normal surface that satisfies all of the conditions of Definition \ref{xdv}. We now need to construct $Z_\rho \in \widehat{\mathfrak{D}}(M;\mathcal{T};S(\Psi_\rho))$ such that $\mathfrak{R}_{(\mathcal{T};S(\Psi_\rho))}(Z_\rho) = \rho$ up to conjugation. Because of the definition of $\mathfrak{R}_{(\mathcal{T};S(\Psi_\rho))}$ (Definition \ref{map_to_R(M)}) and the fact that elements of $\PSL$ are determined by their action on 3 distinct points of $\bdry\Hthree$, it is enough to construct $Z_\rho$ so that $\Psi_\rho = \Phi_{Z_\rho}$ as maps $\til{\mathcal{V}}\rightarrow\bdry\Hthree$.\\

Fix a conjugation of $\Hthree$ so that $\infty \notin \Psi_\rho(\til{\mathcal{V}})$. $\Psi_\rho$ determines the position on $\bdry \Hthree$ of every vertex of $\til{\mathcal{V}}$, and each edge of $\til{\mathcal{E}_0}$ has both endpoints in the same position. To determine the data (see Definition \ref{xdv_data}) for a tetrahedra of type 1111 in $\mathcal{T}$ we simply read off the cross ratio given by the positions of the vertices of one of its lifts in $\til{\mathcal{T}}$. The answer we get is independent of the choice of lift since elements of $\PSL$ preserve cross ratios and lifts are taken to each other by deck transformations $\gamma\in\pi_1M$, and their images in $\bdry \Hthree$ taken to each other by $\rho(\gamma) \in \PSL$. We will use a similar construction to deal with the degenerate tetrahedra:\\

For each $e_i\in \mathcal{E}_0$, arbitrarily choose a lift $\til{e}_i\in \til{\mathcal{E}_0}$. We also arbitrarily choose an {\bf offset} $\delta_i \in \C\setminus\{0\}$ for $\til{e}_i$, viewing it as a directed edge between its endpoints $u_i,v_i\in \til{\mathcal{V}}$. The idea is that although $\Psi_\rho(u_i) = \Psi_\rho(v_i)$, we want to introduce some extra information to talk about the difference between the two positions, namely that the difference will be $\delta_i\zeta  \in \C[[\zeta]]$. We extend the choice of offset to all other lifts of edges in $\mathcal{E}_0$ using $\rho$: if $\til{e}_i$ and $\til{e}'_i$ are lifts of $e_i \in \mathcal{E}_0$ with endpoints $u_i, v_i, u_i',v_i'$ then there is some deck transformation $\gamma\in\pi_1M$ such that $\gamma \til{e}_i=\til{e}'_i$.  If $\Psi_\rho(u_i) = \Psi_\rho(v_i) = x \in \C$ then $\rho(\gamma)(x) = \Psi_\rho(u_i') = \Psi_\rho(v_i')$.  If 
$$\rho(\gamma) = \left(\begin{array}{cc}
a&b\\
c&d\end{array}\right)$$
with determinant 1 then
$$\left(\begin{array}{cc}
a&b\\
c&d\end{array}\right)
\left(\begin{array}{c}
x + \delta_i\zeta\\
1\end{array}\right) = 
\left(\begin{array}{c}
ax + a\delta_i\zeta + b\\
cx + c\delta_i\zeta + d\end{array}\right)$$
and
$$\frac{ax + a\delta_i\zeta + b}{cx + c\delta_i\zeta + d} = \frac{ax+b}{cx+d} + \frac{\delta_i\zeta }{(cx+d)^2} + (\text{h.o.t. in }\zeta)$$
$\frac{ax+b}{cx+d} = \Psi_\rho(u_i') = \Psi_\rho(v_i')$, and we take the offset for $\til{e}'_i$ to be $\frac{\delta_i }{(cx+d)^2} \in \C \setminus \{0\}$. One can verify that this choice is consistent in that we get the same answer under action by products of elements of $\PSL$. Note also that the only element of $\pi_1M$ that fixes an edge of $\til{\mathcal{E}}$ is the identity element, which of course fixes the offset.\\

We can also see the consistency as follows: consider four points $x+\delta_i\zeta, y,x,w \in \C[[\zeta]]$, where $x,y,w\in\C$ are distinct. The cross ratio $z$ of these four points is preserved under elements of $\PSL \subset \text{PSL}_2\big(\C((\zeta))\big)$, and by Lemma \ref{*det_cr}, 
$$z_* = \frac{\delta_i(y-w)}{(x-w)(y-x)}$$
So $\delta_i$ is determined by $z_*$ and $x,y,w$. If we then apply some combination of elements of $\PSL$ then $z_*$ stays fixed and our new $\delta_i$ is determined by the new positions for $x,y,w$, which are independent of the combination of elements of $\PSL$ that take us here. \\

So we have a $\delta_i$ assigned to each $\til{e}_i\in\til{\mathcal{E}}_0$. We use these and Lemma \ref{*det_cr} to read off the lowest order information of the preferred cross ratios for tetrahedra of types 211, and 22. See Figure \ref{read_off_cross_ratios.pdf}. Again the answer we get is independent of the choice of lift of tetrahedron. For tetrahedra of type 31 we only care about the dihedral angle between pairs of triangles that meet at an edge in $\til{\mathcal{E}}_+$. We can read this cross ratio off as 
$$z = \frac{(x-(x+\delta_j\zeta))(w-(x+\delta_i\zeta))}{(x - (x+\delta_i\zeta))(w-(x+\delta_j\zeta))}=\frac{\delta_j(w-(x+\delta_i\zeta))}{\delta_i(w-(x+\delta_j\zeta))}$$
so $z_* = \delta_j/\delta_i$.\\

\begin{figure}[htb]
\centering
\includegraphics[width=1.0\textwidth]{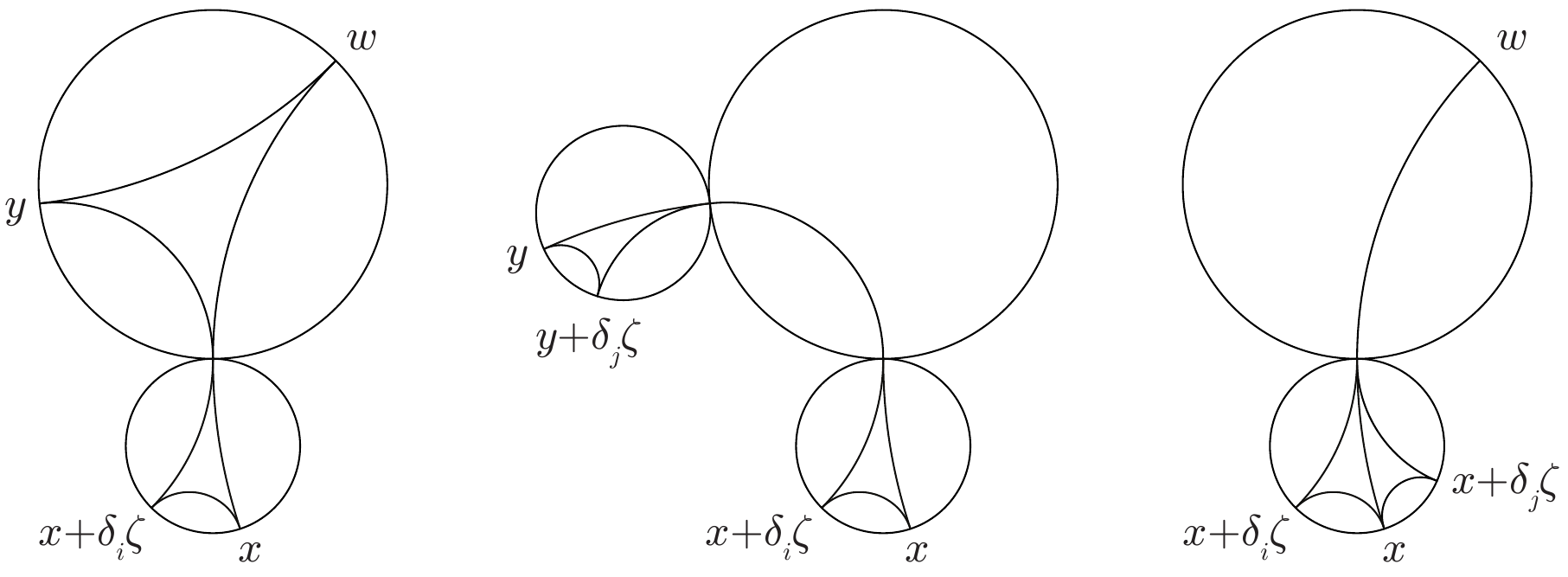}
\caption{Positions of vertices of 211, 22, 31 tetrahedra. The larger circle represents $\bdry \Hthree \cong \C[\zeta]/(\zeta)\cup\{\infty\}$. The smaller circles represent the set of points of the form $x + \zeta x'$ where $x'\in\C[[\zeta]]$. Offsets $\delta_i$ and $\delta_j$ may or may not be associated to edges that are $\pi_1M$-translates of each other.}
\label{read_off_cross_ratios.pdf}
\end{figure}

Notice that we do not require that the offset of the third edge of the 31 tetrahedron together with the first two link up to form a triangle. We track the first order offset (difference) between two points with the same position on $\bdry\Hthree$, but not the absolute first order positions.\\

We now have the data for a point of the extended deformation variety, we need to show that these choices satisfy the consistent development condition. Suppose that we have two triangles $\tri_1$ and $\tri_2$ which share an edge in $\til{\mathcal{E}}_+$, the positions on $\bdry\Hthree$ of the vertices of $\tri_1$ as given by $\Psi_\rho$ and the offset for any edge of $\tri_1$ in $\til{\mathcal{E}}_0$, together with the cross ratio data for the dihedral angle. Then Lemma \ref{*det} tells us that we can recover the position on $\bdry\Hthree$ of the vertex of $\tri_2$ not shared with $\tri_1$, and any offsets for edges of $\tri_2$ in $\til{\mathcal{E}}_0$. As we develop through valid chains we always get the correct answer (agreeing with $\Psi_\rho(\til{\mathcal{V}})$, and with our offsets) no matter which chain of triangle we develop along, so we get consistent development. So if we start developing from a triangle with vertex positions agreeing with $\Psi_\rho$, we get $\Psi_\rho = \Phi_{Z_\rho}$ as maps $\til{\mathcal{V}}\rightarrow\bdry\Hthree$.
\end{proof}

\begin{proof}[Proof of Theorem \ref{xdv_for_all_rho_if_all_horo_omni}]
This follows immediately from Theorem \ref{xdv_for_a_rho} and Lemma \ref{irred_not_dih_works}. \end{proof}

\section{Examples, part 2: The once punctured torus bundle with monodromy $LLR$, revisited}\label{Example, part 2}
We return to the example of Section \ref{LLR_ex_part_1}. The component satisfying $ij=1$ in $\mathfrak{D}(M_{LLR};\mathcal{T}_4)$ has $hk=1$, so the top right diagram of Figure \ref{LLR_4_vs_5_tetra.pdf} has all of the ``equatorial'' dihedral complex angles being 1. (The back dihedral angle is obviously 1, the other two dihedral angles turn out to be 1 via the equations internal to each tetrahedron, (\ref{dv_eq1}) and (\ref{dv_eq2}).) Then the north and south vertices of the two tetrahedra are in the same place on $\bdry \Hthree$, and so the added edge $e$ in $\mathcal{T}_5$ is the single edge in $\mathcal{E}_0$. The corresponding horo-normal surface $S$ is shown in Figures \ref{LLR_5_tetra_3_in_E0_perspective.pdf} and \ref{LLR_5_tetra_3_in_E0_above.pdf}.\\

\begin{figure}[htb]
\centering
\includegraphics[width=0.6\textwidth]{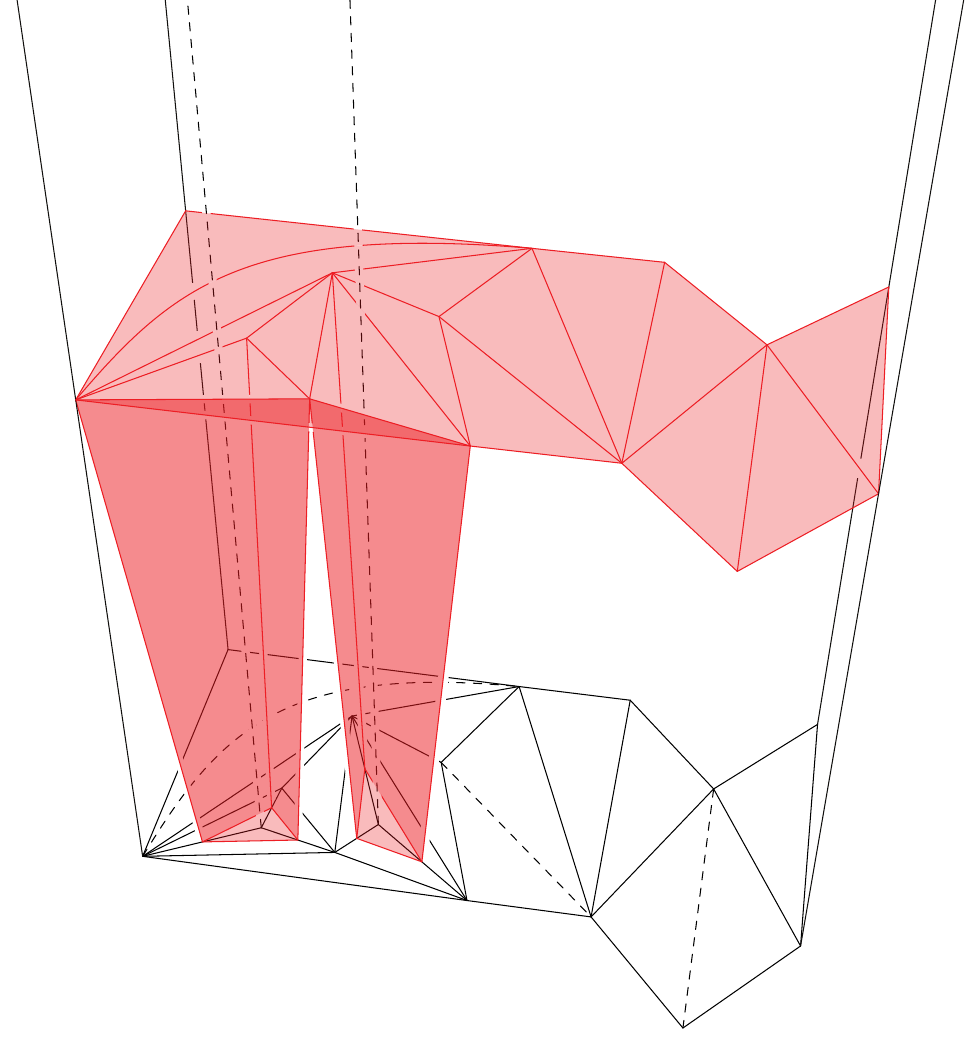}
\caption{Tetrahedra incident to one fundamental domain of the boundary torus, as seen in perspective. Only some of the vertical edges are shown (each vertex below has an edge above it going towards the vertex at infinity). The dashed lines are all translates of the one edge $e \in \mathcal{E}_0$. Shown are some of the pieces of the horo-normal surface:  tubes made from quadrilaterals around the vertical lifts of $e$ and triangles nearest the vertex at infinity.}
\label{LLR_5_tetra_3_in_E0_perspective.pdf}
\end{figure}

\begin{figure}[htb]
\centering
\includegraphics[width=1.0\textwidth]{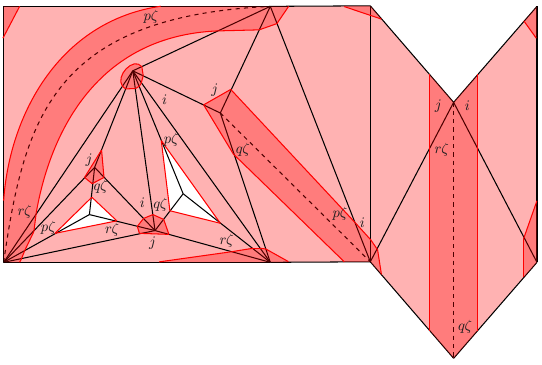}
\caption{The view from above. Shown here are all parts of the horo-normal surface. Each edge of $\mathcal{E}_+$ intersects the surface twice, the edge $e\in \mathcal{E}_0$ (dashed) is disjoint from the surface. Tetrahedra labelled with dihedral angles $i$ and $j$ are of type 1111, $p\zeta,q\zeta $ and $r\zeta $ are type 211.}
\label{LLR_5_tetra_3_in_E0_above.pdf}
\end{figure}

As in Figure \ref{LLR_5_tetra_3_in_E0_above.pdf}, all tetrahedra are of types 1111 or 211, and so $\til{R_\text{in}}$ is connected. (If there are no tetrahedra of type 31 or 4 then all triangles are of type 111 or 21, and so $\til{R_\text{in}}$ connects through the center of each triangular face of $\til{\mathcal{T}}$.) In this example, if we perform one compression move to the surface $S$ we obtain a boundary parallel torus, and consistent development for $\widehat{\mathfrak{D}}(M_{LLR};\mathcal{T}_5;S)$ is achieved if we have the gluing equations (or rather, lowest order versions, using angles as in Figure \ref{4_tetra_types_w_dihedral_angles.pdf}) for each edge apart from $e \in \mathcal{E}_0$, together with consistency for a chain of triangles going around $e$. The gluing equation (\ref{ijpqr_first}) is gone, the gluing equations (\ref{ijpqr_second}) through (\ref{ijpqr_last}) become:

\begin{eqnarray}
\label{ijpqr_new_first} ijq\zeta(-q^{-1}\zeta^{-1})r\zeta(-r^{-1}\zeta^{-1}) &=& 1 \\
\frac{i-1}{i}j(-p^{-1}\zeta^{-1})q\zeta&=& 1 \\
i \frac{1}{1-i} \frac{j-1}{j} \frac{1}{1-j}p\zeta(-q^{-1}\zeta^{-1})1&=& 1 \\
\label{ijpqr_new_last} \frac{i-1}{i} \frac{1}{1-i}\frac{j-1}{j}\frac{1}{1-j} 1 p\zeta(-p^{-1}\zeta^{-1})1r\zeta(-r^{-1}\zeta^{-1})&=&1
\end{eqnarray}
which simplify (caring only about lowest order) to:
\begin{eqnarray}
 ij &=& 1 \\
\frac{i-1}{i}j\left(-\frac{q}{p}\right)&=& 1 \\
\frac{i}{i-1} \frac{1}{j} \left(-\frac{p}{q}\right)&=& 1 \\
\frac{1}{ij} &=&1
\end{eqnarray}
There are obvious redundancies here. To see the consistency for a chain of triangles going around $e$, consider developing around the left hand vertical tube in Figure \ref{LLR_5_tetra_3_in_E0_perspective.pdf}. There are three triangles, all of which share the vertex at infinity. The dihedral angles between those triangles are $p\zeta(-r^{-1}\zeta^{-1})=-\frac{p}{r}, q\zeta(-p^{-1}\zeta^{-1})=-\frac{q}{p}$, and $r\zeta(-q^{-1}\zeta^{-1})=-\frac{r}{q}$. The three triangles are three faces of a tetrahedron (in fact the tetrahedron labelled $h$ in $\mathcal{T}_4$, from before the 2-3 move, see Figure \ref{LLR_4_vs_5_tetra.pdf}), and consistent development around these triangles is the same as equations (\ref{dv_eq1}) and (\ref{dv_eq2}) for that tetrahedron. The first is satisfied automatically, and the second simplifies to the last equation we need for $\widehat{\mathfrak{D}}(M_{LLR};\mathcal{T}_5;S)$:
\begin{equation}
p+q+r=0
\end{equation}

\begin{rmk}
We have three independent equations in five variables, and so this variety is 2-dimensional, whereas the corresponding component of the deformation variety with triangulation $\mathcal{T}_4$ is 1-dimensional. The extra dimension comes from the choices of $p,q$ and $r$, all of which can be scaled by some constant at once to give another point of $\widehat{\mathfrak{D}}(M_{LLR};\mathcal{T}_5;S)$, and the scaled and original points map to the same representation in $\mathfrak{R}(M)$.
\end{rmk}

\section{Application: the $\PSL$ A-polynomial}\label{A-polynomial}

\subsection{Definitions}
The A-polynomial was introduced in \cite{CCGLS}, and originally defined for the $\text{SL}_2(\mathbb{C})$ character variety. For the $\PSL$ version, we follow \cite{champanerkar_thesis}. Assume that $N$ is a 3-manifold with $\bdry N$ being a single torus boundary component and choose generators $L,M\in \pi_1\bdry N$. Let $X(N), X(\bdry N)$ be the $\PSL$ character varieties of $N$ and $\bdry N$ respectively, and $r:X(N)\to X(\bdry N)$ the restriction. Let $\Delta \subset \mathfrak{R}(\bdry N)$ be the subvariety consisting of diagonal representations. Let $p_B:\Delta\to\mathbb{C}^*\times\mathbb{C}^*$ be an isomorphism given as follows: If $\rho\in\Delta$ is such that $\rho(L)=\pm\left(\begin{array}{cc}l&0\\0&l^{-1}\end{array}\right), \rho(M)=\pm\left(\begin{array}{cc}m&0\\0&m^{-1}\end{array}\right)$, then $p_B(\rho)=(l^2,m^2)$. Let $t:\mathfrak{R}(N)\to X(N)$ be the quotient map and $t_\Delta:\Delta\to X(N)$ the restriction to $\Delta$, which is a surjection and generically 2-to-1. Let $X'(N)$ be the union of irreducible components $Y'$ of $X(N)$ such that the closure of $r(Y')$ is 1-dimensional. For each component $W'$ of $X'(N)$ let $W$ be the curve $t_\Delta^{-1}(\overline{r(Y')})\subset \Delta$. Let $D_N$ be the union of curves $W$ as $W'$ varies over all components of $X'(N)$. 
\begin{defn}
The defining polynomial of the closure of the image of $D_N$ in $\mathbb{C}^*\times\mathbb{C}^*$ is called the $\PSL$ A-polynomial of $N$ and denoted by $A_N(l,m)$.
\end{defn}

\subsection{Calculating eigenvalues}

\begin{lemma}\label{can_define_Hol}
There is a well-defined rational map $\text{Hol}:\widehat{\mathfrak{D}}(N;\mathcal{T};S)\to \mathbb{C}^*\times\mathbb{C}^*$ such that the following diagram commutes:
\[
\begin{CD}
X(N)       @<t\circ\mathfrak{R}_{(\mathcal{T};S)}<< \widehat{\mathfrak{D}}(N;\mathcal{T};S) \\
@VrVV            @VV \text{Hol} V\\
X(\bdry N) @<p_B^{-1}\circ t_\Delta<< \mathbb{C}^*\times\mathbb{C}^*
\end{CD}
\]
\end{lemma}
\begin{proof}
Let $Z\in \widehat{\mathfrak{D}}(N;\mathcal{T};S)$, and suppose we have a choice of initial triangle $\tri\in\til{\mathcal{T}}$ with vertices $v,v',v''\in\til{\mathcal{V}}$ where all edges of $\tri$ are in $\til{\mathcal{E}}_+$, and a choice of ideal triangle with vertices $e_0,e_0',e_0''\in\bdry\Hthree$ for the image of $\tri$. Then we have the developing map $\Phi_Z:\til{V}\rightarrow \bdry\Hthree$ as in Definition \ref{dev_map}. For now we assume that $\infty \notin \Phi_Z(\til{V})$. As in the proof of Theorem \ref{is_variety}, the developed positions of cusps are rational functions of the $z^{(i)}_*$ (the lowest non-zero order terms of the cross ratios) and $e_0,e_0',e_0''$.\\

Let $L',M'\in\pi_1(N)$ be images of $L,M$ under the injection $\pi_1(\bdry N)\hookrightarrow\pi_1(N)$ chosen so that the deck transformations of $\til{N}$ corresponding to $L',M'$ fix $v$. In order to calculate $\Phi_Z(L'(v))$, $\Phi_Z(L'(v'))$ and $\Phi_Z(L'(v''))$ we choose a valid chain of triangles starting with $\tri$ and ending at $L'(\tri)$. By the consistent development condition, the values of the developed positions are independent of the particular choice of valid chain. However, as $L'$ is peripheral, we can choose such a chain which follows along the lift of the horo-normal surface which bounds the component $\left(\til{R_{\text{out}}}\right)_0$ of $\til{R_{\text{out}}}$ that contains $v$, and then by Lemma \ref{verts_thru_S_coincide} we get that $\Phi_Z(L'(v))=\Phi_Z(v)$.\\

Now conjugate the whole picture to move $(e_0,e_0',e_0'')$ to $(0,1,\infty)$. We still call the developing map $\Phi_Z$, so now $\Phi_Z((v,v',v''))=(\infty,0,1)$. All developed positions other than those that are now at $\infty$ are still rational functions, now of only the $z^{(i)}_*$. \\

Then $\Phi_Z(L'(\tri))=(\infty,b,b+a)$ for some developed cusp positions $b, b+a$, and so $b$ and $a$ are rational functions of the $z^{(i)}_*$. The corresponding element of $\PSL$ is $\rho_Z(L')=\pm\left(\begin{array}{cc}\sqrt{a}&b/\sqrt{a}\\0&1/\sqrt{a}\end{array}\right)$, and the square of the eigenvalue is $a$, which is therefore a rational function of the $z^{(i)}_*$. Similarly for $M$, and we have constructed a rational map $\text{Hol}:\widehat{\mathfrak{D}}(N;\mathcal{T};S)\to \mathbb{C}^*\times\mathbb{C}^*$ in such a way that the above diagram commutes.
\end{proof}

See equations \ref{calc_merid} and \ref{calc_long}, and figure \ref{knot_8_20_torus_bdry.pdf} for an example of how to compute Hol in practice. One travels along the path, picking up a factor (resp. its inverse) when the path rotates anti-clockwise (resp. clockwise) around the corner of a triangle. The factors are as shown in Figure \ref{4_tetra_types_w_dihedral_angles.pdf}. Note that the $\zeta$ terms will always cancel with each other.\\

Now let $X_S = \cup Y_i$ where $Y_i$ is a component of $\widehat{\mathfrak{D}}(N;\mathcal{T};S)$ whose closure of the image under Hol is a curve in $\mathbb{C}\times\mathbb{C}$. 
\begin{defn}
If $X_S\neq \emptyset$, the image $\text{Hol}(X_S)$ is called the {\bf holonomy variety} with respect to the triangulation $\mathcal{T}$ and horo-normal surface $S$, and is denoted by $H^{(\mathcal{T};S)}(N)$. The defining polynomial of the closure of $H^{(\mathcal{T};S)}(N)$ is denoted by $H^{(\mathcal{T};S)}(l,m)$. 
\end{defn} 

\begin{thm}\label{thm_get_A-poly}
Let $\mathcal{T}_*$ be an ideal triangulation of $N$ for which every surface in horo-normal form relative to $\mathcal{T}_*$ is porous. Then the polynomials $H^{(\mathcal{T}_*;S)}(l,m)$, ranging over each horo-normal surface $S$, contain all factors of the $\PSL$ A-polynomial of $N$ associated to components of irreducible non-dihedral representations.
\end{thm}

\begin{proof}
By Lemma \ref{can_define_Hol}, for each choice of $S$ for which $\widehat{\mathfrak{D}}(N;\mathcal{T};S) \neq \emptyset$, $H^{(\mathcal{T};S)}(l,m)$ divides the $\PSL$ A-polynomial. Moreover, by Lemma \ref{irred_not_dih_works} and Theorem \ref{xdv_for_a_rho},  $\mathfrak{R}_{\mathcal{T_*}}:\widehat{\mathfrak{D}}(M;\mathcal{T_*}) \rightarrow \mathfrak{R}(M)$ is onto the irreducible non-dihedral representations.\end{proof}

\begin{proof}[Proof of Theorem \ref{thm_get_A-poly2}]
Combine Theorem \ref{thm_get_A-poly} with Corollary \ref{get_all_porous}.
\end{proof}

\subsection{Examples, part 3: The knot $8_{20}$}\label{8 20}

Marc Culler has compiled a list of A-polynomials of knot complements, using the standard deformation variety associated to triangulations of knots calculated by Joe Christy and included in SnapPy~\cite{snappy}. The list is available at \url{http://www.indiana.edu/~knotinfo/references/a_polys_table_glueing.html}. See also \url{http://www.math.uic.edu/~culler/talks/apolynomials.pdf} for details on how the calculations are performed. As the standard deformation variety is used in these calculations, Culler does not claim that all factors of the A-polynomial are listed. Indeed, Thomas Mattman~\cite{mattman02} shows that the knot $8_{20}$ must have two factors, one of which is missing in Culler's calculations, which list the following expression for the A-polynomial of $8_{20}$:

\begin{eqnarray*}
m^{10} &+& l(1 - m^2 + 2m^4 - 2m^6 - m^8 + 5m^{10} + m^{12}) \\
 &+& l^2(-1 + 5m^2 - 3m^6 + 3m^8 + 4m^{12} + 2m^{14}) \\
 &+& l^3(2m^4 + 4m^6 + 3m^{10} - 3m^{12} + 5m^{16} - m^{18}) \\
 &+& l^4(m^6 + 5m^8 - m^{10} - 2m^{12} + 2m^{14} - m^{16} + m^{18}) \\
 &+& l^5m^8
 \end{eqnarray*}

In Figure \ref{knot_8_20_torus_bdry.pdf} (top) we see the triangulation of the boundary torus induced by triangulation for the complement of the knot $8_{20}$ as given by SnapPy. We have chosen a labelling of the   angles to match with the preferred cross-ratios for the particular choice of horo-normal surface shown below.\\

One can check that the standard deformation variety for this triangulation has only one component by repeatedly solving an equation for one of the variables and substituting in until there is one polynomial in two variables, then checking that this polynomial does not factor. In cases with a small number of tetrahedra this is often possible to do. However, there is another component in the representation variety which is mapped to by the extended deformation variety with the horo-normal surface as shown in the figure.\\

\begin{figure}[htb]
\centering
\includegraphics[width=\textwidth]{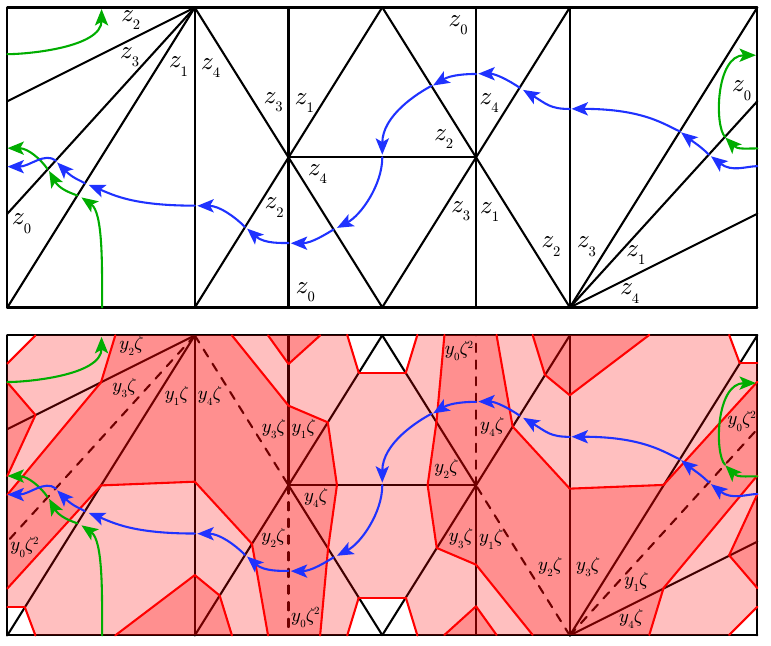}
\caption{Above: tetrahedra incident to one fundamental domain of the boundary torus of the knot $8_{20}$ as viewed from the cusp, with the triangulation as given by SnapPy~\cite{snappy}. The edges of the fundamental domain are identified in the obvious way, and the generators of the holonomy as given by SnapPy are also shown. Below: all parts of the horo-normal surface. Each edge of $\mathcal{E}_+$ intersects the surface twice, the edge in $\mathcal{E}_0$ (dashed) is disjoint from the surface.}
\label{knot_8_20_torus_bdry.pdf}
\end{figure}

As in the previous example, only one edge is in $\mathcal{E}_0$. This time four tetrahedra are of type 112 and one is of type 22. The equations for consistent development, and the holonomies of the meridian and longitude are:

\begin{eqnarray}
(y_0\zeta^2)(-y_3^{-1}\zeta^{-1})(-y_4^{-1}\zeta^{-1})&=&1\\
(-y_0^{-1}\zeta^{-2})(y_4\zeta)(-y_2^{-1}\zeta^{-1})(y_1\zeta)(y_3\zeta)(-y_4^{-1}\zeta^{-1})(y_2\zeta)&=&1\\
(-y_0^{-1}\zeta^{-2})(y_1\zeta)(y_4\zeta)(-y_3^{-1}\zeta^{-1})(-y_1^{-1}\zeta^{-1})(y_2\zeta)(y_3\zeta)&=&1\\
(y_0\zeta^2)(-y_2^{-1}\zeta^{-1})(-y_1^{-1}\zeta^{-1})&=&1\\
y_1 + y_2 + y_3 + y_4 &=&0\\
\label{calc_merid}(-y_0^{-1}\zeta^{-2})^{-1}(-y_3^{-1}\zeta^{-1})(-y_1^{-1}\zeta^{-1})^{-1}(y_0\zeta^2)^{-1} &=&m\\
\nonumber(-y_1^{-1}\zeta^{-1})^{-1}(-y_0^{-1}\zeta^{-2})(y_3\zeta)(y_4\zeta)(-y_0^{-1}\zeta^{-2})(y_2\zeta)\label{calc_long}(y_4\zeta)^{-1}(-y_0^{-1}\zeta^{-2})^{-1}&&\\
(y_2\zeta)^{-1}(y_1\zeta)^{-1}(-y_0^{-1}\zeta^{-2})^{-1}(-y_3^{-1}\zeta^{-1})&=&l
\end{eqnarray}

The first four consistency equations come from four of the usual gluing equations. The last equation arises from developing around the zero-length edge, similarly to as in the previous example. The equations simplify to:

\begin{eqnarray}
y_0&=&y_3y_4\\
y_1y_3&=&-y_0\\
y_2y_4&=&-y_0\\
y_0&=&y_1y_2\\
y_1 + y_2 + y_3 + y_4 &=&0\\
-y_1y_3^{-1}&=&m\\
1&=&l
\end{eqnarray}

We can solve these as $y_4=-y_1, y_3=-y_2, y_0 = y_1y_2$, where $y_1,y_2 \in \mathbb{C}\setminus\{0\}$. Once again, the variety is 2-dimensional, and one of the degrees of freedom comes from the choices of the $y_i$: $(y_0,y_1,y_2,y_3,y_4)$ and $(\lambda^2y_0,\lambda y_1,\lambda y_2,\lambda y_3,\lambda y_4)$ give the same representations, for any $\lambda\in \mathbb{C}\setminus\{0\}$. As for the holonomies, the longitude is constant whereas the meridian has no restriction. Therefore the factor of the $\PSL$ polynomial corresponding to this component is $H^{(\mathcal{T};S)}(l,m)=l-1$, which is clearly different from the factor calculated by Culler. 

\begin{rmk}
The results of this paper suggest that one can try to find extra factors of the A-polynomial with the standard deformation variety by retriangulating to remove edges that are of zero-length for the relevant component of the character variety.  On reading a draft of this paper, Culler did this experiment~\cite{culler_personal_communication}: Using randomised retriangulation, he found a triangulation of the complement of the knot $8_{20}$ which does not contain the zero-length edge. With this triangulation, his calculations do pick up the extra factor $(l-1)$.\\

Note that in general there is no guarantee that retriangulating to remove a bad edge will result in a  triangulation for which the deformation variety picks up a missing component. It is also possible that the edge could be removed but some other added edge be of zero length for the component, and then again the deformation variety would miss it.
\end{rmk}

\begin{rmk}
Stavros Garoufalidis and Thomas Mattman have recently calculated the A-polynomial for all non-abelian factors for all $(-2,3,n)$ pretzel knots using a recursion relation~\cite{garoufalidis_mattman}. The knot $8_{20}$ is the $(-2,3,-3)$ pretzel knot. See also \cite{garoufalidis_koutschan} where the A-polynomial for the $(-2,3,-3)$ pretzel knot is given explicitly. Their calculation also detects the $(l-1)$ factor coming from the component of irreducible representations that we find. There is also a component of abelian representations which give another $(l-1)$ factor.  Note that they work with the mirror image of the version of the manifold that Culler uses, so their polynomial differs by the map $l \mapsto 1/l$ as well as the added factor $(l-1).$ 
\end{rmk}

\section{Further questions}\label{Further questions}
\begin{enumerate}

\item How should we compactify the extended deformation variety, similarly to Tillmann's compactification of the standard deformation variety in \cite{tillmann_degenerations}?

\item If we can solve the previous question, how much of the Culler-Shalen machinery can we reproduce in the context of triangulations? The set of ideal points of the extended deformation variety for a 3-manifold with a triangulation with all horo-normal surfaces porous should contain ideal points corresponding to each ideal point of the character variety (for components not made up of reducible or dihedral representations).

\item Are there manifolds for which the standard deformation variety for \emph{every} triangulation ``misses'' some component seen by the extended deformation variety? 

\end{enumerate}

\bibliographystyle{../../hamsplain}
\bibliography{../../henrybib}
\end{document}